\documentclass[10pt,reqno]{amsart}

\usepackage{enumerate}
\usepackage{amsmath, amssymb, amsthm,amsfonts}
\usepackage{mathrsfs}
\usepackage{esint}
\usepackage{xcolor}
\usepackage{mathtools}
\usepackage{hyperref}
\usepackage{bm}
\usepackage{accents}
\usepackage[foot]{amsaddr}
\usepackage{todonotes}

\makeatletter

\numberwithin{equation}{section}

\newtheorem{thm}{Theorem}[section]
\newtheorem{theorem}[thm]{Theorem}
\newtheorem{lemma}[thm]{Lemma}

\newtheorem{proposition}[thm]{Proposition}
\theoremstyle{definition}

\newtheorem{definition}[thm]{Definition}

\newtheorem{assumption}[thm]{Assumption}
\theoremstyle{remark}

\newtheorem{remark}[thm]{\bf{Remark}}

\newcommand\aint{-\hspace{-0.38cm}\int}

\newcommand\bH{\mathbb{H}}

\newcommand\bL{\mathbb{L}}
\newcommand\bM{\mathbb{M}}
\newcommand\bN{\mathbb{N}}

\newcommand\bR{\mathbb{R}}

\newcommand\fL{\mathbf{L}}

\newcommand\fH{\mathbf{H}}

\newcommand\cA{\mathcal{A}}
\newcommand\cB{\mathcal{B}}
\newcommand\cC{\mathcal{C}}

\newcommand\cL{\mathcal{L}}

\newcommand\sM{\mathscr{M}}

\begin{document}

\title[]{Weighted mixed-norm estimates for fractional parabolic equations with space-time nonlocal operators}

\author{Hongjie Dong$^{1}$}
\address{$^1$ Division of Applied Mathematics, Brown University, 182 George Street, Providence, RI 02912, USA}
\email{Hongjie\_Dong@brown.edu}
\thanks{H. Dong was partially supported by the NSF under agreement DMS-2350129.}

\author{Junhee Ryu$^{2}$}
\address{$^2$ School of Mathematics, Korea Institute for Advanced Study, 85 Hoegi-ro, Dongdaemun-gu, Seoul, 02455, Republic of Korea}
\email{junhryu@kias.re.kr}
\thanks{J. Ryu was supported by a KIAS Individual Grant (MG101501) at Korea Institute for Advanced Study and by the National Research Foundation of Korea (NRF) grant funded by the Korea government (MSIT) (No. RS-2026-25477288).}



\begin{abstract}
We establish weighted mixed-norm estimates for fractional parabolic equations
\begin{equation*} 
\partial_t^\alpha u=Lu-\lambda u+f \text{ in } (0,T)\times\mathbb{R}^d,
\end{equation*}
with nonlocal operators in both time and space. Here, $\partial_t^\alpha$ is the Caputo derivative of order $\alpha\in(0,1)$, and $L$ is a spatially nonlocal operator of order $\sigma\in(0,2)$ whose kernel is merely measurable in time. We also obtain the corresponding estimates in the odd mixed-norm spaces, where the inner integration is taken in time and the outer one in space. The estimates are robust in the limit $\alpha\to1$ and $\sigma\to2$. We also establish unique solvability when either $T<\infty$ or $\lambda>0$. The proof is based on a direction-by-direction extension argument.
\end{abstract}

\maketitle

\section{Introduction}

In this paper, we study
\begin{equation} \label{eq_intro}
    \partial_t^\alpha u=Lu-\lambda u+f \quad \text{ in } \bR_T^d
\end{equation}
with the zero initial condition, where $T\in(0,\infty]$, $\bR_T^d:=(0,T)\times\bR^d$, and $\lambda\geq0$. Here, $\partial_t^\alpha$ is the Caputo fractional derivative of order $\alpha\in(0,1)$, and $L$ is a nonlocal operator of order $\sigma\in(0,2)$ given by
\begin{equation} \label{eq_op}
    Lu(t,x):=\int_{\bR^d}\left(u(t,x+y)-u(t,x)-y^{(\sigma)}\cdot\nabla_xu(t,x)\right)K(t,y)dy,
\end{equation}
where
\begin{equation*}
    y^{(\sigma)}:=y1_{\sigma\in(1,2)}+y1_{\sigma=1}1_{|y|\leq1}.
\end{equation*}
The most well-known example of $L$ is the fractional Laplacian $-(-\Delta)^{\sigma/2}$, corresponding to the case $K(t,y)=c|y|^{-d-\sigma}$ for some $c>0$. We impose a natural upper bound and a nondegeneracy condition on $K$, together with the usual cancellation condition when $\sigma=1$. We refer the reader to Section \ref{sec_main} for the precise assumptions and definitions.

Fractional parabolic equations naturally arise in models of anomalous diffusion, where temporal memory effects and spatial jumps are represented by the Caputo derivative and the spatial nonlocal operator, respectively.
Such space-time fractional equations are also related to scaling limits of continuous-time random walks (see \cite{GM03}).

There is an extensive literature on Sobolev regularity theory for equations that are nonlocal in either time or space. 
Time-fractional equations with local spatial operators of the form
\begin{equation*}
    \partial_t^\alpha u(t,x) = a_{ij}(t,x)D_{ij}u(t,x)+f(t,x)
\end{equation*}
have been studied in \cite{CP92,DK19,DK21,DK23,HKP20,KKL17,D23,Z05}. Different assumptions on the coefficients and various mixed-norm settings were considered in these works. Among these, the most recent result \cite{DK23} extended the theory to equations with partially small mean oscillation (SMO) coefficients in weighted mixed-norm spaces. We also refer the reader to \cite{DL22}, where the case $\alpha\in(1,2)$ is treated.

Parabolic equations that are local in time but nonlocal in space, such as
\begin{equation*}
    \partial_tu=Lu+f,
\end{equation*}
have also been studied extensively; see, for instance, \cite{CKP24,DJK23,DL23,DR25,JSS25,MP19,MP92,MP14,Z13}. In particular, it was shown in \cite{DL23} that weighted estimates are available for fractional Laplacian-like operators with kernels satisfying $K(t,y)\approx |y|^{-d-\sigma}$, while the results of \cite{DR25} suggest that, in a broad sense, such estimates may fail for singular kernels.

Equations that are nonlocal in both time and space, of the form
\begin{equation*}
    \partial_t^\alpha u= Lu + f,
\end{equation*}
have also been studied in \cite{CKPS25,DL23,KPR21,YZ25} under various assumptions on the nonlocal operator $L$. In particular, \cite{DL23} allowed $L$ to be merely measurable in time in the unweighted $L_p$ setting. On the other hand, \cite{YZ25} established weighted mixed-norm estimates for time-independent operators $L$.

These results motivate us to study such equations with a large class of time-dependent operators in weighted mixed-norm spaces. This extension, however, is far from straightforward. In particular, the approach used in \cite{DL23} for mixed-norm estimates when $\alpha=1$ is not directly applicable to the case $\alpha\in(0,1)$. 
A main difficulty comes from the anisotropic structure of the equation, with distinct nonlocal operators in the temporal and spatial variables. As a result, localization produces separate nonlocal tails in each variable, which cannot be treated as a single nonlocal tail.
A further difficulty comes from the time dependence of the spatial operator. In the present paper, we allow the kernel to be merely measurable in time. More precisely, for each $y\in\bR^d$,
\begin{equation*}
    t\mapsto K(t,y)
\end{equation*}
is assumed to be merely measurable.

In \cite{YZ25}, Sobolev estimates are derived from a representation formula and pointwise estimates of the corresponding fundamental solution. Such an approach is not applicable to general time-dependent operators because, to the best of our knowledge, a general representation formula in terms of the fundamental solution is not known.
A main novelty of this paper is that we develop a kernel-free approach to obtain weighted mixed-norm estimates.

We now describe the main results of the paper. Let $p,q\in(1,\infty)$ and $\omega(t,x)=\omega_1(t)\omega_2(x)$, where $\omega_1\in A_p(\bR)$ and $\omega_2\in A_q(\bR^d)$. We consider the weighted mixed-norm spaces
\begin{equation*}
    \bL_{p,q,\omega}(T) := L_p\left((0,T),\omega_1dt; L_q(\bR^d,\omega_2dx)\right)
\end{equation*}
and the odd spaces
\begin{equation*}
    \bm{\fL}_{q,p,\omega}(T) := L_q\left(\bR^d,\omega_2dx; L_p((0,T),\omega_1dt)\right).
\end{equation*}
Our goal is to prove the weighted mixed-norm estimate
\begin{equation} \label{eq_intro_est}
    \|\partial_t^\alpha u\|_{\bL_{p,q,\omega}(T)} + \|(-\Delta)^{\sigma/2}u\|_{\bL_{p,q,\omega}(T)} + \lambda\|u\|_{\bL_{p,q,\omega}(T)} \leq N\|f\|_{\bL_{p,q,\omega}(T)}
\end{equation}
and its counterpart in the odd mixed-norm space,
\begin{equation} \label{eq_intro_est_odd}
    \|\partial_t^\alpha u\|_{\bm{\fL}_{q,p,\omega}(T)} + \|(-\Delta)^{\sigma/2}u\|_{\bm{\fL}_{q,p,\omega}(T)} + \lambda\|u\|_{\bm{\fL}_{q,p,\omega}(T)} \leq N\|f\|_{\bm{\fL}_{q,p,\omega}(T)}.
\end{equation}
The constants $N$ in \eqref{eq_intro_est} and \eqref{eq_intro_est_odd} are independent of $\lambda$ and $T$, and are robust in the sense that they remain bounded as $\alpha\to1$ and $\sigma\to2$.
Moreover, if either $T<\infty$ or $\lambda>0$, then for any $f$ in $\bL_{p,q,\omega}(T)$ or $\bm{\fL}_{q,p,\omega}(T)$, equation \eqref{eq_intro} admits a unique solution with zero initial condition in the corresponding space.

For the proof of the a priori estimates, we introduce a direction-by-direction extension procedure. Instead of treating the mixed space-time norms at once, we fix the exponent and weight in one direction and extend the estimate in the other direction. In this way, the temporal and the spatial nonlocal tails are handled separately.

We first develop an extension principle in the time direction. Fix the spatial exponent $q$ and the spatial weight $\omega_2$, and regard the solution as a function of time with values in the Banach space $X:=L_q(\bR^d,\omega_2dx)$.
We prove that if the estimate and solvability hold for one temporal exponent $p_0$, then they hold for every $p\in(1,\infty)$ and every temporal weight $\omega_1\in A_p(\bR)$. This result can be viewed as a time-fractional analogue of the extrapolation principle for evolution equations in \cite[Theorem 17.2.39]{HNVW23}. To this end, we first use a level set argument together with the ``crawling of ink spots lemma" to cover the full range of exponents in the unweighted case in time. We then derive a mean oscillation estimate through an iteration argument. Finally, we apply the weighted sharp function and maximal function theorems only in the time variable to obtain the desired weighted mixed-norm estimates.

To extend the spatial exponent and spatial weight, we reverse the roles of time and space and consider the odd mixed-norm space. In this case, with the temporal exponent $p$ and temporal weight $\omega_1$ fixed, we regard
\begin{equation*}
    x\mapsto u(\cdot,x)
\end{equation*}
as a function on $\bR^d$ with values in $Y:=L_p((0,T),\omega_1dt)$.
We then use an analogous argument in the spatial variables, which allows us to derive the corresponding spatial estimates and extend the result to general spatial exponents and Muckenhoupt weights. This yields the desired weighted mixed-norm estimates in the odd mixed-norm space.
Finally, since the standard and odd mixed-norm spaces coincide when $p=q$, the result in the standard mixed-norm space follows by combining the result of the spatial extension in the odd space with the time-direction extension.

The remainder of the paper is organized as follows. In Section \ref{sec_main}, we state the assumptions on the operators, introduce the weighted mixed-norm spaces, and present the main results. In Section \ref{sec_time}, keeping the spatial exponent and spatial weight fixed, we extend the estimates and solvability in the time direction. In Section \ref{sec_space}, we complete the theory in the odd Sobolev space by extending the result to general spatial exponents and spatial weights. In Section \ref{sec_proof}, we combine the temporal and spatial extension results to prove the main theorem in the standard Sobolev space. In Appendix \ref{sec_appen}, we present the auxiliary results used in the proofs.

We conclude this section by collecting the notation and conventions used throughout the paper.
 We use ``$:=$'' or ``$=:$'' to denote a definition.
  We denote the set of natural numbers by $\bN$, and write $\bN_0:=\bN\cup\{0\}$. As usual, $\bR^d$ stands for the Euclidean space of points $x=(x_1,\dots,x_d)$, and we write $\bR:=\bR^1$.
  For $r>0$, we denote by $B_r:=\{x\in \bR^d: |x|<r\}$ the spatial ball centered at the origin with radius $r$.
We use $D^n_x u$ to denote the partial derivatives of order $n\in\bN_0$ with respect to the space variables, and $D_xu:=D_x^1u$. 
For $E\subset \bR^{d}$ such that $|E|<\infty$, we denote by
\begin{equation*}
\aint_{E} f(x) dx := (f)_E := \frac{1}{|E|}\int_E f(x) dx
\end{equation*}
the average of $f$ over a measurable set $E$.
For $D\subset \bR^{d+1}$, we denote by $C_c^\infty(D)$ the collection of all infinitely differentiable functions with compact support in $D$.

\section{Main results} \label{sec_main}

Throughout this paper, we always assume that $d\geq1$, and the kernel of any operator of the form \eqref{eq_op} is measurable in both $t$ and $y$.

We impose the following assumptions on $K$, the kernel of \eqref{eq_op}. First, we state the nondegeneracy condition.

\begin{assumption} \label{as_nonde}
    (i) There exist $\nu,\Lambda>0$ such that
\begin{equation*}
    0\leq K(t,y)\leq (2-\sigma) \Lambda |y|^{-d-\sigma}
\end{equation*}
and
\begin{equation*}
    \int_{|\xi\cdot y|\leq1} |\xi\cdot y|^2 K(t,y) dy \geq (2-\sigma) \nu |\xi|^{\sigma}.
\end{equation*}

$(ii)$ If $\sigma=1$, then 
\begin{equation*}
    \int_{r_1\leq |y|\leq r_2} yK(t,y) dy=0, \quad 0<r_1<r_2<\infty.
\end{equation*}
\end{assumption}

For the boundedness of the operators, we only need the following assumption.

\begin{assumption} \label{as_upper}
 There exists $\Lambda>0$ such that
\begin{equation} \label{upper}
    0\leq K(t,y)\leq (2-\sigma) \Lambda |y|^{-d-\sigma},
\end{equation}
and if $\sigma=1$, then 
\begin{equation} \label{cancel}
    \int_{r_1\leq |y|\leq r_2} yK(t,y) dy=0, \quad 0<r_1<r_2<\infty.
\end{equation}
\end{assumption}

We next introduce some notation used in this paper.
For $\alpha\in(0,\infty)$ and $v\in L_1((S,T))$, the Riemann-Liouville fractional integral of order $\alpha$ is defined as
    \begin{equation*}
        I_S^\alpha v(t) := \frac{1}{\Gamma(\alpha)} \int_{S}^{t} (t-s)^{\alpha-1} v(s) ds, \quad S\leq t\leq T.
    \end{equation*}
    We also define $I^0 v := v$.
    For $\alpha\in(0,1)$, the Riemann-Liouville fractional derivative of order $\alpha$ is defined as
    \begin{equation*}
        D_{t,S}^\alpha v(t) := \frac{d}{dt}I_{S}^{1-\alpha}v(t) = \frac{1}{\Gamma(1-\alpha)} \frac{d}{dt}\int_S^t (t-s)^{-\alpha} v(s) ds,
    \end{equation*}
    provided that $I_{S}^{1-\alpha}v$ is absolutely continuous.
    We also define the Caputo fractional derivative of order $\alpha$ by
    \begin{equation*}
        \partial_{t,S}^{\alpha} v(t) := D_{t,S}^{\alpha}(v(\cdot)-v(S))(t) = \frac{1}{\Gamma(1-\alpha)} \frac{d}{dt}\int_S^t (t-s)^{-\alpha} (v(s)-v(S)) ds.
    \end{equation*}
    Note that if $v(S)=0$, then $D_{t,S}^{\alpha}v=\partial_{t,S}^\alpha v$.
    When $S=0$, we omit $S$, that is, $\partial_t^\alpha v:= \partial_{t,0}^\alpha v$.

We introduce spaces with Muckenhoupt weights.
A locally integrable nonnegative function $\omega$ on $\bR^d$ is said to belong to the $A_p(\bR^d)$ Muckenhoupt class of weights if
\begin{equation} \label{eqAp}
    [\omega]_{A_p(\bR^d)}:= \sup_{r>0,x_0\in\bR^d} \left( \aint_{B_r(x_0)} \omega(x) dx \right) \left( \aint_{B_r(x_0)} (\omega(x))^{1/(1-p)} dx \right)^{p-1} <\infty.
\end{equation}
Let $p,q\in(1,\infty)$ and $K_0>0$. Then in $\bR^{d+1}$, we write $[\omega]_{p,q} \leq K_0$ if $\omega(t,x)=\omega_1(t)\omega_2(x)$ for some $\omega_1\in A_p(\bR)$ and $\omega_2\in A_q(\bR^d)$ satisfying
\begin{equation*}
    [\omega_1]_{A_p(\bR)}, [\omega_2]_{A_q(\bR^d)} \leq K_0.
\end{equation*}
We define 
\begin{equation*}
    L_{p,\omega_1}(T):= L_p((0,T),\omega_1 dt), \quad L_{q,\omega_2}(\bR^d) := L_q(\bR^d, \omega_2 dx).
\end{equation*}
For $-\infty\leq S< T\leq \infty$, $\omega_1\in A_p(\bR)$, and $\omega_2\in A_q(\bR^d)$, we write
\begin{equation*}
    \bL_{p,q,\omega}(S,T) := L_p((S,T),\omega_1dt;L_{q}(\bR^d,\omega_2 dx))
\end{equation*}
and
\begin{equation*}
\bm{\fL}_{q,p,\omega}(S,T):= L_{q}(\bR^d,\omega_2 dx;L_p((S,T),\omega_1dt)).
\end{equation*}
For $S=0$, we write $\bL_{p,q,\omega}(T):=\bL_{p,q,\omega}(0,T)$ and $\bm{\fL}_{q,p,\omega}(T) := \bm{\fL}_{q,p,\omega}(0,T)$.
We denote $\bL_{p,q}(S,T):=\bL_{p,q,1}(S,T)$ and $\bm{\fL}_{q,p}(S,T):=\bm{\fL}_{q,p,1}(S,T)$. Here, we also omit $S$ if $S=0$.

\begin{remark} \label{remCc}
    We briefly show that $C_c^\infty((S,T)\times \bR^d)$ is dense in both $\bL_{p,q,\omega}(S,T)$ and $\bm{\fL}_{q,p,\omega}(S,T)$.

    By the dominated convergence theorem, $f1_{(S+1/n,T-1/n)\times B_{n}} \to f$ as $n\to \infty$ in either of the above spaces. In the case when $S=-\infty$, we can consider $f1_{(-n,T-1/n)\times B_n}$. The case $T=\infty$ can be handled similarly.
    Thus, we may assume that $f$ is compactly supported. By extending $f$ to be zero outside $(S,T)\times \bR^d$, it suffices to consider the case when $S=-\infty$ and $T=\infty$.
    Let $\eta\in C_c^\infty((0,1))$ and $\varphi\in C_c^\infty(B_1)$ be nonnegative functions such that $\int_{\bR} \eta dt=1$ and $\int_{\bR^d} \varphi dx =1$. For $\varepsilon>0$, define $\eta^\varepsilon(t):=\varepsilon^{-1}\eta(t/\varepsilon)$ and $\varphi^\varepsilon(x):=\varepsilon^{-d}\varphi(x/\varepsilon)$. Consider the following space-time mollification
    \begin{equation*}
        f^\varepsilon(t,x) := \int_{\bR^{d+1}} f(s,y) \eta^\varepsilon(t-s) \varphi^\varepsilon(x-y) dyds.
    \end{equation*}
    One can easily find that
    \begin{align*}
        |f^\varepsilon(t,x)| \leq N\aint_{t-\varepsilon}^{t+\varepsilon} \aint_{B_\varepsilon(x)} |f(s,y)| dyds.
    \end{align*}
    Consequently,
    \begin{align*}
        |f^\varepsilon(t,x)| \leq \bM_t\bM_x f(t,x), \qquad |f^\varepsilon(t,x)| \leq \bM_x\bM_t f(t,x).
    \end{align*}
    Here, $\bM_t$ denotes the standard maximal operator in the time variable, defined by
\begin{equation*}
    \bM_t h(t_0) := \sup_{(t-R,t+R)\ni t_0} \aint_{t-R}^{t+R} |h(s)| ds,
\end{equation*}
and $\bM_x$ denotes the uncentered Hardy-Littlewood maximal operator in the spatial variables, defined by
\begin{equation*}
    \bM_x h(x_0):=\sup_{B_R(x)\ni x_0}\aint_{B_R(x)}|h(y)| dy.
\end{equation*}
Since $f^\varepsilon\to f$ almost everywhere as $\varepsilon\to0$, the weighted maximal function theorem (see \cite[Theorem 2.2]{DK18}) provides an integrable dominating function in each of the two mixed-norm spaces.
Hence, by the dominated convergence theorem, we see that $f^\varepsilon\to f$ as $\varepsilon\to0$ in both $\bL_{p,q,\omega}(S,T)$ and $\bm{\fL}_{q,p,\omega}(S,T)$.
Since $f$ is compactly supported, $f^\varepsilon\in C_c^\infty((S,T)\times\bR^d)$ for all sufficiently small $\varepsilon>0$.
The proof is completed.
\end{remark}

The operator $\partial_{t,S}^\alpha$ defined above is for sufficiently regular functions. We now extend this definition to functions in weighted mixed-norm spaces in the weak sense. For $u$ in $\bL_{p,q,\omega}(S,T)$ or $\bm{\fL}_{q,p,\omega}(S,T)$, we say that $\partial_{t,S}^\alpha u$ exists in the weak sense if there exists a function $g$ in the same space as $u$ such that for any $\varphi\in C_c^\infty((S,T)\times\bR^d)$,
\begin{equation} \label{eq7291255}
    \int_S^T \int_{\bR^d} g(t,x) \varphi(t,x) dxdt = -\int_S^T \int_{\bR^d} I_S^{1-\alpha}u(t,x) \partial_t\varphi(t,x) dxdt.
\end{equation}
In this case, we write $\partial_{t,S}^\alpha u = g$.
The spatial operator $L$ is understood analogously in the weak sense. More precisely, we say that $Lu$ exists in the weak sense if there exists a function $h$ in the same space as $u$ such that for every $\varphi\in C_c^\infty((S,T)\times\bR^d)$,
\begin{equation*}
\int_S^T\int_{\bR^d} h(t,x)\varphi(t,x) dxdt = \int_S^T\int_{\bR^d} u(t,x)L^*\varphi(t,x) dxdt,
\end{equation*}
where $L^*$ is the operator with the kernel $K(t,-y)$. In this case, we write $Lu=h$.

We now introduce the standard and odd Sobolev spaces. In particular, the spaces $\bH_{p,q,\omega,0}^{\alpha,\sigma}(S,T)$ and $\bm{\fH}_{q,p,\omega,0}^{\alpha,\sigma}(S,T)$ are defined to deal with the zero initial condition.

\begin{definition}
    Let $p,q\in(1,\infty)$, $\alpha\in(0,1)$, $\sigma\in(0,2)$, $-\infty<S<T\leq \infty$, and $\omega$ be a weight such that $[\omega]_{p,q}\leq K_0$.
    
    $(i)$ We write
\begin{equation*}
    \bH_{p,q,\omega}^{\alpha,\sigma}(S,T) := \{ u\in \bL_{p,q,\omega}(S,T): \partial_{t,S}^{\alpha}u, (-\Delta)^{\sigma/2}u \in \bL_{p,q,\omega}(S,T) \},
\end{equation*}
where the norm is given by
\begin{equation*}
    \|u\|_{\bH_{p,q,\omega}^{\alpha,\sigma}(S,T)} := \|u\|_{\bL_{p,q,\omega}(S,T)} + \|\partial_{t,S}^{\alpha}u\|_{\bL_{p,q,\omega}(S,T)} + \|(-\Delta)^{\sigma/2}u\|_{\bL_{p,q,\omega}(S,T)}.
\end{equation*}
Similarly, we define the odd Sobolev space
\begin{equation*}
    \bm{\fH}_{q,p,\omega}^{\alpha,\sigma}(S,T) := \{u\in \bm{\fL}_{q,p,\omega}(S,T): \partial_{t,S}^\alpha u,(-\Delta)^{\sigma/2}u \in \bm{\fL}_{q,p,\omega}(S,T) \},
\end{equation*}
equipped with the norm
\begin{equation*}
    \|u\|_{\bm{\fH}_{q,p,\omega}^{\alpha,\sigma}(S,T)} := \|u\|_{\bm{\fL}_{q,p,\omega}(S,T)} + \|\partial_{t,S}^{\alpha}u\|_{\bm{\fL}_{q,p,\omega}(S,T)} + \|(-\Delta)^{\sigma/2}u\|_{\bm{\fL}_{q,p,\omega}(S,T)}.
\end{equation*}

$(ii)$ For $u\in \bH_{p,q,\omega}^{\alpha,\sigma}(S,T)$, we say that $u \in \bH_{p,q,\omega,0}^{\alpha,\sigma}(S,T)$ if \eqref{eq7291255} is satisfied for all $\varphi\in C_c^\infty([S,T)\times \bR^d)$. The space $\bm{\fH}_{q,p,\omega,0}^{\alpha,\sigma}(S,T)$ is defined in a similar way.
\end{definition}

Next, we prove the density result in $\bH_{p,q,\omega,0}^{\alpha,\sigma}(S,T)$ and $\bm{\fH}_{q,p,\omega,0}^{\alpha,\sigma}(S,T)$. For function spaces corresponding to spatially local operators, we refer the reader to \cite[Proposition 3.6]{KW25}.

\begin{lemma} \label{lem_dense}
    Let $u\in\bH_{p,q,\omega,0}^{\alpha,\sigma}(S,T)$ or $u\in\bm{\fH}_{q,p,\omega,0}^{\alpha,\sigma}(S,T)$. Then there exists a sequence of functions $u_n \in C_c^\infty([S,T]\times\bR^d)$ with $u_n(S,\cdot)=0$ such that $u_n\to u$ in the corresponding space.
\end{lemma}
\begin{proof}
    As in Remark \ref{remCc}, we consider a space-time mollification of $u$, say $u^\varepsilon$, which is infinitely differentiable in $[S,T]\times \bR^d$. By investigating the proof of \cite[Proposition 3.2]{DK19}, one can deduce that $u^\varepsilon(S,x)=0$ and
    \begin{equation} \label{eq8072150}
        (\partial_{t,S}^\alpha u)^\varepsilon = \partial_{t,S}^\alpha u^\varepsilon, \qquad ((-\Delta)^{\sigma/2}u)^\varepsilon = (-\Delta)^{\sigma/2}u^\varepsilon.
    \end{equation}
    Thus, we see that $u^\varepsilon \to u$ in the corresponding space.

Although $u^\varepsilon$ is smooth, it is not necessarily compactly supported. We next show that $u^\varepsilon$ can be approximated by smooth compactly supported functions.
    By an interpolation inequality (see e.g. \cite[Theorem 1.2]{SV20}),
    \begin{equation} \label{eq7292333}
        \|(-\Delta)^{\sigma/2}u^\varepsilon(t,\cdot)\|_{L_{q,\omega_2}(\bR^d)} \leq N\left( \|u^\varepsilon(t,\cdot)\|_{L_{q,\omega_2}(\bR^d)} + \|\Delta u^\varepsilon(t,\cdot)\|_{L_{q,\omega_2}(\bR^d)} \right).
    \end{equation}
    Thus, for any $\omega_1\in A_p(\bR)$, by raising both sides of this inequality to the power of $p$, multiplying by $\omega_1(t)$, and integrating with respect to $t$,
    \begin{align} \label{eq7291548}
    \|u^\varepsilon\|_{\bH_{p,q,\omega}^{\alpha,\sigma}(S,T)} &\leq N \|u^\varepsilon\|_{\bH_{p,q,\omega}^{\alpha,2}(S,T)} \nonumber
    \\
    &:= N \left(\|u^\varepsilon\|_{\bL_{p,q,\omega}(S,T)} + \|\partial_{t,S}^{\alpha}u^\varepsilon\|_{\bL_{p,q,\omega}(S,T)} + \|\Delta u^\varepsilon\|_{\bL_{p,q,\omega}(S,T)}\right).
\end{align}
Let $\eta_n \in C_c^\infty(B_{2n})$ be a sequence of cutoff functions such that $\eta_n=1$ in $B_n$, and $|D_x^k\eta_n|\leq Nn^{-k}$. Then one can verify that $u^\varepsilon\eta_n\to u^\varepsilon$ in $\bH_{p,q,\omega}^{\alpha,2}(S,T)$. Since $u^\varepsilon\eta_n \in C_c^\infty([S,T]\times \bR^d)$ and $(u^\varepsilon\eta_n)(S,x)=0$, \eqref{eq7291548} leads to the desired result in $\bH_{p,q,\omega,0}^{\alpha,\sigma}(S,T)$.

Lastly, we consider $\bm{\fH}_{q,p,\omega,0}^{\alpha,\sigma}(S,T)$. When $p=q$, we have $\bH_{p,p,\omega,0}^{\alpha,\sigma}(S,T) = \bm{\fH}_{p,p,\omega,0}^{\alpha,\sigma}(S,T)$. Thus, by applying the extrapolation theorem to \eqref{eq7291548} in the case $p=q$ (see e.g. \cite[Theorem 2.5]{DK18}), we obtain that
\begin{equation} \label{eq8102022}
    \|u^\varepsilon\|_{\bm{\fH}_{q,p,\omega}^{\alpha,\sigma}(S,T)} \leq N \|u^\varepsilon\|_{\bm{\fH}_{q,p,\omega}^{\alpha,2}(S,T)},
\end{equation}
where $\bm{\fH}_{q,p,\omega}^{\alpha,2}(S,T)$ is defined in the natural way. As above, one can also show that $u^\varepsilon\eta_n\to u$ in $\bm{\fH}_{q,p,\omega}^{\alpha,\sigma}(S,T)$ by first letting $n\to\infty$ and then $\varepsilon\to0$, which proves the density result in $\bm{\fH}_{q,p,\omega,0}^{\alpha,\sigma}(S,T)$. The proof is completed.
\end{proof}

Let
\begin{equation*}
    H_{q}^{\sigma}(\bR^d,\omega_2 dx) := \{ u\in L_{q}(\bR^d,\omega_2 dx): (-\Delta)^{\sigma/2}u \in L_{q}(\bR^d,\omega_2 dx) \}.
\end{equation*}
Suppose that $L$ satisfies Assumption \ref{as_upper}. Then for each fixed $t$, the continuity of $L=L(t)$ from $H_{q}^{\sigma}(\bR^d,\omega_2 dx)$ to $L_{q}(\bR^d,\omega_2 dx)$ was established in \cite[Proposition 4.1]{DJK23}. Using this result, we show the continuity of $L$ from both $\bH_{p,q,\omega}^{\alpha,\sigma}(S,T)$ and $\bm{\fH}_{q,p,\omega}^{\alpha,\sigma}(S,T)$ into the corresponding weighted mixed-norm spaces.

\begin{proposition} \label{lem_conti}
    Let $\alpha\in(0,1)$, $\sigma\in(0,2)$, $-\infty<S<T\leq \infty$, $p,q\in(1,\infty)$, $K_0>0$, and $[\omega]_{p,q}\leq K_0$.
Suppose that $L$ satisfies Assumption \ref{as_upper}.

    (i) Then $L$ is continuous from $\bH_{p,q,\omega}^{\alpha,\sigma}(S,T)$ to $\bL_{p,q,\omega}(S,T)$. Moreover,
\begin{equation} \label{eq7131450}
    \|Lu\|_{\bL_{p,q,\omega}(S,T)} \leq N \|(-\Delta)^{\sigma/2}u\|_{\bL_{p,q,\omega}(S,T)}.
\end{equation}
    Here, the constant $N$ can be chosen so that $N=N(d,p,q,K_0,\Lambda,\tilde{\sigma})$, where
    \begin{equation} \label{sigma}
        \tilde{\sigma} = 
\begin{cases}
(\sigma_0, \sigma_1) & \text{when } 0 < \sigma_0 \leq \sigma \leq \sigma_1 < 1, \\
1 & \text{when } \sigma = 1, \\
\sigma_0 & \text{when } 1 < \sigma_0 \leq \sigma < 2.
\end{cases}
    \end{equation}
    In particular, when $\sigma=1$, one can also consider $\nabla u$ instead of $Lu$.

    (ii) The same assertions hold with $\bm{\fH}_{q,p,\omega}^{\alpha,\sigma}(S,T)$ and $\bm{\fL}_{q,p,\omega}(S,T)$ in place of $\bH_{p,q,\omega}^{\alpha,\sigma}(S,T)$ and $\bL_{p,q,\omega}(S,T)$.
\end{proposition}

\begin{proof}
    $(i)$ 
    By \cite[Proposition 4.1]{DJK23}, for each $t\in(S,T)$,
    \begin{equation*}
    \|Lu(t,\cdot)\|_{L_{q,\omega_2}(\bR^d)} \leq N(d,q,K_0,\Lambda,\tilde{\sigma}) \|(-\Delta)^{\sigma/2}u(t,\cdot)\|_{L_{q,\omega_2}(\bR^d)}.
\end{equation*}
In the case when $\sigma=1$ and $L=\nabla$, use \cite[Lemma 3.5]{DJK23}.
We raise both sides of this inequality to the power of $p$, multiply by $\omega_1(t)$, and integrate with respect to $t$. Then we obtain \eqref{eq7131450}.

$(ii)$ 
Since $\bm{\fL}_{p,p,\omega}(S,T)=\bL_{p,p,\omega}(S,T)$, the case $p=q$ follows from \eqref{eq7131450}. For the general case $p\neq q$, one just needs to use the extrapolation theorem. The proposition is proved.
\end{proof}

Our first main result is formulated in the standard Sobolev space $\bH_{p,q,\omega,0}^{\alpha,\sigma}(T)$.

\begin{theorem} \label{thm_main}
Let $\alpha\in(0,1)$, $\sigma\in(0,2)$, $T\in(0,\infty]$, $\lambda\geq0$, $p,q\in(1,\infty)$, $K_0>0$, and $[\omega]_{p,q}\leq K_0$. 

(i) Suppose that $L$ and $\cL$ satisfy Assumptions \ref{as_nonde} and \ref{as_upper}, respectively. Then for any $u\in \bH_{p,q,\omega,0}^{\alpha,\sigma}(T)$ satisfying
\begin{equation} \label{eq_main}
    \partial_t^\alpha u = Lu -\lambda u +f \text{ in } \bR^{d}_T,
\end{equation}
we have
\begin{equation} \label{est_main}
    \|\partial_t^{\alpha}u\|_{\bL_{p,q,\omega}(T)} + \|\cL u\|_{\bL_{p,q,\omega}(T)} + \lambda\|u\|_{\bL_{p,q,\omega}(T)} \leq N \|f\|_{\bL_{p,q,\omega}(T)}.
\end{equation}
Here, the constant $N$ can be chosen so that $N=N(d,p,q,K_0,\Lambda,\nu,\tilde{\alpha},\tilde{\sigma})$, where $0<\tilde{\alpha}\leq \alpha<1$ and $\tilde{\sigma}$ is given by \eqref{sigma}.

(ii) Suppose, in addition, that either $T<\infty$ or $\lambda>0$. Then for any $f\in \bL_{p,q,\omega}(T)$, there exists a unique solution $u\in \bH_{p,q,\omega,0}^{\alpha,\sigma}(T)$ to \eqref{eq_main}.
\end{theorem}

The next result is the counterpart of Theorem \ref{thm_main} in the odd Sobolev space $\bm{\fH}_{q,p,\omega,0}^{\alpha,\sigma}(T)$.

\begin{theorem} \label{thm_odd}
    Let $\alpha\in(0,1)$, $\sigma\in(0,2)$, $T\in(0,\infty]$, $\lambda\geq0$, $p,q\in(1,\infty)$, $K_0>0$, and $[\omega]_{p,q}\leq K_0$. 

(i) Suppose that $L$ and $\cL$ satisfy Assumptions \ref{as_nonde} and \ref{as_upper}, respectively. Then for any $u\in \bm{\fH}_{q,p,\omega,0}^{\alpha,\sigma}(T)$ satisfying
\begin{equation} \label{eqmainodd}
    \partial_t^\alpha u = Lu -\lambda u +f \text{ in } \bR^{d}_T,
\end{equation}
we have
\begin{equation} \label{est_odd}
    \|\partial_t^{\alpha}u\|_{\bm{\fL}_{q,p,\omega}(T)} + \|\cL u\|_{\bm{\fL}_{q,p,\omega}(T)} + \lambda\|u\|_{\bm{\fL}_{q,p,\omega}(T)} \leq N \|f\|_{\bm{\fL}_{q,p,\omega}(T)}.
\end{equation}
Here, $N$ can be chosen so that $N=N(d,p,q,K_0,\Lambda,\nu,\tilde{\alpha},\tilde{\sigma})$, where $0<\tilde{\alpha}\leq\alpha<1$ and $\tilde{\sigma}$ is given by \eqref{sigma}.

(ii) Suppose, in addition, that either $T<\infty$ or $\lambda>0$. Then for any $f\in \bm{\fL}_{q,p,\omega}(T)$, there exists a unique solution $u\in \bm{\fH}_{q,p,\omega,0}^{\alpha,\sigma}(T)$ to \eqref{eqmainodd}. 
\end{theorem}

\begin{remark}
    $(i)$ 
    By the extrapolation theorem (see \cite[Theorem 2.5]{DK18}), each of \eqref{est_main} and \eqref{est_odd} follows from the other. For instance, once \eqref{est_main} is established, the estimate \eqref{est_odd} in the odd space follows readily from the extrapolation theorem. 
    
    $(ii)$ In both Theorems \ref{thm_main} and \ref{thm_odd}, the estimates \eqref{est_main} and \eqref{est_odd} are robust in the sense that the constant $N$ does not blow up as $\alpha\to1$ and $\sigma\to2$.
\end{remark}

\section{Extensions in the time direction} \label{sec_time}

In this section, we show that, once the spatial exponent and weight are fixed, the result can be extended in the time direction to arbitrary temporal exponents and weights.
More precisely, the aim of this section is to prove the following.

\begin{theorem} \label{thm_time}
    Let the parameters in Theorem \ref{thm_main}, except for $p$ and the weight $\omega_1$, be fixed. Suppose that, for some $p_0\in(1,\infty)$, the assertions of Theorem \ref{thm_main} hold with $p=p_0$ and with the spatial weight $\omega_2(x)$ in place of the product weight $\omega_1(t)\omega_2(x)$. Then the assertions of Theorem \ref{thm_main} hold for every $p\in(1,\infty)$ and every product weight $\omega(t,x)=\omega_1(t)\omega_2(x)$, where $\omega_1$ is an arbitrary weight in $A_p(\bR)$ satisfying $[\omega_1]_{A_p(\bR)}\leq K_0$.
\end{theorem}

\begin{remark}
    We refer the reader to \cite[Theorem 17.2.39]{HNVW23}, where extrapolation in the time variable is established for equations of the form
    \begin{equation*}
        \partial_t u = Au + f,
    \end{equation*}
    where $A$ is a certain linear operator. Our theorem can be viewed as a time-fractional analogue of that result. Our purely analytic approach, however, allows us to treat equations that are nonlocal in both time and space, with the kernel of the spatial operator merely measurable in time.
\end{remark}

\begin{remark} \label{rem_pqtime}
    In \cite[Theorem 2.6]{DL23}, Theorem \ref{thm_main} is proved when $p=q$ and $\omega=1$. Together with Theorem \ref{thm_time}, this yields \eqref{est_main} and the solvability result for general $p\neq q$ with a temporal weight $\omega(t,x)=\omega_1(t)$.
\end{remark}

We divide the proof into two parts, with the spatial exponent $q$ and  the spatial weight $\omega_2$ fixed throughout. We first extend the result to all temporal exponents $p\in(1,\infty)$ under $\omega(t,x)=\omega_2(x)$ in Section \ref{sec_time_level}, and then to product weights $\omega(t,x)=\omega_1(t)\omega_2(x)$ in Section \ref{sec_time_gen}.

\subsection{Extension of the temporal exponent} \label{sec_time_level} 
The aim of this subsection is to prove the following statement.

\begin{proposition} \label{prop7021543}
Let the parameters in Theorem \ref{thm_main}, except for $p$ and the weight $\omega$, be fixed, and suppose that $L$ and $\cL$ satisfy the assumptions of Theorem \ref{thm_main}. Suppose that, for some $p_0\in(1,\infty)$, the assertions of Theorem \ref{thm_main} hold with $p=p_0$ and $\omega(t,x)=\omega_2(x)$. Then the same assertions hold for every $p\in(1,\infty)$ with $\omega(t,x)=\omega_2(x)$.
\end{proposition}

\begin{lemma} \label{lem7021436}
    Let the assumptions of Proposition \ref{prop7021543} be satisfied. Assume further that $\cL$ is independent of $t$.
    Let $u\in \bH_{p_0,q,\omega_2,0}^{\alpha,\sigma}(T)$ be a solution to \eqref{eq_main}. Then there exists $p_1=p_1(\tilde{\alpha},p_0)\in (p_0,\infty]$ such that
    \begin{equation} \label{eq6171544}
        p_1-p_0>\frac{\tilde{\alpha}}{2}>0
    \end{equation}
    and the following holds. For any $t_0\in (0,T]$ (with the convention that $(0,\infty]:=(0,\infty)$ if $T=\infty$) and $R>0$, there exist $v,w\in \bH_{p_0,q,\omega_2,0}^{\alpha,\sigma}(t_0)$ such that $u=v+w$,
    \begin{align} \label{eq7011825}
       &\left(\aint_{t_0-R}^{t_0} \|\cL w(t,\cdot)\|_{X}^{p_0} dt \right)^{1/p_0} + \left(\aint_{t_0-R}^{t_0} \|\lambda w(t,\cdot)\|_{X}^{p_0} dt \right)^{1/p_0} \nonumber
       \\
       &\leq N \left(\aint_{t_0-2R}^{t_0} \|f(t,\cdot)\|_{X}^{p_0} dt\right)^{1/p_0},
    \end{align}
    \begin{align} \label{eq7020015}
        &\left(\aint_{t_0-R/4}^{t_0} \|\cL v(t,\cdot)\|_{X}^{p_1} dt\right)^{1/p_1} \nonumber
        \\
        &\leq N \sum_{k=0}^{\infty} 2^{-k\alpha} \left( \aint_{t_{0}-2^{k}R}^{t_0} \|\cL u(t,\cdot)\|_{X}^{p_0} dt \right)^{1/p_0} + N \left(\aint_{t_0-2R}^{t_0} \|f(t,\cdot)\|_{X}^{p_0} dt\right)^{1/p_0},
    \end{align}
    and
        \begin{align*}
                &\left(\aint_{t_0-R/4}^{t_0} \|\lambda v(t,\cdot)\|_{X}^{p_1} dt\right)^{1/p_1} \nonumber
        \\
        &\leq N \sum_{k=0}^{\infty} 2^{-k\alpha} \left( \aint_{t_{0}-2^{k}R}^{t_0} \|\lambda u(t,\cdot)\|_{X}^{p_0} dt \right)^{1/p_0} + N \left(\aint_{t_0-2R}^{t_0} \|f(t,\cdot)\|_{X}^{p_0} dt\right)^{1/p_0},
    \end{align*}
    where $X:=L_{q,\omega_2}(\bR^d)$, and $N=N(d,p_0,q,K_0,\Lambda,\nu,\tilde{\alpha},\tilde{\sigma})$ with $\tilde{\alpha}\leq\alpha<1$ and $\tilde{\sigma}$ defined as in \eqref{sigma}.
    Here, $u,v,w$, and $f$ are extended to be zero for $t<0$, and
    \begin{equation*}
        \left(\aint_{t_0-R/4}^{t_0} \|\cdot\|_{X}^{p_1} dt\right)^{1/p_1} := \|\cdot\|_{L_{\infty}((t_0-R/4,t_0);X)}
    \end{equation*}
    when $p_1=\infty$.
\end{lemma}

\begin{proof}
For general $R>0$, denote
\begin{equation*}
\tilde{u}(t,x) = R^{-\alpha} u(Rt, R^{\alpha/\sigma}x) \quad \text{ and } \quad \tilde{f}(t,x) = f(Rt, R^{\alpha/\sigma}x).
\end{equation*}
Then $\tilde{u}$ satisfies
\begin{equation*}
    \partial_t^\alpha \tilde{u} - \tilde{L}\tilde{u} + \lambda R^{\alpha} \tilde{u} = \tilde{f} \quad \text{in } (0, R^{-1}T) \times \mathbb{R}^d,
\end{equation*}
where $\tilde{L}$ is the operator with the kernel $R^{\alpha(d+\sigma)/\sigma} K(Rt, R^{\alpha/\sigma}y)$ satisfying Assumption \ref{as_nonde} with the same $\nu$ and $\Lambda$ as $K$, the kernel of $L$. Moreover, by \cite[Proposition 7.1.5]{G14},
\begin{equation*}
    [\omega_2(R^{\alpha/\sigma}\cdot)]_{A_q(\bR^d)} = [\omega_2]_{A_q(\bR^d)}.
\end{equation*}
Thus, we may assume that $R=1$.

Let $\eta_0 \in C^\infty(\bR)$ be a cutoff function such that $\eta_0=1$ for $t\geq t_0-1$, and $\eta_0=0$ for $t\leq t_0-2$. Since Theorem \ref{thm_main} holds for $p=p_0$ and $\omega=\omega_2(x)$, one can find a solution $w\in \bH_{p_0,q,\omega_2,0}^{\alpha,\sigma}(t_0)$ to
    \begin{equation*}
        \partial_t^\alpha w = Lw -\lambda w + f\eta_0.
    \end{equation*}
    Here, due to $t_0<\infty$, the existence of $w$ is guaranteed regardless of whether $\lambda=0$ or not.
    Moreover, we have
    \begin{eqnarray} \label{eq7181804}
        \|\partial_t^{\alpha}w\|_{\bL_{p_0,q,\omega_2}(t_0)} + \|\cL w\|_{\bL_{p_0,q,\omega_2}(t_0)} + \lambda\|w\|_{\bL_{p_0,q,\omega_2}(t_0)} \leq N \|f\eta_0\|_{\bL_{p_0,q,\omega_2}(t_0)}.
    \end{eqnarray}
    Since $w$ and $f$ are zero for $t<0$, we get \eqref{eq7011825}.

    Next, we consider $v:=u-w\in \bH_{p_0,q,\omega_2,0}^{\alpha,\sigma}(t_0)$. Due to the similarity, we only prove \eqref{eq7020015}.
    Let $\varphi\in C_c^\infty(\bR^d)$ be a nonnegative function such that $\int_{\bR^d} \varphi \, dx =1$. For $\varepsilon>0$, define $\varphi^\varepsilon(x):=\varepsilon^{-d}\varphi(x/\varepsilon)$, and for any function $h$ on $(-\infty,T)\times \bR^d$, $h^\varepsilon=h*\varphi^\varepsilon$, where $*$ denotes the convolution in the spatial variable.
    
    Let $\eta \in C^\infty(\bR)$ be a cutoff function such that $\eta=1$ for $t\geq t_0-1/4$, and $\eta=0$ for $t\leq t_0-1/2$. As in \cite[Lemma 3.6]{DK19}, one can see that $\tilde{v}^\varepsilon:=\cL(v^\varepsilon\eta)=\eta \cL v^\varepsilon\in \bH_{p_0,q,\omega_2,0}^{\alpha,\sigma}(t_0-1/2,t_0)$ satisfies
    \begin{equation} \label{eq7161314}
        \partial_{t,t_0-1/2}^\alpha \tilde{v}^\varepsilon = L\tilde{v}^\varepsilon - \lambda \tilde{v}^\varepsilon + g^{\varepsilon} \quad \text{ in } (t_0-1/2,t_0)\times \bR^d,
    \end{equation}
    where
    \begin{align*}
        g^{\varepsilon}(t,x) &= \partial_{t,t_0-1/2}^\alpha(\eta\cL v^\varepsilon) - \eta\partial_{t,t_0-1/2}^\alpha\cL v^\varepsilon 
        \\
        &= \frac{\alpha}{\Gamma(1-\alpha)} \int_{t_0-1/2}^{t} (t-s)^{-\alpha-1} (\eta(t)-\eta(s)) \cL v^\varepsilon(s,x) ds.
    \end{align*}
 Thus, by applying the time-shifted version of \eqref{est_main} with $p=p_0$ and $\omega(t,x)=\omega_2(x)$,
    \begin{align*}
        \|\partial_{t,t_0-1/2}^{\alpha}\tilde{v}^\varepsilon\|_{\bL_{p_0,q,\omega_2}(t_0-1/2,t_0)} \leq N \|g^{\varepsilon}\|_{\bL_{p_0,q,\omega_2}(t_0-1/2,t_0)}.
    \end{align*}
    As in Remark \ref{remCc}, one can show that
    \begin{equation*}
        |h^\varepsilon(t,x)| \leq N(\bM_x h(t,\cdot))(x).
    \end{equation*}
    Thus, by the dominated convergence theorem, $h^\varepsilon\to h$ in $\bL_{p_0,q,\omega_2}(t_0-1/2,t_0)$ as $\varepsilon\to 0$, which yields
    \begin{align*}
        \|\partial_{t,t_0-1/2}^{\alpha}\tilde{v}\|_{\bL_{p_0,q,\omega_2}(t_0-1/2,t_0)} \leq N \|g\|_{\bL_{p_0,q,\omega_2}(t_0-1/2,t_0)}.
    \end{align*}

Now we estimate $g$. Since $|\eta(t)-\eta(s)|\leq N\min\{|t-s|,1\}$, for $t\in(t_0-1/2,t_0)$,
\begin{align*}
    |g(t,x)| &\leq \frac{N}{\Gamma(1-\alpha)}\int_{t-1/2}^{t} |t-s|^{-\alpha} |\cL v(s,x)| ds 
    \\
    &\quad+ \frac{N}{\Gamma(1-\alpha)} \int_{-\infty}^{t-1/2} |t-s|^{-\alpha-1} |\cL v(s,x)|ds
    \\
    &=: I_1 + I_2.
\end{align*}
By a change of variables and the Minkowski inequality,
\begin{align} \label{eq7181844}
    &\|I_1\|_{\bL_{p_0,q,\omega_2}(t_0-1/2,t_0)} \nonumber
    \\
    &\leq \frac{N}{\Gamma(1-\alpha)} \int_0^{1/2} s^{-\alpha} \|\cL v(\cdot-s)\|_{\bL_{p_0,q,\omega_2}(t_0-1/2,t_0)} ds \nonumber
    \\
    &\leq \frac{N}{\Gamma(1-\alpha)} \|\cL v\|_{\bL_{p_0,q,\omega_2}(t_0-1,t_0)} \int_0^{1/2} s^{-\alpha} ds \nonumber
    \\
    &\leq \frac{N}{(1-\alpha)\Gamma(1-\alpha)} \|\cL v\|_{\bL_{p_0,q,\omega_2}(t_0-1,t_0)} = N\|\cL v\|_{\bL_{p_0,q,\omega_2}(t_0-1,t_0)}.
\end{align}
Here, for the last equality, we used the facts that $z\Gamma(z)=\Gamma(z+1)$ and that $1/\Gamma(z)$ is bounded for $z\in[1,2]$.
For $I_2$, since $t\in (t_0-1/2,t_0)$,
\begin{align*}
    |I_2(t,x)| &\leq N\sum_{k=0}^{\infty} \int_{t-2^{k}}^{t-2^{k-1}} |t-s|^{-\alpha-1} |\cL v(s,x)|ds
    \\
    &\leq N\sum_{k=0}^{\infty} 2^{-k(\alpha+1)} \int_{t-2^{k}}^{t-2^{k-1}}|\cL v(s,x)|ds
    \\
    &\leq N\sum_{k=0}^{\infty} 2^{-k(\alpha+1)} \int_{t_0-2^{k+1}}^{t_0}|\cL v(s,x)|ds.
\end{align*}
    Thus, again by the Minkowski inequality,
    \begin{align} \label{eq7181845}
    \|I_2\|_{\bL_{p_0,q,\omega_2}(t_0-1/2,t_0)} \leq N \sum_{k=0}^{\infty} 2^{-k\alpha} \aint_{t_{0}-2^{k+1}}^{t_0} \|\cL v(t,\cdot)\|_{X} dt.
\end{align}
    Hence, by H\"older's inequality,
    \begin{align} \label{eq7161456}
        \|g\|_{\bL_{p_0,q,\omega_2}(t_0-1/2,t_0)} &\leq N\|\cL v\|_{\bL_{p_0,q,\omega_2}(t_0-1,t_0)} \nonumber
        \\
        &\quad+ N \sum_{k=0}^{\infty} 2^{-k\alpha} \aint_{t_{0}-2^{k+1}}^{t_0} \|\cL v(t,\cdot)\|_{X} dt \nonumber
\\
        &\leq N \sum_{k=0}^{\infty} 2^{-k\alpha} \left(\aint_{t_{0}-2^{k}}^{t_0} \|\cL v(t,\cdot)\|_{X}^{p_0} dt \right)^{1/p_0}.
    \end{align}

    Take $p_1\in(p_0,\infty]$ so that
    \begin{equation*}
    1/p_1=1/p_0-\tilde{\alpha}/2 \quad \text{ if } p_0\leq 1/\tilde{\alpha},
    \end{equation*}
    and $p_1=\infty$ if $p_0>1/\tilde{\alpha}$.
    In particular, if $p_0 \leq 1/\tilde{\alpha}$, then
    \begin{equation*}
        p_1-p_0 = \frac{\tilde{\alpha} p_0^2}{2-\tilde{\alpha} p_0} \geq \frac{\tilde{\alpha} p_0^2}{2} >  \frac{\tilde{\alpha}}{2},
    \end{equation*}
    which yields \eqref{eq6171544}.
    Then by \eqref{eq7151540},
\begin{equation} \label{eq7022034}
     \|\tilde{v}\|_{\bL_{p_1,q,\omega_2}(t_0-1/2,t_0)} \leq N \|\partial_{t,t_0-1/2}^\alpha\tilde{v}\|_{\bL_{p_0,q,\omega_2}(t_0-1/2,t_0)}.
\end{equation}
    This together with \eqref{eq7161456} leads to
    \begin{align} \label{eq7172038}
        \left(\aint_{t_0-1/4}^{t_0} \|\cL v(t,\cdot)\|_{X}^{p_1} dt\right)^{1/p_1} \leq N \sum_{k=0}^{\infty} 2^{-k\alpha} \left(\aint_{t_{0}-2^{k}}^{t_0} \|\cL v(t,\cdot)\|_{X}^{p_0} dt \right)^{1/p_0}.
    \end{align}
Since $v=u-w$, by \eqref{eq7181804},
\begin{align*}
    &\left(\aint_{t_{0}-2^{k}}^{t_0} \|\cL v(t,\cdot)\|_{X}^{p_0} dt \right)^{1/p_0}
    \\
    &\leq \left(\aint_{t_{0}-2^{k}}^{t_0} \|\cL u(t,\cdot)\|_{X}^{p_0} dt \right)^{1/p_0} + \left(\aint_{t_{0}-2^{k}}^{t_0} \|\cL w(t,\cdot)\|_{X}^{p_0} dt \right)^{1/p_0}
\\
&\leq \left(\aint_{t_{0}-2^{k}}^{t_0} \|\cL u(t,\cdot)\|_{X}^{p_0} dt \right)^{1/p_0} + N2^{-k/p_0}\left(\aint_{t_{0}-2}^{t_0} \|f(t,\cdot)\|_{X}^{p_0} dt \right)^{1/p_0}.
\end{align*}
This and \eqref{eq7172038} yield \eqref{eq7020015} with $R=1$.
    The lemma is proved.    
\end{proof}

We introduce some notation which will be used below. Let
\begin{equation*}
    \cC_R(t) := (t-R,t+R), \quad \widehat{\cC_R}(t) := \cC_R(t)\cap (-\infty,T).
\end{equation*}
For $s>0$, $\gamma>0$, $p_0\in(1,\infty)$, and $p_1\in(p_0,\infty]$ from Lemma \ref{lem7021436},
\begin{equation} \label{eq7172355}
    \cA(s) :=\{t\in(-\infty,T): \|\cL u(t,\cdot)\|_{X} > s\},
\end{equation}
\begin{equation*}
    \cA'(s) :=\{t\in(-\infty,T): \|\lambda u(t,\cdot)\|_{X} > s\},
\end{equation*}
\begin{align} \label{eq7172356}
    \cB_{\gamma}(s) := \{t\in (-\infty,T): \gamma^{-1/p_0} \sM_{t,p_0}[f](t) + \gamma^{-1/p_1} \sM_{t,p_0}[\cL u](t) >s\},
\end{align}
and
\begin{align*}
    \cB_{\gamma}'(s) := \{t\in (-\infty,T): \gamma^{-1/p_0} \sM_{t,p_0}[f](t) + \gamma^{-1/p_1} \sM_{t,p_0}[\lambda u](t) >s \},
\end{align*}
where $X:=L_{q,\omega_2}(\bR^d)$. 
Here, for an $X$-valued measurable function $h=h(t,\cdot)$, the expression $(\sM_{t,p}[h](t))^p$ means that the maximal operator $\bM_t$ is applied to the scalar-valued function $t\to \|h(t,\cdot)\|_X^{p}$. More precisely,
\begin{equation} \label{eq7172203}
    \sM_{t,p}[h](t) := \left(\bM_t \|h(\star,\cdot)\|_X^p(t)\right)^{1/p} = \left(\sup_{\cC_R(t')\ni t} \aint_{\cC_R(t')} \|h(s,\cdot)\|_X^{p} 1_{(-\infty,T)}(s) ds\right)^{1/p}.
\end{equation}

\begin{lemma} \label{lem7172353}
    Let $\gamma\in(0,1)$, $R>0$, and the assumptions of Proposition \ref{prop7021543} be satisfied. Assume further that $\cL$ is independent of $t$.
    Let $u\in \bH_{p_0,q,\omega_2,0}^{\alpha,\sigma}(T)$ be a solution to \eqref{eq_main}.
    Then there is a sufficiently large constant $\kappa=\kappa(d,p_0,q,K_0,\nu,\Lambda,\tilde{\alpha},\tilde{\sigma})>1$ such that for any $t_0\in(-\infty,T)$ and $s>0$, if
    \begin{equation} \label{eq7172235}
        |\cC_{\tilde{R}}(t_0)\cap \cA(\kappa s)| \geq \gamma |\cC_{\tilde{R}}(t_0)|,
    \end{equation}
    then 
    \begin{equation*}
        \widehat{\cC_{\tilde{R}}}(t_0) \subset \cB_{\gamma}(s),
    \end{equation*}
    where $\tilde{R}:=R/8$. Moreover, the claim remains valid if $\cA(\kappa s)$ and $\cB_{\gamma}(s)$ are replaced with $\cA'(\kappa s)$ and $\cB_{\gamma}'(s)$, respectively.
\end{lemma}

\begin{proof}
    Due to the similarity, we only consider $\cA(\kappa s)$ and $\cB_{\gamma}(s)$. By dividing the equation \eqref{eq_main} by $s$, it suffices to assume that $s=1$. On the other hand, if $t_0+\tilde{R}\leq0$, then $\cC_{\tilde{R}}(t_0)\cap \cA(\kappa)=\emptyset$. Thus, we only consider the case when $t_0+\tilde{R}>0$.

    We proceed by contradiction. Suppose that there is $t\in \widehat{\cC_{\tilde{R}}}(t_0)$ such that $t\notin \cB_{\gamma}(1)$. Then we have
    \begin{align} \label{eq7172206}
        \gamma^{-1/p_0} \sM_{t,p_0}[f](t) + \gamma^{-1/p_1} \sM_{t,p_0}[\cL u](t) \leq 1.
    \end{align}
    Note that for $t_1:=\min\{t_0+\tilde{R},T\} \geq 0$, we have
    \begin{equation*}
        t\in \widehat{\cC_{\tilde{R}}}(t_0) \subset (t_1-R/4,t_1) \subset (t_1-2^kR,t_1) \quad \forall k\in\bN_0.
    \end{equation*}
Thus, by Lemma \ref{lem7021436} with $t_1$ in place of $t_0$, \eqref{eq7172203}, and \eqref{eq7172206}, one can decompose $u$ as $u=v+w$, where $v,w \in \bH_{p_0,q,\omega_2,0}^{\alpha,\sigma}(t_1)$, and $v$ and $w$ satisfy
    \begin{align*}
        \left(\aint_{t_1-R}^{t_1} \|\cL w(t,\cdot)\|_{X}^{p_0} dt\right)^{1/p_0} \leq N \left(\aint_{t_1-2R}^{t_1} \|f(t,\cdot)\|_{X}^{p_0} dt\right)^{1/p_0} \leq N\gamma^{1/p_0}
    \end{align*}
    and
    \begin{align} \label{eq7172238}
        &\left(\aint_{t_1-R/4}^{t_1} \|\cL v(t,\cdot)\|_{X}^{p_1} dt\right)^{1/p_1} \nonumber
        \\
        &\leq N \sum_{k=0}^{\infty} 2^{-k\alpha} \left(\aint_{t_{1}-2^{k}R}^{t_1} \|\cL u(t,\cdot)\|_{X}^{p_0} dt \right)^{1/p_0} + N \left(\aint_{t_1-2R}^{t_1} \|f(t,\cdot)\|_{X}^{p_0} dt \right)^{1/p_0} \nonumber
        \\
        &\leq N(\gamma^{1/p_1} + \gamma^{1/p_0}) \leq N\gamma^{1/p_1}.
    \end{align}

Let $\kappa>0$ and $C_1\in(0,\kappa)$ be constants, which will be determined later.
    By the Chebyshev inequality, 
  \begin{align*}
      |\cC_{\tilde{R}}(t_0)\cap \cA(\kappa)| &\leq |\{t\in (t_1-R/4,t_1): \|\cL u(t,\cdot)\|_{X} > \kappa\}|
      \\
      &\leq |\{t\in (t_1-R/4,t_1): \|\cL w(t,\cdot)\|_{X}>\kappa-C_1\}|
      \\
      &\quad+ |\{t\in(t_1-R/4,t_1): \|\cL v(t,\cdot)\|_{X}>C_1\}|
      \\
      &\leq \int_{t_1-R/4}^{t_1} \left( \frac{\|\cL w(t,\cdot)\|_{X}^{p_0}}{(\kappa-C_1)^{p_0}} + \frac{\|\cL v(t,\cdot)\|_{X}^{p_1}}{C_1^{p_1}} \right) dt
      \\
      &\leq N R \left( \frac{\gamma}{(\kappa-C_1)^{p_0}} + \frac{\gamma}{C_1^{p_1}} \right).
  \end{align*}
  By taking sufficiently large $C_1$ and $\kappa$ in order, we obtain
    \begin{equation*}
         |\cC_{\tilde{R}}(t_0)\cap \cA(\kappa)| < \gamma |\cC_{\tilde{R}}(t_0)|,
    \end{equation*}
    which contradicts \eqref{eq7172235}. In particular, when $p_1=\infty$, it suffices to first choose $C_1>N$, where $N$ is the last constant in \eqref{eq7172238}, and then proceed with the above argument.
    The lemma is proved.
\end{proof}

Next, we show that if the assertions of Theorem \ref{thm_main} hold in $\bH_{p,q,\omega,0}^{\alpha,\sigma}(T)$, then the corresponding assertions also hold in its dual space $\bH_{p',q',\omega',0}^{\alpha,\sigma}(T)$, where $1/p+1/p'=1$, $1/q+1/q'=1$, and $\omega':=\omega_1^{-1/(p-1)} \omega_2^{-1/(q-1)}$. We emphasize that this duality principle remains valid in the setting of general space-time weights and also for the odd spaces.

\begin{lemma} \label{lem_dual} 
Let the parameters in Theorem \ref{thm_main}, except for $\omega_1$, be fixed.

(i) Suppose that the assertions of Theorem \ref{thm_main} hold for every $\omega_1$ satisfying $[\omega_1]_{A_p(\bR)}\leq K_0$.
Then for every such $\omega_1$, the same assertions also hold with $(p',q',\omega')$ in place of $(p,q,\omega)$, where $1/p+1/p'=1$, $1/q+1/q'=1$, and $\omega'$ is of the form
\begin{equation*}
\omega'(t,x) = \omega_1(t)^{-1/(p-1)} \omega_2(x)^{-1/(q-1)}.
\end{equation*}

(ii) If the assertions of Theorem \ref{thm_main} hold with $\omega_1 \equiv 1$, then they also hold with $(p',q',\omega')$ in place of $(p,q,\omega)$, where $\omega'(t,x):=\omega_2^{-1/(q-1)}(x)$.

(iii) The statements analogous to (i) and (ii) hold in the setting of Theorem \ref{thm_odd}.
\end{lemma}

\begin{proof}
We only prove $(i)$ and $(ii)$, since the corresponding result for the odd space follows in the same way.
We establish $(i)$ and $(ii)$ together. In both cases, we note that $\omega_1^{-1/(p-1)}\in A_{p'}(\bR)$ and $\omega_2^{-1/(q-1)}\in A_{q'}(\bR^d)$ (see \cite[Proposition 7.1.5]{G14}), and thus the weight $\omega'$ is admissible for the dual exponents $(p',q')$.

We first show that it suffices to consider the case when $T<\infty$. 
For the a priori estimate \eqref{est_main}, we clearly have
\begin{equation*}
    \bH_{p',q',\omega',0}^{\alpha,\sigma}(\infty) \subset \bH_{p',q',\omega',0}^{\alpha,\sigma}(T)
\end{equation*}
for any $T<\infty$. Thus, once \eqref{est_main} is established for $T<\infty$ with a constant $N$ independent of $T$, the case $T=\infty$ follows directly by letting $T\to\infty$.

Next, we address the solvability when $T=\infty$ under the assumption that the assertions are proved for every $T<\infty$.
Let $f\in \bL_{p',q',\omega'}(\infty)$. Then one can find a solution $u_n\in \bH_{p',q',\omega',0}^{\alpha,\sigma}(n)$ to \eqref{eq_main} for each $n\in\bN$. Obviously, $\bH_{p',q',\omega',0}^{\alpha,\sigma}(n) \subset \bH_{p',q',\omega',0}^{\alpha,\sigma}(m)$ if $m\leq n$. Hence, by the uniqueness result in $\bH_{p',q',\omega',0}^{\alpha,\sigma}(m)$, we have $u_n = u_m$ when $m\leq n$. Thus, we can define $u$ by setting $u = u_n$ on $(0,n)\times\bR^d$ for each $n\in\bN$. It remains to verify that $u\in \bH_{p',q',\omega',0}^{\alpha,\sigma}(\infty)$. Recall that $\lambda>0$ when considering the solvability for $T=\infty$. Thus, by using \eqref{est_main} with $\cL=-(-\Delta)^{\sigma/2}$, one can deduce that
\begin{align*}
    \|u\|_{\bH_{p',q',\omega'}^{\alpha,\sigma}(n)} = \|u_n\|_{\bH_{p',q',\omega'}^{\alpha,\sigma}(n)} \leq N(\lambda) \|f\|_{\bL_{p',q',\omega'}(n)} \leq N(\lambda) \|f\|_{\bL_{p',q',\omega'}(\infty)} <\infty.
\end{align*}
Letting $n\to\infty$, we have $u\in \bH_{p',q',\omega'}^{\alpha,\sigma}(\infty)$.
To complete this step, it remains to notice that \eqref{eq7291255} is satisfied for $\varphi\in C_c^\infty([0,\infty)\times\bR^d)$, which leads to $u\in \bH_{p',q',\omega',0}^{\alpha,\sigma}(\infty)$. Hence, throughout the rest of the proof, we assume that $T<\infty$.

We use a duality argument.
Since the solvability result holds for $(p,q,\omega)$, after shifting the time interval, for any $\phi\in \bL_{p,q,\omega}(T)$, there exists $w \in \mathbb{H}_{p,q,\tilde{\omega},0}^{\alpha,\sigma}(-T,0)$ satisfying
\begin{equation} \label{eq7292222}
\partial_{t,-T}^\alpha w(t,x) = L^*w(t,x) - \lambda w(t,x) + \phi(-t,x) \quad \text{ in }(-T,0)\times \bR^d,
\end{equation}
where $\tilde{\omega}(t,x):=\omega_1(-t)\omega_2(x)$ and $L^*$ is the operator with the kernel $K(-t,-y)$.
Moreover,
\begin{equation} \label{eq8081618}
\|\cL^* w\|_{\bL_{p,q,\tilde{\omega}}(-T,0)} + \lambda \|w\|_{\bL_{p,q,\tilde{\omega}}(-T,0)} \leq N\|\phi(-\cdot,\cdot)\|_{\bL_{p,q,\tilde{\omega}}(-T,0)} \leq N\|\phi\|_{\bL_{p,q,\omega}(T)},
\end{equation}
where $\cL^*$ is defined in the same way as $L^*$.
Here, we used the fact that either $\omega_1\in A_p(\bR)$ can be chosen arbitrarily or $\omega_1=1$.

We aim to prove the estimate \eqref{est_main} with $(p',q',\omega')$ in place of $(p,q,\omega)$.
Let $u\in \bH_{p',q',\omega',0}^{\alpha,\sigma}(T)$ be a solution to \eqref{eq_main}. By Lemma \ref{lem_dense}, there are sequences of functions $u_n\in C_c^\infty([0,T]\times\bR^d)$ and $w_n\in C_c^\infty([-T,0]\times\bR^d)$ such that $u_n(0,\cdot)=0$ and $w_n(-T,\cdot)=0$, and $u_n\to u$ and $w_n\to w$ in the corresponding spaces.
Letting
\begin{equation*}
    \phi_n(-t,x) := \partial_{t,-T}^\alpha w_n(t,x) - L^*w_n(t,x) + \lambda w_n(t,x)
\end{equation*}
and
\begin{equation*}
    f_n(t,x) := \partial_t^\alpha u_n(t,x) - Lu_n(t,x) + \lambda u_n(t,x),
\end{equation*}
by an integration by parts and the fact that $\bL_{p,q,\omega}(T)$ is the dual space of $\bL_{p',q',\omega'}(T)$,
\begin{align} \label{eq7211503}
&\int_0^T \int_{\bR^d}\phi_n(t,x) \cL u_n(t,x) dxdt \nonumber
\\
&= \int_{-T}^0 \int_{\bR^d} \phi_n(-t,x) \cL u_n(-t,x) dxdt \nonumber
\\
&= \int_{-T}^0\int_{\bR^d} \left(\partial_{t,-T}^\alpha w_n(t,x)-L^*w_n(t,x) +\lambda w_n(t,x)\right) \cL u_n(-t,x) dxdt \nonumber
\\
&= \int_0^T \int_{\bR^d} \left(\partial_t^\alpha u_n(t,x)-Lu_n(t,x) +\lambda u_n(t,x) \right) \cL^*w_n(-t,x) dxdt \nonumber
\\
&=
\int_0^T \int_{\bR^d} f_n(t,x) \cL^*w_n(-t,x) dxdt \nonumber
\\
&\leq
N\|f_n\|_{\bL_{p',q',\omega'}(T)} \|\cL^* w_n\|_{\bL_{p,q,\tilde{\omega}}(-T,0)} \nonumber
\\
&\leq
N\|f_n\|_{\bL_{p',q',\omega'}(T)} \|\phi_n\|_{\bL_{p,q,\omega}(T)}.
\end{align}
For the last inequality, we used \eqref{est_main}.
If we pass to the limit as $n\to\infty$, then by Proposition \ref{lem_conti}, we have
\begin{equation*}
    \left| \int_0^T \int_{\bR^d} \phi(t,x) \cL u(t,x) dxdt \right| \leq N\|f\|_{\bL_{p',q',\omega'}(T)} \|\phi\|_{\bL_{p,q,\omega}(T)}.
\end{equation*}
Since $\phi \in \bL_{p,q,\omega}(T)$ is arbitrary, we obtain the estimate for $\cL u$ in $\bL_{p',q',\omega'}(T)$. By considering $\lambda u$ in place of $\cL u$ and using \eqref{eq_main}, we obtain the a priori estimate \eqref{est_main} with $(p',q',\omega')$ in place of $(p,q,\omega)$.

Lastly, we deal with the solvability.
Let $f\in C_c^\infty((0,T)\times\bR^d)$.
The existence result in \cite[Theorem 2.6]{DL23} yields a solution $u\in\bH_{q',q',1,0}^{\alpha,\sigma}(T)$ to \eqref{eq_main}. 
Let $w\in C_c^\infty([-T,0]\times\bR^d)$ with $w(-T,x)=0$, and denote
\begin{equation} \label{eq8081624}
    \phi(-t,x):= \partial_{t,-T}^\alpha w(t,x) - L^*w(t,x) + \lambda w(t,x).
\end{equation}
Since $w$ is smooth, as in \eqref{eq7211503}, by an integration by parts and H\"older's inequality,
\begin{align*}
    \left|\int_0^T \int_{\bR^d}\phi \cL u dxdt\right| &= \left|\int_0^T \int_{\bR^d} f(t,x) \cL^*w(-t,x) dxdt\right|
\\
&\leq
N\|f\|_{\bL_{p',q',\omega'}(T)} \|\cL^* w\|_{\bL_{p,q,\tilde{\omega}}(-T,0)}.
\end{align*}
Here, we note that $w$ satisfies \eqref{eq8081618}, which leads to
\begin{equation} \label{eq8081627}
    \left|\int_0^T \int_{\bR^d}\phi \cL u dxdt\right| \leq N\|f\|_{\bL_{p',q',\omega'}(T)} \|\phi\|_{\bL_{p,q,\omega}(T)}.
\end{equation}

We now take $\phi \in\bL_{p,q,\omega}(T)$ to be arbitrary, rather than defined by \eqref{eq8081624}. By the assumption, there is a solution $w\in \bH_{p,q,\tilde{\omega},0}^{\alpha,\sigma}(-T,0)$ to \eqref{eq7292222}. By Lemma \ref{lem_dense}, there is a sequence of functions $w_n\in C_c^\infty([-T,0]\times\bR^d)$ such that $w_n\to w$ in $\bH_{p,q,\tilde{\omega},0}^{\alpha,\sigma}(-T,0)$. By Proposition \ref{lem_conti}, letting
\begin{equation*}
    \phi_n(-t,x):=\partial_{t,-T}^\alpha w_n(t,x) - L^*w_n(t,x) + \lambda w_n(t,x),
\end{equation*}
we have $\phi_n\to\phi$ in $\bL_{p,q,\omega}(T)$. Thus, 
by replacing $\phi$ with $\phi_n$ in \eqref{eq8081627} and letting $n\to\infty$, we see that \eqref{eq8081627} holds for general $\phi\in\bL_{p,q,\omega}(T)$. By the arbitrary choice of $\phi$, we deduce that $\cL u\in \bL_{p',q',\omega'}(T)$. Similarly, one can show that $\partial_t^\alpha u,\lambda u\in \bL_{p',q',\omega'}(T)$, which yields $u \in \bH_{p',q',\omega',0}^{\alpha,\sigma}(T)$ if $\lambda>0$.
When $\lambda=0$, by \cite[Lemma 2.9]{KW25},
\begin{equation} \label{eq8110024}
    \|u\|_{\bL_{p',q',\omega'}(T)} \leq NT^\alpha \|\partial_t^\alpha u\|_{\bL_{p',q',\omega'}(T)}<\infty.
\end{equation}
Thus, the assertions of the theorem are proved for $(p',q',\omega')$.
Lastly, we only remark that the odd-space version of \eqref{eq8110024} follows by the extrapolation theorem. Thus, $(iii)$ can be proved by repeating the same argument.
The lemma is proved.
\end{proof}

\begin{proof}[Proof of Proposition \ref{prop7021543}]
As in the proof of Lemma \ref{lem_dual}, it suffices to prove the assertions only when $T<\infty$. We divide the proof into several steps according to the range of $p$.

    \textbf{Step 1.} $p\in(p_0,p_1)$, where $p_1$ is taken from Lemma \ref{lem7021436}.

    In this step, we prove both the a priori estimate \eqref{est_main} and the solvability when $p\in (p_0,p_1)$.

First, we handle \eqref{est_main}. Actually, it suffices to prove the result when $\cL=-(-\Delta)^{\sigma/2}$. Indeed, for general $\cL$, one just needs to use \eqref{eq7131450}.

Suppose that $u\in \bH_{p,q,\omega_2,0}^{\alpha,\sigma}(T)$.
    Since $T<\infty$, we have
    \begin{equation} \label{eq7292027}
        \bL_{p,q,\omega_2}(T) \subset \bL_{p_0,q,\omega_2}(T),
    \end{equation}
    and thus $u\in \bH_{p_0,q,\omega_2,0}^{\alpha,\sigma}(T)$. Since $\cL=-(-\Delta)^{\sigma/2}$ is independent of $t$, we can apply Lemma \ref{lem7172353} and the crawling of ink spots lemma (see \cite[Lemma A.20]{DK19}) to obtain that there is $\kappa(d,p_0,q,K_0,\nu,\Lambda,\tilde{\alpha},\tilde{\sigma}) > 1$ so that
    \begin{equation*}
        |\cA(\kappa s)| \leq N\gamma |\cB_{\gamma}(s)|.
    \end{equation*}
    Here, $\cA(s)$ and $\cB_{\gamma}(s)$ are defined as in \eqref{eq7172355} and \eqref{eq7172356}, respectively. Thus,
    \begin{align*}
        \|\cL u\|_{\bL_{p,q,\omega_2}(T)}^p &= p\kappa^p \int_0^\infty |\cA(\kappa s)| s^{p-1} ds \leq N\gamma\kappa^p \int_0^\infty |\cB_{\gamma}(s)|s^{p-1} ds
        \\
        &\leq N\gamma^{1-p/p_0} \|\sM_{t,p_0}[f]\|_{L_p((-\infty,T))}^p
        \\
        &\quad+ N\gamma^{1-p/p_1} \|\sM_{t,p_0}[\cL u]\|_{L_p((-\infty,T))}^p,
    \end{align*}
    where $X:=L_{q,\omega_2}(\bR^d)$.
    By the Hardy–Littlewood theorem,
    \begin{equation} \label{eq7180051}
        \|\cL u\|_{\bL_{p,q,\omega_2}(T)}^p \leq N\gamma^{1-p/p_0} \|f\|_{\bL_{p,q,\omega_2}(T)}^p + N\gamma^{1-p/p_1} \|\cL u\|_{\bL_{p,q,\omega_2}(T)}^p.
    \end{equation}
    Let us recall that $\|\cL u\|_{\bL_{p,q,\omega_2}(T)}<\infty$ due to Proposition \ref{lem_conti} and $u\in \bH_{p,q,\omega_2,0}^{\alpha,\sigma}(T)$. 
Therefore, by choosing a sufficiently small $\gamma\in(0,1)$ so that $N\gamma^{1-p/p_1}<1/2$, we can absorb the last term of \eqref{eq7180051} to the left. Similarly, by repeating the above argument with $\cA'(\kappa s)$ and $\cB_{\gamma}'(s)$ in place of $\cA(\kappa s)$ and $\cB_{\gamma}(s)$, one can estimate $\lambda u$. The estimate for $\partial_t^\alpha u$ then follows from \eqref{eq_main}.
Therefore, \eqref{est_main} is proved for $p\in (p_0,p_1)$.

Next, we prove the solvability. Due to the method of continuity, it suffices to consider $L=-(-\Delta)^{\sigma/2}$. Let $f\in \bL_{p,q,\omega_2}(T)$. Since the assertions of Theorem \ref{thm_main} hold with $p=p_0$, \eqref{eq7292027} implies that there is a solution $u\in \bH_{p_0,q,\omega_2,0}^{\alpha,\sigma}(T)$ to \eqref{eq_main}.
Let $\eta\in C_c^\infty((0,1))$ be a nonnegative function such that $\int_{\bR} \eta dt=1$. 
For a function $h=h(t,x)$, we denote its mollification in the time variable by
    \begin{equation*}
        h^\varepsilon(t,x) := \int_{\bR} h(s,x) \eta^\varepsilon(t-s)ds, \quad \varepsilon>0,
    \end{equation*}
    where $\eta^\varepsilon(t):=\varepsilon^{-1}\eta(t/\varepsilon)$.
    By mollifying \eqref{eq_main} in the time variable, we can proceed as in the derivation of \eqref{eq8072150} to obtain that
    \begin{equation*}
        \partial_t^\alpha u^\varepsilon = -(-\Delta)^{\sigma/2}u^\varepsilon -\lambda u^\varepsilon + f^\varepsilon \quad \text{in } \bR^d_T.
    \end{equation*}
By H\"older's inequality,
    \begin{equation*}
        \|h^\varepsilon(t,\cdot)\|_X \leq N(\varepsilon)\|h\|_{\bL_{p_0,q,\omega_2}(T)}\|\eta\|_{L_{p_0'}(\bR)} \quad \forall t\in[0,T],
    \end{equation*}
    where $1/p_0+1/p_0'=1$. Since $T<\infty$, one can see that $u^\varepsilon \in \bH_{p,q,\omega_2,0}^{\alpha,\sigma}(T)$ for every $\varepsilon>0$. Thus, we have \eqref{est_main} with $u^\varepsilon$ and $f^\varepsilon$ in place of $u$ and $f$, respectively. As in Remark \ref{remCc}, $f^\varepsilon\to f$ in $\bL_{p,q,\omega_2}(T)$, which yields that
    \begin{equation*}
         \|\partial_t^{\alpha}u\|_{\bL_{p,q,\omega_2}(T)} + \|\cL u\|_{\bL_{p,q,\omega_2}(T)} + \lambda\|u\|_{\bL_{p,q,\omega_2}(T)} <\infty.
    \end{equation*}
If $\lambda>0$, it follows immediately that $u\in \bH_{p,q,\omega_2,0}^{\alpha,\sigma}(T)$. When $\lambda=0$, due to \eqref{eq8110024}, we also have $u\in \bH_{p,q,\omega_2,0}^{\alpha,\sigma}(T)$, and the solvability follows in both cases.

    \textbf{Step 2.} The case $p\in(p_0,\infty)$.
    
By the result of Step \textbf{1}, we see that the assumptions of Lemma \ref{lem7021436} are satisfied with any $p\in(p_0,p_1)$ in place of $p_0$.
Recall that $p_1-p_0>\tilde{\alpha}/2$, where $\tilde{\alpha}$ is independent of $p_0$. We may therefore repeatedly apply the preceding result, taking at each step a newly obtained exponent $p\in(p_0,p_1)$ as the new starting exponent. Since the increment at each step can be chosen uniformly positive, after finitely many iterations we obtain both \eqref{est_main} and the solvability for any $p>p_0$.

\textbf{Step 3.} The case $p\in(1,\infty)$.

Let $q':=q/(q-1)$ and $\omega_2':=\omega_2^{-1/(q-1)}\in A_{q'}(\bR^d)$.
Let $p'\in (1,\frac{p_0}{p_0-1})$ and denote $p:=p'/(p'-1)\in (p_0,\infty)$.
Then by the result of Step \textbf{2}, the assertions of Theorem \ref{thm_main} hold with $(p,q,\omega_2)$. Thus, by Lemma \ref{lem_dual} $(ii)$, the estimate \eqref{est_main} and the solvability are obtained in $\bH_{p',q',\omega_2',0}^{\alpha,\sigma}(T)$. Again by the result of Step \textbf{2}, the assertions hold for every $p'\in(1,\infty)$ with $q'$ and $\omega_2'$ fixed.

We now return to the function space $\bH_{p,q,\omega_2,0}^{\alpha,\sigma}(T)$ originally under consideration. Let $p'=p/(p-1)$. Due to the above argument, the assertions hold in $\bH_{p',q',\omega_2',0}^{\alpha,\sigma}(T)$. Thus, by Lemma \ref{lem_dual} $(ii)$, both the estimate and the solvability are obtained in $\bH_{p,q,\omega_2,0}^{\alpha,\sigma}(T)$.
The proposition is proved.
\end{proof}

\subsection{Extension to product weights in Sobolev spaces} \label{sec_time_gen}

In this section, we prove Theorem \ref{thm_time}, which generalizes Proposition \ref{prop7021543} to function spaces with weights in the time variable.

For a Banach space $X$, $B\subset \bR$, and $\vartheta\in(0,1)$, we denote
\begin{equation*}
    [h]_{C^\vartheta(B;X)}:=\sup_{s,t\in B, s\neq t}\frac{\|h(s,\cdot)-h(t,\cdot)\|_X}{|s-t|^\vartheta}.
\end{equation*}

\begin{lemma} \label{lem7182203}
    Let the parameters in Theorem \ref{thm_main}, except for $p$ and the weight $\omega_1$, be given, and let $\omega_2\in A_q(\bR^d)$ satisfy $[\omega_2]_{A_q(\bR^d)}\leq K_0$. Suppose that, for some $p_0\in(1,\infty)$, the assertions of Theorem \ref{thm_main} hold with $p=p_0$ and with the spatial weight $\omega_2(x)$ in place of the product weight $\omega_1(t)\omega_2(x)$.
    Assume further that $\cL$ is independent of $t$.
    Let $u\in \bH_{p_0,q,\omega_2,0}^{\alpha,\sigma}(T)$ be a solution to \eqref{eq_main}.
    Then for any $t_0\in (0,T)$ and $R>0$, there exist $v,w\in \bH_{p_0,q,\omega_2,0}^{\alpha,\sigma}(t_0)$ such that $u=v+w$,
    \begin{align} \label{eq7021535}
    &\left(\aint_{t_0-R}^{t_0} \|\cL w(t,\cdot)\|_{X}^{p_0} dt\right)^{1/p_0} + \left(\aint_{t_0-R}^{t_0} \|\lambda w(t,\cdot)\|_{X}^{p_0} dt \right)^{1/p_0} \nonumber
       \\
       &\leq N \left(\aint_{t_0-2R}^{t_0} \|f(t,\cdot)\|_{X}^{p_0} dt\right)^{1/p_0},
    \end{align}
    and for any $\vartheta\in(0,\tilde{\alpha})$,
    \begin{align} \label{eq7021537}
         &R^{\vartheta}\left[ \cL v \right]_{C^{\vartheta}((t_0-R/4,t_0);X)} \nonumber
         \\
         &\leq N \sum_{k=0}^{\infty} 2^{-k\alpha} \left(\aint_{t_{0}-2^kR}^{t_0} \|\cL u(t,\cdot)\|_{X}^{p_0} dt \right)^{1/p_0} + N \left(\aint_{t_0-2R}^{t_0} \|f(t,\cdot)\|_{X}^{p_0} dt\right)^{1/p_0}
    \end{align}
    and
    \begin{align*}
        &R^{\vartheta}\left[ \lambda v\right]_{C^{\vartheta}((t_0-R/4,t_0);X)} \nonumber
         \\
         &\leq N \sum_{k=0}^{\infty} 2^{-k\alpha} \left(\aint_{t_{0}-2^kR}^{t_0} \|\lambda u(t,\cdot)\|_{X}^{p_0} dt \right)^{1/p_0} + N \left(\aint_{t_0-2R}^{t_0} \|f(t,\cdot)\|_{X}^{p_0} dt\right)^{1/p_0},
    \end{align*}
    where $X:=L_{q,\omega_2}(\bR^d)$ and $N=N(d,p_0,q,K_0,\Lambda,\nu,\tilde{\alpha},\tilde{\sigma},\vartheta)$ with $\tilde{\alpha}\leq \alpha<1$ and $\tilde{\sigma}$ defined as in \eqref{sigma}.
    Here, $u,v,w$, and $f$ are extended to be zero for $t<0$.
\end{lemma}

\begin{proof}
As in the proof of Lemma \ref{lem7021436}, we may assume that $R=1$.
Let $w\in \bH_{p_0,q,\omega_2,0}^{\alpha,\sigma}(t_0)$ be a solution to
    \begin{equation*}
        \partial_t^\alpha w = Lw -\lambda w + f\eta_0,
    \end{equation*}
 where $\eta_0 \in C^\infty(\bR)$ is a cutoff function such that $\eta_0=1$ for $t\geq t_0-1$, and $\eta_0=0$ for $t\leq t_0-2$. Then obviously, \eqref{eq7021535} follows from \eqref{eq7181804}.

Now we consider $v:=u-w$. Due to the similarity, we only prove \eqref{eq7021537}.
We claim that for $p\in(1,\infty)$,
 \begin{align} \label{eq7182101}
    &\left(\aint_{t_0-1/4}^{t_0} \|\cL v(t,\cdot)\|_{X}^{p} dt\right)^{1/p} \nonumber
        \\
        &\leq N \left(\aint_{t_{0}-1}^{t_0} \|\cL v(t,\cdot)\|_{X}^{p_0} dt \right)^{1/p_0} + N \sum_{k=0}^{\infty} 2^{-k\alpha} \aint_{t_{0}-2^{k+1}}^{t_0} \|\cL v(t,\cdot)\|_{X} dt.
    \end{align}
    Actually, due to H\"older's inequality, it suffices to deal with the case $p\in(p_0,\infty)$.

For $n\in \bN$, we denote $1/p_n:=1/p_0-n\tilde{\alpha}/2$ as long as the right-hand side is positive. We also consider $\eta_n \in C^\infty(\bR)$, which is a cutoff function such that $\eta_n=1$ for $t\geq t_0-2^{-2n}$, and $\eta_n=0$ for $t\leq t_0-2^{1-2n}$. Let $\tilde{v}_n:=\cL (v\eta_n) =\eta_n \cL v$, which satisfies
\begin{equation} \label{eq7182028}
    \partial_t^\alpha \tilde{v}_n = L\tilde{v}_n - \lambda \tilde{v}_n + g_n,
\end{equation}
where $\partial_t^\alpha \tilde{v}_n = \partial_t I_{t_0-2^{1-2n}}^{1-\alpha} \tilde{v}_n$ and
    \begin{equation*}
        g_n(t,x) := \frac{\alpha}{\Gamma(1-\alpha)} \int_{t_0-2^{1-2n}}^{t} (t-s)^{-\alpha-1} (\eta_n(t)-\eta_n(s)) \cL v(s,x) ds.
    \end{equation*}
    Here, as in \eqref{eq7161314}, we mollify \eqref{eq7182028} to justify this procedure.
Let us repeat the proofs of \eqref{eq7181844} and \eqref{eq7181845}. Using $|\eta_n(t)-\eta_n(s)|\leq N(n)\min\{|t-s|,1\}$, for $t\in(t_0-2^{1-2n},t_0)$,
\begin{align*}
    |g_n(t,x)| &\leq \frac{N(n)}{\Gamma(1-\alpha)}\int_{t-2^{1-2n}}^{t} |t-s|^{-\alpha} |\cL v(s,x)| ds 
    \\
    &\quad+ \frac{N(n)}{\Gamma(1-\alpha)} \int_{-\infty}^{t-2^{1-2n}} |t-s|^{-\alpha-1} |\cL v(s,x)|ds
    \\
    &=: I_1^n + I_2^n.
\end{align*}
For any $p\in(1,\infty)$,
\begin{align*}
    \|I_1^n\|_{\bL_{p,q,\omega_2}(t_0-2^{1-2n},t_0)} &\leq \frac{N}{\Gamma(1-\alpha)} \int_0^{2^{1-2n}} s^{-\alpha} \|\cL v(\cdot-s)\|_{\bL_{p,q,\omega_2}(t_0-2^{1-2n},t_0)} ds \nonumber
    \\
    &\leq N\|\cL v\|_{\bL_{p,q,\omega_2}(t_0-2^{2-2n},t_0)}.
\end{align*}
For $I_2^n$, using
\begin{align*}
    |I_2^n(t,x)| &\leq N \int_{t-1}^{t-2^{1-2n}} |t-s|^{-\alpha-1} |\cL v(s,x)|ds 
    \\
    &\quad+ N\sum_{k=1}^{\infty} \int_{t-2^{k}}^{t-2^{k-1}} |t-s|^{-\alpha-1} |\cL v(s,x)|ds
    \\
    &\leq N(n) \int_{t-1}^{t} |\cL v(s,x)|ds + N(n) \sum_{k=1}^{\infty} 2^{-k(\alpha+1)} \int_{t-2^{k}}^{t-2^{k-1}}|\cL v(s,x)|ds
    \\
    &\leq N(n) \sum_{k=0}^{\infty} 2^{-k(\alpha+1)} \int_{t_0-2^{k+1}}^{t_0}|\cL v(s,x)|ds,
\end{align*}
one can deduce that
\begin{align*}
    \|I_2^n\|_{\bL_{p,q,\omega_2}(t_0-2^{1-2n},t_0)} \leq N(n) \sum_{k=0}^{\infty} 2^{-k\alpha} \aint_{t_{0}-2^{k+1}}^{t_0} \|\cL v(t,\cdot)\|_{L_{q,\omega_2}(\bR^d)} dt.
\end{align*}
Hence, for $p\in(1,\infty)$,
\begin{align} \label{eq7182027}
    \|g_n\|_{\bL_{p,q,\omega_2}(t_0-2^{1-2n},t_0)} &\leq N \left(\aint_{t_{0}-2^{2-2n}}^{t_0} \|\cL v(t,\cdot)\|_{X}^{p} dt \right)^{1/p} \nonumber
        \\
        &\quad+ N \sum_{k=0}^{\infty} 2^{-k\alpha} \aint_{t_{0}-2^{k+1}}^{t_0} \|\cL v(t,\cdot)\|_{X} dt.
\end{align}
    
    Suppose that $p_0<2/\tilde{\alpha}$. Since $u\in \bH_{p_0,q,\omega_2,0}^{\alpha,\sigma}(T)$, it follows from \eqref{eq7182027} with $p=p_0$ that $\|g_1\|_{\bL_{p_0,q,\omega_2}(t_0-1/2,t_0)}<\infty$. Thus, applying \eqref{est_main},
        \begin{align*}
        \|\partial_t^{\alpha}\tilde{v}_1\|_{\bL_{p_0,q,\omega_2}(t_0-1/2,t_0)} &\leq N \|g_1\|_{\bL_{p_0,q,\omega_2}(t_0-1/2,t_0)}
        \\
        &\leq N \left(\aint_{t_{0}-1}^{t_0} \|\cL v(t,\cdot)\|_{X}^{p_0} dt \right)^{1/p_0}
        \\
        &\quad+ N \sum_{k=0}^{\infty} 2^{-k\alpha} \aint_{t_{0}-2^{k+1}}^{t_0} \|\cL v(t,\cdot)\|_{X} dt.
    \end{align*}
    As in \eqref{eq7022034}, this and \eqref{eq7151540} yield \eqref{eq7182101} when $p\in[p_0,p_1]$.
    
    Next, we consider the case when $p_1<2/\tilde{\alpha}$. 
    We notice that \eqref{eq7182101} with $p=p_1$ implies that $\tilde{v}_2 \in \bH_{p_1,q,\omega_2,0}^{\alpha,\sigma}(t_0)$ and $\|g_2\|_{\bL_{p_1,q,\omega_2}(t_0-1/8,t_0)}<\infty$. Thus, one can repeat the above process to obtain
    \begin{align*}
        &\left(\aint_{t_0-1/16}^{t_0} \|\cL v(t,\cdot)\|_{X}^{p_2} dt\right)^{1/p_2} \nonumber
        \\
        &\leq N \left(\aint_{t_{0}-1}^{t_0} \|\cL v(t,\cdot)\|_{X}^{p_0} dt \right)^{1/p_0} + N \sum_{k=0}^{\infty} 2^{-k\alpha} \aint_{t_{0}-2^{k+1}}^{t_0} \|\cL v(t,\cdot)\|_{X} dt.
    \end{align*}
    By a covering argument (see e.g. \cite[Remark 2.15]{XX22}) and H\"older's inequality, we have \eqref{eq7182101} when $p\in[p_0,p_2]$. 
    Actually, due to \eqref{eq6171544}, by iterating this procedure finitely many times until $1/p_n\leq\tilde{\alpha}/2$, \eqref{eq7182101} is proved for $p\in[p_0,p_n]$, including the case $n=0$.
   At this stage, we define $p_{n+1}:=\infty$ and repeat the same argument. Then we have \eqref{eq7182101} for any $p\in(p_0,\infty)$.

    Lastly, we prove \eqref{eq7021537}. Take $p\in(1,\infty)$ so that $\vartheta=\tilde{\alpha}-1/p$.
By \eqref{eq7181421} with $\tilde{v}_2 = \eta_2 \cL v$, \eqref{est_main}, and \eqref{eq7182027},
    \begin{align*}
    \left[ \cL v \right]_{C^{\vartheta}((t_0-1/16,t_0);X)} &\leq N \|\partial_t^\alpha\tilde{v}_2\|_{\bL_{p,q,\omega_2}(t_0-1/4,t_0)}
    \\
    &\leq N \left(\aint_{t_{0}-1/4}^{t_0} \|\cL v(t,\cdot)\|_{X}^{p} dt \right)^{1/p}
        \\
        &\quad+ N \sum_{k=0}^{\infty} 2^{-k\alpha} \aint_{t_{0}-2^{k+1}}^{t_0} \|\cL v(t,\cdot)\|_{X} dt.
\end{align*}
By using \eqref{eq7182101}, $v=u-w$, H\"older's inequality, and \eqref{eq7181804},
\begin{align*}
    \left[\cL v \right]_{C^{\vartheta}((t_0-1/16,t_0);X)} &\leq N \sum_{k=0}^{\infty} 2^{-k\alpha} \left(\aint_{t_{0}-2^k}^{t_0} \|\cL u(t,\cdot)\|_{X}^{p_0} dt \right)^{1/p_0} 
    \\
    &\quad+ N \sum_{k=0}^{\infty} 2^{-k\alpha} \left(\aint_{t_{0}-2^k}^{t_0} \|\cL w(t,\cdot)\|_{X}^{p_0} dt \right)^{1/p_0}
    \\
    &\leq N \sum_{k=0}^{\infty} 2^{-k\alpha} \left(\aint_{t_{0}-2^k}^{t_0} \|\cL u(t,\cdot)\|_{X}^{p_0} dt \right)^{1/p_0} 
    \\
    &\quad+ N \sum_{k=0}^{\infty} 2^{-k(\alpha+1/p_0)} \left(\aint_{t_{0}-2}^{t_0} \|f(t,\cdot)\|_{X}^{p_0} dt \right)^{1/p_0}.
\end{align*}
Thus, again by a covering argument, we get \eqref{eq7021537}.
    The lemma is proved.
\end{proof}

\begin{lemma}
Let the assumptions of Lemma \ref{lem7182203} be satisfied.
    Then for any $t_0\in (-\infty,T)$, $R>0$, $\vartheta\in(0,\tilde{\alpha})$, and $\rho\in(0,1/4)$,
    \begin{align} \label{eq7182250}
        &\aint_{t_0-R}^{t_0} \aint_{t_0-R}^{t_0} \left| \|\cL u(t,\cdot)\|_{X} - \|\cL u(s,\cdot)\|_{X} \right| dtds \nonumber
        \\
        &\leq N\rho^\vartheta \sum_{k=0}^{\infty} 2^{-k\alpha} \left(\aint_{t_{0}-2^k\rho^{-1}R}^{t_0} \|\cL u(t,\cdot)\|_{X}^{p_0} dt \right)^{1/p_0} \nonumber
        \\
        &\quad+ N \rho^{-1} \left( \aint_{t_0-2\rho^{-1}R}^{t_0} \|f(t,\cdot)\|_{X}^{p_0} dt \right)^{1/p_0}
    \end{align}
and
    \begin{align} \label{eq7190032}
         &\aint_{t_0-R}^{t_0} \aint_{t_0-R}^{t_0} \left| \|\lambda u(t,\cdot)\|_{X} - \|\lambda u(s,\cdot)\|_{X} \right| dtds \nonumber
         \\
            &\leq N\rho^\vartheta \sum_{k=0}^{\infty} 2^{-k\alpha} \left(\aint_{t_{0}-2^k\rho^{-1}R}^{t_0} \|\lambda u(t,\cdot)\|_{X}^{p_0} dt \right)^{1/p_0} \nonumber
        \\
        &\quad+ N \rho^{-1} \left( \aint_{t_0-2\rho^{-1}R}^{t_0} \|f(t,\cdot)\|_{X}^{p_0} dt \right)^{1/p_0},
    \end{align}
    where $X:=L_{q,\omega_2}(\bR^d)$ and $N=N(d,p_0,q,K_0,\Lambda,\nu,\tilde{\alpha},\tilde{\sigma},\vartheta)$ with $\tilde{\alpha}\leq\alpha<1$ and $\tilde{\sigma}$ defined as in \eqref{sigma}.
    Here, $u$ and $f$ are extended to be zero for $t<0$.
\end{lemma}

\begin{proof}
Due to the similarity, we only prove \eqref{eq7182250}. Since $u$ and $f$ are zero for $t<0$, we may assume that $t_0>0$.
    By Lemma \ref{lem7182203}, there are $v,w\in \bH_{p_0,q,\omega_2,0}^{\alpha,\sigma}(t_0)$ satisfying \eqref{eq7021535} and \eqref{eq7021537} with $\rho^{-1}R$ in place of $R$. In particular, for $w$, by \eqref{eq7021535} and H\"older's inequality,
    \begin{align*}
        &\aint_{t_0-R}^{t_0} \aint_{t_0-R}^{t_0} \left| \|\cL w(t,\cdot)\|_{X} - \|\cL w(s,\cdot)\|_{X} \right| dtds
        \\
        &\leq N \aint_{t_0-R}^{t_0} \|\cL w(t,\cdot)\|_{X} dt
        \\
        &\leq N \rho^{-1} \left( \aint_{t_0-\rho^{-1}R}^{t_0} \|\cL w(t,\cdot)\|_{X}^{p_0} dt \right)^{1/p_0}
        \\
        &\leq N \rho^{-1} \left( \aint_{t_0-2\rho^{-1}R}^{t_0} \|f(t,\cdot)\|_{X}^{p_0} dt \right)^{1/p_0}.
    \end{align*}
    For $v$, by \eqref{eq7021537},
    \begin{align*}
        &\aint_{t_0-R}^{t_0} \aint_{t_0-R}^{t_0} \left| \|\cL v(t,\cdot)\|_{X} - \|\cL v(s,\cdot)\|_{X} \right| dtds
        \\
        &\leq R^\vartheta \left[ \cL v \right]_{C^{\vartheta}((t_0-R,t_0);X)}
        \\
        &\leq R^\vartheta \left[ \cL v \right]_{C^{\vartheta}((t_0-\rho^{-1}R/4,t_0);X)}
        \\
         &\leq N\rho^\vartheta \sum_{k=0}^{\infty} 2^{-k\alpha} \left(\aint_{t_{0}-2^k\rho^{-1}R}^{t_0} \|\cL u(t,\cdot)\|_{X}^{p_0} dt \right)^{1/p_0}
        \\
        &\quad+ N \rho^\vartheta \left(\aint_{t_0-2\rho^{-1}R}^{t_0} \|f(t,\cdot)\|_{X}^{p_0} dt\right)^{1/p_0}.
    \end{align*}
    By the triangle inequality, we get \eqref{eq7182250}. The lemma is proved.
\end{proof}

\begin{proof}[Proof of Theorem \ref{thm_time}]

$(i)$ We aim to prove \eqref{est_main}.
    As in Step \textbf{1} of the proof of Proposition \ref{prop7021543}, we may assume that $\cL$ is independent of $t$. Suppose that $u\in \bH_{p,q,\omega,0}^{\alpha,\sigma}(T)$ is a solution to \eqref{eq_main}. Actually, by Lemma \ref{lem_dense}, we further assume that $u\in C_c^\infty([0,T]\times\bR^d)$ and $u(0,x)=0$.
    
By reverse H\"older's inequality (see e.g. \cite[Corollary 7.2.6]{G14}), there is $\gamma=\gamma(p,K_0)>0$ so that $p-\gamma>1$ and $\omega_1\in A_{p-\gamma}(\bR)$. Denote
\begin{equation*}
    p_0:=\frac{p}{p-\gamma}\in (1,\infty).
\end{equation*}
Then we have
\begin{equation} \label{eq7190028}
    \omega_1 \in A_{p/p_0}(\bR) = A_{p-\gamma}(\bR).
\end{equation}
    
    Let $t_0\in (-\infty,T)$. 
    Due to $u\in C_c^\infty([0,T]\times\bR^d)$ and \eqref{eq7292333}, one can verify that $u\in \bH_{p_0,q,\omega_2,0}^{\alpha,\sigma}(T_0)$ for every $T_0<\infty$. Thus, for any $R>0$, $\rho\in(0,1/4)$, and $t_1\in(-\infty,T)$ such that
    \begin{equation*}
        t_0 \in (t_1-R,t_1) \subset (t_1-\rho^{-1}R,t_1),
    \end{equation*}
    it follows from \eqref{eq7182250} that
    \begin{align*}
        &\aint_{t_1-R}^{t_1} \aint_{t_1-R}^{t_1} \left| \|\cL u(t,\cdot)\|_{X} - \|\cL u(s,\cdot)\|_{X} \right| dtds
        \\
        &\leq N\rho^{\vartheta} \sM_{t,p_0}[\cL u](t_0) + N\rho^{-1} \sM_{t,p_0}[f](t_0),
    \end{align*}
    where $X:=L_{q,\omega_2}(\bR^d)$.
    Taking the supremum over all intervals $(t_1-R,t_1)$ containing $t_0$, we obtain the following estimate for the sharp function of $\|\cL u(t,\cdot)\|_{X}$ with respect to $t$:
    \begin{align*}
        \left(\|\cL u(\star,\cdot)\|_{X}\right)^{\sharp}(t_0) &:= \sup_{(t_1-R,t_1)\ni t_0}\aint_{t_1-R}^{t_1} \aint_{t_1-R}^{t_1} \left| \|\cL u(t,\cdot)\|_{X} - \|\cL u(s,\cdot)\|_{X} \right| dtds
        \\
        &\leq N\rho^{\vartheta}  \sM_{t,p_0}[\cL u](t_0) + N\rho^{-1} \sM_{t,p_0}[f](t_0).
    \end{align*}
    By \eqref{eq7190028}, the weighted sharp function theorem, and the weighted maximal function theorem (see \cite[Theorems 2.2 and 2.3]{DK18}),
\begin{align*}
    \|\cL u\|_{\bL_{p,q,\omega}(T)} \leq N\rho^{\vartheta} \|\cL u\|_{\bL_{p,q,\omega}(T)} + N\rho^{-1}\|f\|_{\bL_{p,q,\omega}(T)}.
\end{align*}
Since $\|\cL u\|_{\bL_{p,q,\omega}(T)}<\infty$ by the assumption that $u\in\bH_{p,q,\omega,0}^{\alpha,\sigma}(T)$, by taking a sufficiently small $\rho\in(0,1/4)$ so that $N\rho^\vartheta<1/2$, we see that
\begin{equation*}
    \|\cL u\|_{\bL_{p,q,\omega}(T)} \leq N \|f\|_{\bL_{p,q,\omega}(T)}.
\end{equation*}
Notice that one can also estimate $\lambda u$ by using \eqref{eq7190032} in place of \eqref{eq7182250}. Now $\partial_t^\alpha u$ can be estimated due to \eqref{eq_main}, and thus we have \eqref{est_main}.

$(ii)$ Next, we deal with the solvability. Due to Remark \ref{remCc} and \eqref{est_main}, it suffices to find a solution to \eqref{eq_main} when $f\in C_c^\infty((0,T)\times \bR^d)$. 
By Proposition \ref{prop7021543}, there is a solution $u\in \bH_{p,q,\omega_2,0}^{\alpha,\sigma}(T)$ to \eqref{eq_main}. Let
\begin{equation*}
    \omega^{(n)}(t,x):= \min\{\omega_1(t),n\} \cdot \omega_2(x).
\end{equation*}
Then by \eqref{eq7312315}, $[\omega^{(n)}]_{p,q}\leq N(K_0,p) [\omega]_{p,q}$.
Actually, $u$ is in the space $\bH_{p,q,\omega^{(n)},0}^{\alpha,\sigma}(T)$. Thus, we can use the a priori estimate \eqref{est_main} to obtain that
\begin{equation*}
    \|\partial_t^{\alpha}u\|_{\bL_{p,q,\omega^{(n)}}(T)} + \|\cL u\|_{\bL_{p,q,\omega^{(n)}}(T)} + \lambda\|u\|_{\bL_{p,q,\omega^{(n)}}(T)} \leq N \|f\|_{\bL_{p,q,\omega^{(n)}}(T)}.
\end{equation*}
By the monotone convergence theorem, letting $n\to\infty$ yields \eqref{est_main}. Due to $\|f\|_{\bL_{p,q,\omega}(T)}<\infty$, we obtain $u\in \bH_{p,q,\omega,0}^{\alpha,\sigma}(T)$ if $\lambda>0$.
When $\lambda=0$, one just needs to use \eqref{eq8110024} to conclude that $u\in \bH_{p,q,\omega,0}^{\alpha,\sigma}(T)$.
The theorem is proved.
\end{proof}

\begin{remark} \label{remsame}
    We show that a solution in Theorem \ref{thm_time} is independent of $p$. More precisely, if $\omega_1\in A_{p_1}(\bR)\cap A_{p_2}(\bR)$, and $u\in\bH_{p_1,q,\omega,0}^{\alpha,\sigma}(T)$ and $v\in\bH_{p_2,q,\omega,0}^{\alpha,\sigma}(T)$ are solutions to \eqref{eq_main} with $f \in \bL_{p_1,q,\omega}(T) \cap \bL_{p_2,q,\omega}(T)$, then $u=v$.
    
    Without loss of generality, we assume that $p_1\leq p_2$. Then for every finite $T_0<T$, $\bH_{p_2,q,\omega,0}^{\alpha,\sigma}(T_0) \subset \bH_{p_1,q,\omega,0}^{\alpha,\sigma}(T_0)$. Thus, both $u$ and $v$ are solutions to \eqref{eq_main} in $\bH_{p_1,q,\omega,0}^{\alpha,\sigma}(T_0)$. By the uniqueness result, $u=v$ in $(0,T_0)\times\bR^d$. Letting $T_0\to T$ leads to the desired result.
\end{remark}

\section{Completion of the theory in the odd Sobolev spaces} \label{sec_space}

In this section, we extend the result in the spatial variables. Since this requires taking the spatial norm outside the temporal norm, we prove Theorem \ref{thm_odd} for the odd Sobolev spaces before proving Theorem \ref{thm_main}.

\subsection{Extension of the spatial exponent}
In this subsection, we prove the following analogue of Proposition \ref{prop7021543} in the odd Sobolev spaces.

\begin{proposition} \label{prop7210028}
    Let the parameters in Theorem \ref{thm_odd}, except for $q$ and the weight $\omega$, be given.  Then the assertions of Theorem \ref{thm_odd} hold for every $q\in(1,\infty)$ and every weight of the form $\omega(t,x)=\omega_1(t)$, where $\omega_1\in A_p(\bR)$ satisfies $[\omega_1]_{A_p(\bR)}\leq K_0$.
\end{proposition}

\begin{remark}
Since
\begin{equation*}
    \bH_{p,p,\omega,0}^{\alpha,\sigma}(T) = \bm{\fH}_{p,p,\omega,0}^{\alpha,\sigma}(T),
\end{equation*}
    the solvability in the odd Sobolev space follows from Remark \ref{rem_pqtime} when $p=q$ and $\omega(t,x)=\omega_1(t)$. Therefore, unlike in Proposition \ref{prop7021543}, there is no need to assume in Proposition \ref{prop7210028} that the assertions of Theorem \ref{thm_odd} hold for some $q=q_0$.
\end{remark}

For a Banach space $Y$ and a weight $\varpi(x):=1/(1+|x|^{d+\sigma})$, we denote
\begin{equation*}
    \|u\|_{L_{1,\varpi}(\bR^d;Y)} := \int_{\bR^d} \frac{\|u(\cdot,x)\|_{Y}}{1+|x|^{d+\sigma}} dx.
\end{equation*}

\begin{lemma} \label{lem7191752}
    Let $p,q\in(1,\infty)$, $T\in(0,\infty]$, $K_0>0$, $[\omega_1]_{A_p(\bR)}\leq K_0$, $R>0$, and $u\in \bm{\fH}_{q,p,\omega_1,0}^{\alpha,\sigma}(T)$. For 
    \begin{equation*}
        r_k:=R(1-2^{-k-1}), \quad k\in \bN_0,
    \end{equation*}
    we take $\zeta_k\in C_c^\infty(B_{r_{k+1}})$ such that $0\leq\zeta_k\leq1$, $\zeta_k=1$ in $B_{r_k}$, and
    \begin{equation} \label{eqzetak}
        |D_x\zeta_k|\leq N\frac{2^k}{R}, \quad |D_x^2\zeta_k|\leq N\frac{2^{2k}}{R^2}.
    \end{equation}
    Then under Assumption \ref{as_upper}, the following estimates hold for
    \begin{equation*}
        I_k:=\|L(\zeta_k u) -\zeta_k Lu\|_{\bm{\fL}_{q,p,\omega_1}(T)} = \|L(\zeta_k u) -\zeta_k Lu\|_{L_q(\bR^d;Y)},
    \end{equation*}
    where $Y:=L_{p,\omega_1}(T)$.

    (i) If $\sigma\in(0,1)$, then
\begin{align} \label{eq7191619}
I_k &\leq N \frac{2^{\sigma k}}{R^\sigma} \|u\|_{L_q(B_R;Y)} + N \frac{2^{(d+\sigma)k}}{R^{d+\sigma-d/q}} (1 + R^{d+\sigma}) \|u\|_{L_{1,\varpi}(\bR^d;Y)},
\end{align}
    
    (ii) If $\sigma\in(1,2)$, then
\begin{align} \label{eq7191715}
I_k &\leq N \frac{2^{(\sigma-1)k}}{R^{\sigma-1}} \|D_xu\|_{L_q(B_{r_{k+3}};Y)} + N \frac{2^{\sigma k}}{R^\sigma} \|u\|_{L_q(B_R;Y)} \nonumber
\\
&\quad+ N \frac{2^{(d+\sigma)k}}{R^{d+\sigma-d/q}} (1 + R^{d+\sigma})  \|u\|_{L_{1,\varpi}(\bR^d;Y)},
\end{align}
    
    (iii) If $\sigma =1$, then for any $\varepsilon \in (0,1)$,
\begin{align} \label{eq7191953}
I_k &\leq N \varepsilon^3 \|D_xu\|_{L_q(B_{r_{k+3}};Y)} + N \varepsilon^{-3} \frac{2^{k}}{R} \|u\|_{L_q(B_R;Y)} \nonumber
\\
&\quad+ N \frac{2^{(d+1)k}}{R^{d+1-d/q}} (1+R^{d+1}) \|u\|_{L_{1,\varpi}(\bR^d;Y)},
\end{align}
where $N=N(d,p,q,K_0,\Lambda,\tilde{\sigma})$ with $\tilde{\sigma}$ defined as in \eqref{sigma}.
\end{lemma}

\begin{proof}
Let $\Tilde{r}_k = r_{k+3} - r_{k+2}=2^{-k-4}R$. Due to \eqref{eq_op},
\begin{align} \label{eq7191056}
&L(\zeta_k u)(t,x) - \zeta_k Lu(t,x) \nonumber
\\
&= \int_{\bR^d} \left((\zeta_k(x+y) - \zeta_k(x))u(t,x+y) - y^{(\sigma)}\cdot \nabla\zeta_k(x) u(t,x)\right) K(t,y) dy.
\end{align}

$(i)$
Let $\sigma\in(0,1)$. In this case, $y^{(\sigma)}=0$ in \eqref{eq7191056}, which leads to
\begin{align} \label{eq7191620}
&|L(\zeta_k u)(t,x) - \zeta_k Lu(t,x)| \nonumber
\\
&\leq N \left(\int_{B_{\tilde{r}_k}} + \int_{B^c_{\tilde{r}_k}} \right) |\zeta_k(x+y) - \zeta_k(x)||u(t,x+y) ||y|^{-d-\sigma} dy=: I^1_{k1} +  I^1_{k2}.
\end{align}

Let us consider $I^1_{k1}$. By \eqref{eqzetak}, for $y\in B_{\tilde{r}_k}$,
\begin{align} \label{eq7191636}
    |\zeta_k(x+y) - \zeta_k(x)|
\leq N \frac{2^k}{R}|y|1_{|x|<r_{{k+2}}}.
\end{align}
Moreover, $I_{k1}^1=0$ for $|x|>r_{k+2}$.
Hence, the Minkowski inequality yields that
\begin{align} \label{eq7191621}
\|I_{k1}^1\|_{L_q(\bR^d;Y)} &\leq N \frac{2^k}{R} \int_{B_{\tilde{r}_k}} \|u(\cdot,\cdot+y)\|_{L_q(B_{r_{k+2}};Y)} |y|^{1-d-\sigma} dy \nonumber
\\
&\leq N\frac{2^k}{R}  \|u\|_{L_q(B_{R};Y)} \int_{B_{\tilde{r}_k}} |y|^{1-d-\sigma} dy \nonumber
\\
&\leq N (1-\sigma)^{-1}\frac{2^k}{R}  \Tilde{r}_k^{1-\sigma} \|u\|_{L_q(B_{R};Y)} \leq N(\sigma_1) \frac{2^{\sigma k}}{R^{\sigma}} \|u\|_{L_q(B_{R};Y)}.
\end{align}

For $I_{k2}^1$,
\begin{align*}
    I^1_{k2} \leq (2-\sigma) \int_{B^c_{\tilde{r}_k}} \bigg(1_{|x+ y|< r_{{k+1}}}+ 1_{|x|< r_{{k+1}}}\bigg) |u(t,x+y)| |y|^{-d-\sigma} dy:= I^1_{k21} + I^1_{k22}.
\end{align*}
For the first component of the last term, by the Minkowski inequality,
\begin{align} \label{eq7191623}
\|I_{k21}^1\|_{L_q(\bR^d;Y)} &\leq N \|u\|_{L_q(B_{R};Y)}\int_{B^c_{\tilde{r}_k}} |y|^{-d-\sigma} dy \nonumber
\\
&\leq N(d,\sigma_0) \tilde{r}_k^{-\sigma} \|u\|_{L_q(B_{R};Y)} \leq N \frac{2^{\sigma k}}{R^{\sigma}} \|u\|_{L_q(B_{R};Y)}. 
\end{align}
Next, we consider $I_{k22}^1$. For $x\in B_{r_{k+1}}$ and $y\in B^c_{\tilde{r}_k}$,
\begin{equation*}
    \frac{1 + |x+y|^{d+ \sigma}}{|y|^{d+ \sigma}} \leq \frac{1+ (r_{k+1}+ \tilde{r}_k)^{d+\sigma}}{\tilde{r}_k^{d+\sigma}} \leq \frac{1 + r_{k+3}^{d+\sigma}}{\tilde{r}_k^{d+\sigma}}.
\end{equation*}
Thus, for $x\in B_{r_{k+1}}$,
\begin{align*} 
    \|I_{k22}^1(\cdot,x)\|_{Y} \leq N\int_{B^c_{\tilde{r}_k}} \|u(\cdot,x+y)\|_{Y} |y|^{-d-\sigma} dy \leq N \frac{1 + r_{k+3}^{d+\sigma}}{\tilde{r}_k^{d+\sigma}} \|u\|_{L_{1,\varpi}(\bR^d;Y)}.
\end{align*}
Hence, by H\"older's inequality and the Minkowski inequality,
\begin{align} \label{eq7191624}
\|I_{k22}^1\|_{L_q(\bR^d;Y)} &\leq N r^{d/q}_{k+1} \|I_{k22}^1\|_{L_{\infty}(B_{r_{k+1}};Y)} \nonumber
\\
&\leq N \frac{2^{(d+\sigma)k}}{R^{d+\sigma-d/q}}(1 + R^{d+\sigma}) \|u\|_{L_{1,\varpi}(\bR^d;Y)}.
\end{align}
Now \eqref{eq7191619} follows from \eqref{eq7191620}, \eqref{eq7191621}, \eqref{eq7191623}, and \eqref{eq7191624}.

$(ii)$
Let $\sigma\in(1,2)$.
In this case, by \eqref{eq7191056},
\begin{align*}
&|L(\zeta_k u)(t,x) - \zeta_k Lu(t,x)|
\\
&\leq N (2-\sigma) \int_{B_{\tilde{r}_k}}|\zeta_k(x+y) - \zeta_k(x)|
|u(t,x+y) - u(t,x)| |y|^{-d-\sigma} dy
\\
&\quad+ N (2-\sigma) \int_{B_{\tilde{r}_k}}|\zeta_k(x+y) - \zeta_k(x)- y\cdot \nabla\zeta_k(x)| |u(t,x)||y|^{-d-\sigma} dy
\\
&\quad+ N (2-\sigma) \int_{B_{\tilde{r}_k}^c}|\zeta_k(x+y) - \zeta_k(x)| |u(t,x+y)| |y|^{-d-\sigma} dy
\\
&\quad+ N (2-\sigma) \int_{B_{\tilde{r}_k}^c}|y \cdot \nabla \zeta_k(x)||u(t,x)| |y|^{-d-\sigma} dy
\\
&=: I^2_{k1} +  I^2_{k2} + I^2_{k3} + I^2_{k4}.
\end{align*}

For $I^2_{k1}$, by \eqref{eq7191636} and the fundamental theorem of calculus,
\begin{align*}
I^2_{k1} &\leq N(2-\sigma) \frac{2^k}{R}1_{|x|<r_{k+2}} \int_{B_{\tilde{r}_k}} |u(t,x+y) - u(t,x)||y|^{1-d-\sigma} dy
\\
&\leq N(2-\sigma) \frac{2^k}{R}1_{|x|<r_{k+2}} \int_{B_{\tilde{r}_k}} \int_0^1  |\nabla u(t,x + sy)| |y|^{2-d-\sigma} dsdy.
\end{align*}
Thus, as in \eqref{eq7191621},
\begin{align} \label{eq7191710}
\|I_{k1}^2\|_{L_q(\bR^d;Y)} &\leq N \frac{2^{k}}{R} \|D_xu\|_{L_q(B_{r_{k+3}};Y)} (2-\sigma) \int_{B_{\tilde{r}_k}} |y|^{2-d-\sigma} dy \nonumber
\\
&= N \frac{2^{(\sigma-1) k}}{R^{\sigma-1}} \|D_xu\|_{L_q(B_{r_{k+3}};Y)}.
\end{align}

Now we handle $I^2_{k2}$. As in \eqref{eq7191636}, due to \eqref{eqzetak} and the mean value theorem, for $y \in B_{\tilde{r}_k}$,
\begin{align*}
&|\zeta_k(x+y) - \zeta_k(x)- y\cdot \nabla\zeta_k(x)| \leq N \frac{2^{2k}}{R^2}|y|^2 1_{|x|<r_{{k+2}}}.
\end{align*}
Thus, we see that
\begin{align*}
I^2_{k2} &\leq N (2-\sigma)\frac{2^{2k}}{R^2} |u(t,x)| 1_{|x|<r_{{k+2}}} \int_{B_{\tilde{r}_k}} |y|^{2-d-\sigma}  dy \leq N \frac{2^{\sigma k}}{R^\sigma} |u(t,x)| 1_{|x|<r_{{k+2}}},
\end{align*}
which yields that
\begin{align} \label{eq7191711}
\|I_{k2}^2\|_{L_q(\bR^d;Y)} &\leq N \frac{2^{\sigma k}}{R^\sigma} \|u\|_{L_q(B_R;Y)}.
\end{align}

Next, we notice that $I^2_{k3} = I^1_{k2}$. Thus,
\begin{align} \label{eq7191712}
    \|I_{k3}^2\|_{L_q(\bR^d;Y)} \leq N \frac{2^{\sigma k}}{R^\sigma} \|u\|_{L_q(B_R;Y)} + N \frac{2^{(d+\sigma)k}}{R^{d+\sigma-d/q}} (1 + R^{d+\sigma}) \|u\|_{L_{1,\varpi}(\bR^d;Y)}.
\end{align}

Lastly,
\begin{align*}
    I^2_{k4} \leq N(2-\sigma) \frac{2^k}{R}|u(t,x)|1_{|x|<r_{k+1}} \int_{B_{\tilde{r}_k}^c} |y|^{1-d-\sigma} dy \leq N\frac{2^{\sigma k}}{R^\sigma}|u(t,x)|1_{|x|<r_{k+1}},
\end{align*}
which yields
\begin{align} \label{eq7191713}
    \|I_{k4}^2\|_{L_q(\bR^d;Y)} \leq N \frac{2^{\sigma k}}{R^\sigma} \|u\|_{L_q(B_R;Y)}.
\end{align}
Combining \eqref{eq7191710}, \eqref{eq7191711}, \eqref{eq7191712}, and \eqref{eq7191713} leads to \eqref{eq7191715}.

$(iii)$ We consider the case when $\sigma=1$. Let $c\in(0,1)$, which will be determined below, and denote $\delta_k:=c\tilde{r}_k$.
Due to \eqref{cancel} and \eqref{eq7191056},
\begin{align*}
&|L(\zeta_k u)(t,x) - \zeta_k Lu(t,x)|
\\
&\leq N \int_{B_{\delta_k}}|\zeta_k(x+y) - \zeta_k(x)|
|u(t,x+y) - u(t,x)| |y|^{-d-1} dy
\\
&\quad+ N \int_{B_{\delta_k}}|\zeta_k(x+y) - \zeta_k(x)- y\cdot \nabla\zeta_k(x)| |u(t,x)||y|^{-d-1} dy
\\
&\quad+ N \int_{\delta_k\leq |y|<\tilde{r}_k}|\zeta_k(x+y) - \zeta_k(x)| |u(t,x+y)| |y|^{-d-1} dy
\\
&\quad+ N \int_{B_{\tilde{r}_k}^c}|\zeta_k(x+y) - \zeta_k(x)| |u(t,x+y)| |y|^{-d-1} dy
\\
&=: I^3_{k1} +  I^3_{k2} + I^3_{k3} + I^3_{k4}.
\end{align*}
As in \eqref{eq7191710}, \eqref{eq7191711}, and \eqref{eq7191712},
\begin{align*}
    \|I_{k1}^3\|_{L_q(\bR^d;Y)} \leq N c\|D_xu\|_{L_q(B_{r_{k+3}};Y)}, \quad \|I_{k2}^3\|_{L_q(\bR^d;Y)} \leq N c\frac{2^k}{R}\|u\|_{L_q(B_{R};Y)}, 
\end{align*}
and
\begin{align*}
    \|I_{k4}^3\|_{L_q(\bR^d;Y)} \leq N \frac{2^{k}}{R} \|u\|_{L_q(B_R;Y)} + N \frac{2^{(d+1)k}}{R^{d+1-d/q}} (1 + R^{d+1}) \|u\|_{L_{1,\varpi}(\bR^d;Y)}.
\end{align*}
Thus, it remains to estimate $I^3_{k3}$. By \eqref{eq7191636},
\begin{align*}
    \|I_{k3}^3\|_{L_q(\bR^d;Y)} &\leq N\frac{2^k}{R} \int_{\delta_k\leq |y|<\tilde{r}_k} \|u(\cdot,\cdot+y)\|_{L_q(B_{r_{k+2}};Y)}|y|^{-d} dy
    \\
    &\leq N \frac{2^k}{R} \|u\|_{L_q(B_{R};Y)} \int_{\delta_k\leq |y|<\tilde{r}_k} |y|^{-d} dy 
    \\
    &\leq N \log(c^{-1}) \frac{2^k}{R} \|u\|_{L_q(B_{R};Y)} \leq N c^{-1} \frac{2^k}{R} \|u\|_{L_q(B_{R};Y)}.
\end{align*}
By choosing $c=\varepsilon^3/N$, we get the desired estimate.
The lemma is proved.
\end{proof}

\begin{lemma}
    Let the parameters in Theorem \ref{thm_odd}, except for $q$ and the weight $\omega_2$, be given. Suppose that, for some $q_0\in(1,\infty)$, the assertions of Theorem \ref{thm_odd} hold with $q=q_0$ and $\omega(t,x)=\omega_1(t)$.
Assume that $u\in \bm{\fH}_{q_0,p,\omega_1,0}^{\alpha,\sigma}(T)$ is a solution to \eqref{eqmainodd}.
Let $R>0$, and $\zeta_0\in C_c^\infty(B_R)$ be a cutoff function such that 
    \begin{equation} \label{cutoff}
        \zeta_0=1 \text{ in } B_{R/2}, \quad |D_x\zeta_0|\leq NR^{-1}, \quad |D_x^2\zeta_0|\leq NR^{-2}.
    \end{equation}
  Then we have
    \begin{align} \label{eq7031124}
        &\|\partial_t^\alpha (u\zeta_0)\|_{\bm{\fL}_{q_0,p,\omega_1}(T)} + \|\cL(u\zeta_0)\|_{\bm{\fL}_{q_0,p,\omega_1}(T)} + \lambda\|u\zeta_0\|_{\bm{\fL}_{q_0,p,\omega_1}(T)} \nonumber
        \\
        &\leq N\|f\|_{L_{q_0}(B_R;Y)} + NR^{-\sigma} \|u\|_{L_{q_0}(B_R;Y)} + N \frac{1+R^{d+\sigma}}{R^{d+\sigma-d/q_0}} \|u\|_{L_{1,\varpi}(\bR^d;Y)},
    \end{align}
    where $Y:=L_{p,\omega_1}(T)$ and $N=N(d,p,q_0,K_0,\Lambda,\nu,\tilde{\alpha},\tilde{\sigma})$ with $\tilde{\sigma}$ defined as in \eqref{sigma}.
\end{lemma}

\begin{proof}
By Proposition \ref{lem_conti}, it suffices to obtain \eqref{eq7031124} with $(-\Delta)^{\sigma/2}$ in place of $\cL$. 
   To prove the result, we repeat the proof of \cite[Lemma 3.2]{DL23}. Let $\zeta_k$ be cutoff functions introduced in Lemma \ref{lem7191752}. Then one can see that $u\zeta_k \in \bm{\fH}_{q_0,p,\omega_1,0}^{\alpha,\sigma}(T)$ satisfies
   \begin{equation*}
       \partial_t^\alpha(u\zeta_k) = L(u\zeta_k)-\lambda u\zeta_k +f\zeta_k +\zeta_kLu-L(u\zeta_k).
   \end{equation*}
   By \eqref{est_odd},
   \begin{align} \label{eq7191810}
       &\|\partial_t^\alpha (u\zeta_k)\|_{\bm{\fL}_{q_0,p,\omega_1}(T)} + \|(-\Delta)^{\sigma/2}(u\zeta_k)\|_{\bm{\fL}_{q_0,p,\omega_1}(T)} + \lambda\|u\zeta_k\|_{\bm{\fL}_{q_0,p,\omega_1}(T)} \nonumber
       \\
       &\leq N \|f\zeta_k\|_{\bm{\fL}_{q_0,p,\omega_1}(T)} + N \|L(\zeta_k u) -\zeta_k Lu\|_{\bm{\fL}_{q_0,p,\omega_1}(T)} \nonumber
       \\
       &\leq N\|f\|_{L_{q_0}(B_R;Y)} + N \|L(\zeta_k u) -\zeta_k Lu\|_{\bm{\fL}_{q_0,p,\omega_1}(T)}.
   \end{align}
   To estimate the last term above, we use Lemma \ref{lem7191752} by considering the three cases.

   First, we handle $\sigma\in(0,1)$. In this case, applying both \eqref{eq7191619} and \eqref{eq7191810} with $k=0$, we obtain \eqref{eq7031124}.

   Next, we deal with $\sigma\in(1,2)$. By \eqref{eq7191715} and \eqref{eq7191810},
   \begin{align} \label{eq7191948}
       &\|\partial_t^\alpha (u\zeta_k)\|_{\bm{\fL}_{q_0,p,\omega_1}(T)} + \|(-\Delta)^{\sigma/2}(u\zeta_k)\|_{\bm{\fL}_{q_0,p,\omega_1}(T)} + \lambda\|u\zeta_k\|_{\bm{\fL}_{q_0,p,\omega_1}(T)} \nonumber
       \\
       &\leq N\|f\|_{L_{q_0}(B_R;Y)} + N \frac{2^{(\sigma-1)k}}{R^{\sigma-1}} \|D_xu\|_{L_{q_0}(B_{r_{k+3}};Y)} + N \frac{2^{\sigma k}}{R^\sigma} \|u\|_{L_{q_0}(B_R;Y)} \nonumber
\\
&\quad+ N \frac{2^{(d+\sigma)k}}{R^{d+\sigma-d/q_0}} (1 + R^{d+\sigma})  \|u\|_{L_{1,\varpi}(\bR^d;Y)}.
   \end{align}
   Since $D_xu=D_x(u\zeta_{k+3})$ in $B_{r_{k+3}}$, by using the interpolation inequality \eqref{eq7191852}, for any $\varepsilon>0$,
   \begin{align*}
       &\|\partial_t^\alpha (u\zeta_k)\|_{\bm{\fL}_{q_0,p,\omega_1}(T)} + \|(-\Delta)^{\sigma/2}(u\zeta_k)\|_{\bm{\fL}_{q_0,p,\omega_1}(T)} + \lambda\|u\zeta_k\|_{\bm{\fL}_{q_0,p,\omega_1}(T)}
       \\
       &\leq N\|f\|_{L_{q_0}(B_R;Y)} + \varepsilon^3 \|(-\Delta)^{\sigma/2}(u\zeta_{k+3})\|_{\bm{\fL}_{q_0,p,\omega_1}(T)} + N\frac{2^{\sigma k}}{R^\sigma} \varepsilon^{3/(1-\sigma)} \|u\|_{L_{q_0}(B_R;Y)} 
       \\
       &\quad + N \frac{2^{(d+\sigma)k}}{R^{d+\sigma-d/q_0}} (1 + R^{d+\sigma}) \|u\|_{L_{1,\varpi}(\bR^d;Y)}.
   \end{align*}
   Let us multiply both sides of this inequality by $\varepsilon^k$, and take the sum over $k$. Then we have
   \begin{align} \label{eq7191949}
       &\sum_{k=0}^\infty \varepsilon^k \left( \|\partial_t^\alpha (u\zeta_k)\|_{\bm{\fL}_{q_0,p,\omega_1}(T)} + \|(-\Delta)^{\sigma/2}(u\zeta_k)\|_{\bm{\fL}_{q_0,p,\omega_1}(T)} + \lambda\|u\zeta_k\|_{\bm{\fL}_{q_0,p,\omega_1}(T)} \right) \nonumber
       \\
       &\leq N\|f\|_{L_{q_0}(B_R;Y)} \sum_{k=0}^\infty \varepsilon^k + \sum_{k=0}^\infty \varepsilon^{k+3} \|(-\Delta)^{\sigma/2}(u\zeta_{k+3})\|_{\bm{\fL}_{q_0,p,\omega_1}(T)} \nonumber
       \\
       &\quad+ N\frac{\varepsilon^{3/(1-\sigma)}}{R^\sigma}  \|u\|_{L_{q_0}(B_R;Y)}  \sum_{k=0}^\infty (\varepsilon2^{\sigma})^k + N \frac{1 + R^{d+\sigma}}{R^{d+\sigma-d/q_0}} \|u\|_{L_{1,\varpi}(\bR^d;Y)}\sum_{k=0}^\infty (\varepsilon2^{d+\sigma})^k.
   \end{align}
   Similarly, by multiplying both sides of \eqref{eq7191948} by $\varepsilon^k$ and summing in $k$, one can deduce that for $\varepsilon2^{d+\sigma}<1$,
   \begin{equation*}
       \sum_{k=0}^\infty \varepsilon^{k} \|(-\Delta)^{\sigma/2}(u\zeta_{k})\|_{\bm{\fL}_{q_0,p,\omega_1}(T)} <\infty.
   \end{equation*}
   Hence, we can absorb the second term on the right-hand side of \eqref{eq7191949} to get \eqref{eq7031124}.

   Lastly, we consider $\sigma=1$. By Proposition \ref{lem_conti}, \eqref{eq7191953}, and \eqref{eq7191810},
   \begin{align*}
    &\|\partial_t^\alpha (u\zeta_k)\|_{\bm{\fL}_{q_0,p,\omega_1}(T)} + \|(-\Delta)^{1/2}(u\zeta_k)\|_{\bm{\fL}_{q_0,p,\omega_1}(T)} + \lambda\|u\zeta_k\|_{\bm{\fL}_{q_0,p,\omega_1}(T)}
       \\
       &\leq N\|f\|_{L_{q_0}(B_R;Y)} + N\varepsilon^3 \|D_xu\|_{L_{q_0}(B_{r_{k+3}};Y)} + N \varepsilon^{-3}\frac{2^{k}}{R} \|u\|_{L_{q_0}(B_R;Y)}
\\
&\quad+ N \frac{2^{(d+1)k}}{R^{d+1-d/q_0}} (1 + R^{d+1})  \|u\|_{L_{1,\varpi}(\bR^d;Y)}
\\
&\leq N\|f\|_{L_{q_0}(B_R;Y)} + \varepsilon^3 \|(-\Delta)^{1/2}(u\zeta_{k+3})\|_{\bm{\fL}_{q_0,p,\omega_1}(T)} + N \varepsilon^{-3}\frac{2^{k}}{R} \|u\|_{L_{q_0}(B_R;Y)}
\\
&\quad+ N \frac{2^{(d+1)k}}{R^{d+1-d/q_0}} (1 + R^{d+1})  \|u\|_{L_{1,\varpi}(\bR^d;Y)}.
   \end{align*}
  Then \eqref{eq7031124} can be obtained by repeating the case when $\sigma\in(1,2)$. The lemma is proved.
\end{proof}

\begin{lemma} \label{lem7202344}
    Let the parameters in Theorem \ref{thm_odd}, except for $q$ and the weight $\omega_2$, be given.
    Assume further that $T<\infty$ and that $\cL$ is independent of $t$.
    Suppose that, for some $q_0\in(1,\infty)$, the assertions of Theorem \ref{thm_odd} hold with $q=q_0$ and $\omega(t,x)=\omega_1(t)$.
Assume that $u\in \bm{\fH}_{q_0,p,\omega_1,0}^{\alpha,\sigma}(T)$ is a solution to \eqref{eqmainodd}.
Then there exists $q_1=q_1(\tilde{\sigma},q_0)\in (q_0,\infty]$ such that
\begin{equation} \label{eq7202135}
    q_1-q_0 > \delta(\tilde{\sigma},d)>0,
\end{equation}
and the following holds.
For any $R>0$ and $x_0\in \bR^d$, there exist $v,w\in \bm{\fH}_{q_0,p,\omega_1,0}^{\alpha,\sigma}(T)$ such that $u=v+w$,
    \begin{align} \label{eq7031544}
    &\left(\aint_{B_{R/2}(x_0)} \|\cL w(\cdot,x)\|_{Y}^{q_0} dx\right)^{1/q_0} + \left(\aint_{B_{R/2}(x_0)} \|\lambda w(\cdot,x)\|_{Y}^{q_0} dx\right)^{1/q_0} \nonumber
       \\
       &\leq N \left(\aint_{B_{R}(x_0)} \|f(\cdot,x)\|_{Y}^{q_0} dx\right)^{1/q_0},
    \end{align}
    \begin{align} \label{eq7031622}
        \left(\aint_{B_{R/4}(x_0)} \|\cL v(\cdot,x)\|_{Y}^{q_1} dx\right)^{1/q_1} &\leq N \sum_{k=0}^\infty 2^{-k\sigma} \left(\aint_{B_{2^{k}R}(x_0)} \|\cL u(\cdot,x)\|_Y^{q_0}dx\right)^{1/q_0} \nonumber
        \\
        &\quad+ N \sum_{k=0}^\infty 2^{-k\sigma} \left(\aint_{B_{2^{k}R}(x_0)} \|f(\cdot,x)\|_Y^{q_0}dx\right)^{1/q_0},
    \end{align}
    and
    \begin{align*}
        \left(\aint_{B_{R/4}(x_0)} \|\lambda v(\cdot,x)\|_{Y}^{q_1} dx\right)^{1/q_1} &\leq N \sum_{k=0}^\infty 2^{-k\sigma} \left(\aint_{B_{2^{k}R}(x_0)} \|\lambda u(\cdot,x)\|_Y^{q_0}dx\right)^{1/q_0} 
        \\
        &\quad+ N \sum_{k=0}^\infty 2^{-k\sigma} \left(\aint_{B_{2^{k}R}(x_0)} \|f(\cdot,x)\|_Y^{q_0}dx\right)^{1/q_0},
    \end{align*}
where $Y:=L_{p,\omega_1}(T)$ and $N=N(d,p,q_0,K_0,\Lambda,\nu,\tilde{\alpha},\tilde{\sigma})$ with $\tilde{\sigma}$ defined as in \eqref{sigma}.
\end{lemma}

\begin{proof}
As in the proof of Lemma \ref{lem7021436}, by scaling and using the fact that $N$ is independent of $T$, we may assume that $R=1$. Moreover, by shifting the coordinates, we also assume that $x_0=0$. Due to the similarity, we only prove \eqref{eq7031544} and \eqref{eq7031622}.

Let $\zeta\in C_c^\infty(B_1)$ be a cutoff function such that $\zeta=1$ in $B_{3/4}$. Since Theorem \ref{thm_odd} holds in $\bm{\fH}_{q_0,p,\omega_1,0}^{\alpha,\sigma}(T)$, there is a solution $w\in \bm{\fH}_{q_0,p,\omega_1,0}^{\alpha,\sigma}(T)$ to
\begin{equation*}
    \partial_t^\alpha w= Lw -\lambda w + \zeta f.
\end{equation*}
Moreover, we have
    \begin{align} \label{eq7202139}
        \|\partial_t^{\alpha}w\|_{\bm{\fL}_{q_0,p,\omega_1}(T)} + \|\cL w\|_{\bm{\fL}_{q_0,p,\omega_1}(T)} + \lambda\|w\|_{\bm{\fL}_{q_0,p,\omega_1}(T)} \leq N \|f\zeta\|_{\bm{\fL}_{q_0,p,\omega_1}(T)},
    \end{align}
    which easily yields \eqref{eq7031544}.
Here, due to $T<\infty$, we remark that the existence is guaranteed regardless of whether $\lambda=0$.

    Let $v:=u-w\in \bm{\fH}_{q_0,p,\omega_1,0}^{\alpha,\sigma}(T)$.
    For $\varepsilon>0$ and any function $h$ on $(0,T)\times \bR^d$, $h^\varepsilon :=h*\varphi^\varepsilon$, where $\int_{\bR^d}\varphi dx=1$ and $\varphi^\varepsilon(x):=\varepsilon^{-d}\varphi(x/\varepsilon)$.
    Then it can be shown that $\tilde{v}^\varepsilon:=\cL v^\varepsilon = (\cL v)^{\varepsilon} \in \bm{\fH}_{q_0,p,\omega_1,0}^{\alpha,\sigma}(T)$ satisfies
    \begin{equation*}
            \partial_t^\alpha\tilde{v}^\varepsilon= L\tilde{v}^\varepsilon -\lambda \tilde{v}^\varepsilon + \cL([(1-\zeta)f]^\varepsilon).
    \end{equation*}
    By \eqref{eq7031124} with $R=1/2$, for a cutoff function $\zeta_0\in C_c^\infty(B_{1/2})$ satisfying \eqref{cutoff},
    \begin{align*}
        \|(-\Delta)^{\sigma/2}\zeta_0(\cL v)^\varepsilon\|_{\bm{\fL}_{q_0,p,\omega_1}(T)} &\leq N\|\cL([(1-\zeta)f]^\varepsilon)\|_{L_{q_0}(B_{1/2};Y)} \nonumber
        \\
        &\quad+ N \|\cL v^\varepsilon\|_{L_{q_0}(B_{1/2};Y)} + N \|\cL v^\varepsilon\|_{L_{1,\varpi}(\bR^d;Y)}.
    \end{align*}
    Letting $\varepsilon\to0$,
    \begin{align} \label{eq7202131}
        \|(-\Delta)^{\sigma/2}(\zeta_0\cL v)\|_{\bm{\fL}_{q_0,p,\omega_1}(T)} &\leq N\|\cL((1-\zeta)f)\|_{L_{q_0}(B_{1/2};Y)} \nonumber
        \\
        &\quad+ N \|\cL v\|_{L_{q_0}(B_{1/2};Y)} + N \|\cL v\|_{L_{1,\varpi}(\bR^d;Y)}.
    \end{align}
    Let us estimate $\cL((1-\zeta)f)$. Since $1-\zeta=0$ in $B_{3/4}$, by \eqref{eq_op} and \eqref{upper}, for $x\in B_{1/2}$
    \begin{align*}
        |\cL((1-\zeta)f)(t,x)| &\leq N \int_{|x+y|>3/4} |(1-\zeta)(x+y)||f(t,x+y)| |y|^{-d-\sigma} dy \nonumber
        \\
        &\leq N \int_{|y|>1/4} |f(t,x+y)| |y|^{-d-\sigma} dy \nonumber
        \\
        &= N \sum_{k=0}^\infty \int_{B_{2^{k-1}}\setminus B_{2^{k-2}}} |f(t,x+y)| |y|^{-d-\sigma} dy \nonumber
        \\
        &\leq N \sum_{k=0}^\infty 2^{-kd-k\sigma} \int_{B_{2^{k}}} |f(t,y)|dy \nonumber
        \\
        &\leq N \sum_{k=0}^\infty 2^{-k\sigma} \aint_{B_{2^{k}}} |f(t,y)|dy.
    \end{align*}
    By the Minkowski inequality and H\"older's inequality,
    \begin{align} \label{eq7202132}
        \|\cL((1-\zeta)f)\|_{L_{q_0}(B_{1/2};Y)} \leq N \sum_{k=0}^\infty 2^{-k\sigma} \left(\aint_{B_{2^{k}}} \|f(\cdot,x)\|_Y^{q_0}dx\right)^{1/q_0}.
    \end{align}
    For $\|\cL v\|_{L_{1,\varpi}(\bR^d;Y)}$, by H\"older's inequality,
    \begin{align} \label{eq7202133}
        \|\cL v\|_{L_{1,\varpi}(\bR^d;Y)} &\leq \int_{B_1} \|\cL v(\cdot,x)\|_{Y} dx + \sum_{k=0}^\infty 2^{-kd-k\sigma} \int_{B_{2^{k+1}}\setminus B_{2^k}} \|\cL v(\cdot,x)\|_{Y} dx \nonumber
        \\
        &\leq N \sum_{k=0}^\infty 2^{-k\sigma} \aint_{B_{2^{k}}} \|\cL v(\cdot,x)\|_{Y} dx \nonumber
        \\
        &\leq N \sum_{k=0}^\infty 2^{-k\sigma} \left(\aint_{B_{2^{k+1}}} \|\cL v(\cdot,x)\|_Y^{q_0}dx\right)^{1/q_0}.
    \end{align}

Let $q_1\in (q_0,\infty]$ so that
\begin{equation*}
    1/q_1=1/q_0-\sigma_0/2d \quad \text{ if } q_0\leq 2d/\sigma_0,
\end{equation*}
and $q_1=\infty$ if $q_0 > 2d/\sigma_0$. Then we see that when $q_0< 2d/\sigma_0$,
\begin{equation*}
    q_1-q_0 = \frac{\sigma_0q_0^2}{2d-\sigma_0q_0} > \frac{\sigma_0}{2d-\sigma_0}>0,
\end{equation*}
which implies \eqref{eq7202135}. Here, we put $\sigma_0:=1$ if $\sigma=1$.

    Now we apply \eqref{eq7201656} with $\zeta_0\cL v$ and $(q_1,q_0)$ in place of $u$ and $(\tilde{q},q)$, respectively. Then we have
    \begin{align*}
        \|\zeta_0\cL v\|_{\bm{\fL}_{q_1,p,\omega_1}(T)} \leq N \|\zeta_0\cL v\|_{\bm{\fL}_{q_0,p,\omega_1}(T)} + N \|(-\Delta)^{\sigma/2}(\zeta_0\cL v)\|_{\bm{\fL}_{q_0,p,\omega_1}(T)}.
    \end{align*}
    This together with \eqref{eq7202131}, \eqref{eq7202132}, and \eqref{eq7202133} leads to
    \begin{align}
        &\left(\aint_{B_{1/4}} \|\cL v(\cdot,x)\|_{Y}^{q_1} dx\right)^{1/q_1} \nonumber
        \\
        &\leq N \sum_{k=0}^\infty 2^{-k\sigma} \left(\aint_{B_{2^{k}}} \|\cL v(\cdot,x)\|_Y^{q_0}dx\right)^{1/q_0} + N \sum_{k=0}^\infty 2^{-k\sigma} \left(\aint_{B_{2^{k}}} \|f(\cdot,x)\|_Y^{q_0}dx\right)^{1/q_0} \nonumber
        \\
        &\leq N \sum_{k=0}^\infty 2^{-k\sigma} \left(\aint_{B_{2^{k}}} \|\cL u(\cdot,x)\|_Y^{q_0}dx\right)^{1/q_0} + N \sum_{k=0}^\infty 2^{-k\sigma} \left(\aint_{B_{2^{k}}} \|\cL w(\cdot,x)\|_Y^{q_0}dx\right)^{1/q_0} \nonumber
             \\
        &\quad+ N \sum_{k=0}^\infty 2^{-k\sigma} \left(\aint_{B_{2^{k}}} \|f(\cdot,x)\|_Y^{q_0}dx\right)^{1/q_0}. \label{eq7202143}
    \end{align}
    Here, for the last inequality, we used $v=u-w$. Due to \eqref{eq7202139}, for each $k\in \bN_0$,
    \begin{align*}
        \left(\aint_{B_{2^{k}}} \|\cL w(\cdot,x)\|_Y^{q_0}dx\right)^{1/q_0} \leq N 2^{-kd/q_0} \left(\aint_{B_{1}} \|f(\cdot,x)\|_Y^{q_0}dx\right)^{1/q_0}.
        \end{align*}
        This and \eqref{eq7202143} yield \eqref{eq7031622}.
The lemma is proved.
\end{proof}

We introduce the notation used below.
Let $Y:=L_{p,\omega_1}(T)$ and
\begin{equation*}
    \sM_{x,q}[h](x) := \left(\bM_x \|h(\cdot,\star)\|_Y^q(x)\right)^{1/q} = \left(\sup_{B_R(y)\ni x} \aint_{B_R(y)} \|h(\cdot,z)\|_Y^{q} dz\right)^{1/q},
\end{equation*}
where $\bM_x$ is the uncentered Hardy-Littlewood maximal operator in the spatial variables defined by
\begin{equation*}
    \bM_x g(x):=\sup_{B_R(y)\ni x}\aint_{B_R(y)}|g(z)| dz.
\end{equation*}
For $s>0$ and $\gamma\in(0,1)$, let
\begin{equation*}
    \widetilde{\cA}(s):=\big\{x\in\bR^d: \|\cL u(\cdot,x)\|_Y>s\big\}, \qquad
    \widetilde{\cA}'(s):=\big\{x\in\bR^d:\|\lambda u(\cdot,x)\|_Y>s\big\},
\end{equation*}
and
\begin{align*}
    \widetilde{\cB}_\gamma(s) := \big\{x\in\bR^d:
    \gamma^{-1/q_0}\sM_{x,q_0}[f](x) + \gamma^{-1/q_1}\sM_{x,q_0}[\cL u](x) > s\big\},
\end{align*}
where $q_1$ is taken from Lemma \ref{lem7202344}. We also define
\begin{align*}
    \widetilde{\cB}'_\gamma(s) := \big\{x\in\bR^d:
    \gamma^{-1/q_0} \sM_{x,q_0}[f](x) + \gamma^{-1/q_1} \sM_{x,q_0}[\lambda u](x) > s\big\}.
\end{align*}

\begin{lemma} \label{lem7301237}
Let $\gamma\in(0,1)$ and the assumptions of Lemma \ref{lem7202344} hold. Then there exists a sufficiently large constant $\kappa=\kappa(d,p,q_0,K_0,\nu,\Lambda,\tilde\alpha,\tilde\sigma)>1$
such that the following assertion holds. Let $x_0\in\bR^d$, $s>0$, $R>0$, and $\tilde R:=R/4$. If
\begin{equation} \label{eq7210027}
    |B_{\tilde R}(x_0)\cap\widetilde{\cA}(\kappa s)| \geq \gamma |B_{\tilde R}(x_0)|,
\end{equation}
then
\begin{equation*}
    B_{\widetilde R}(x_0)\subset \widetilde{\cB}_\gamma(s).
\end{equation*}
The same assertion holds with $(\widetilde{\cA},\widetilde{\cB}_\gamma)$ replaced by
$(\widetilde{\cA}',\widetilde{\cB}'_\gamma)$.
\end{lemma}

\begin{proof}
As in the proof of Lemma \ref{lem7172353}, we only prove the assertion for $\widetilde{\cA}$ and $\widetilde{\cB}_\gamma$ with $s=1$.

Suppose, to the contrary, that there exists
$x\in B_{\widetilde R}(x_0)$ such that $x\notin\widetilde{\cB}_\gamma(1)$. Then
\begin{equation*} 
    \gamma^{-1/q_0}\sM_{x,q_0}[f](x) + \gamma^{-1/q_1} \sM_{x,q_0}[\cL u](x) \leq 1.
\end{equation*}
Since $x\in B_{\widetilde R}(x_0)\subset B_{2^kR}(x_0)$ for every
$k\in\bN_0$,
\begin{equation*}
    \left(\aint_{B_{2^kR}(x_0)} \|f(\cdot,y)\|_{Y}^{q_0} dy\right)^{1/q_0} \leq \gamma^{1/q_0}, \quad \left(\aint_{B_{2^kR}(x_0)} \|\cL u(\cdot,y)\|_{Y}^{q_0} dy\right)^{1/q_0} \leq \gamma^{1/q_1}.
\end{equation*}
Thus, similar to \eqref{eq7172238}, by Lemma \ref{lem7202344}, there are $v,w\in\bm{\fH}_{q_0,p,\omega_1,0}^{\alpha,\sigma}(T)$ such that $u=v+w$, and
\begin{align*}
    &\left(\aint_{B_{R/2}(x_0)} \|\cL w(\cdot,y)\|_{Y}^{q_0} dy \right)^{1/q_0} \leq N\gamma^{1/q_0},
    \\
    &\left(\aint_{B_{R/4}(x_0)} \|\cL v(\cdot,y)\|_{Y}^{q_1} dy \right)^{1/q_1} \leq N\gamma^{1/q_1}.
\end{align*}

Let $\kappa>0$ and $C_1\in(0,\kappa)$ be constants, which will be determined later. Since $u=v+w$, the Chebyshev inequality yields that
\begin{align*}
    |B_{\widetilde R}(x_0)\cap\widetilde{\cA}(\kappa)|
    &\leq |\{y\in B_{R/2}(x_0): \|\cL w(\cdot,y)\|_{Y}>\kappa-C_1\}|
    \\
    &\quad+|\{y\in B_{R/4}(x_0): \|\cL v(\cdot,y)\|_{Y} >C_1\}|
    \\
    &\leq NR^d\left(
       \frac{\gamma}{(\kappa-C_1)^{q_0}}
       +\frac{\gamma}{C_1^{q_1}}\right).
\end{align*}
We first choose $C_1$ sufficiently large and then choose $\kappa$ sufficiently
large, so that
\begin{equation*}
    |B_{\tilde{R}}(x_0)\cap\widetilde{\cA}(\kappa)|
    <\gamma|B_{\tilde{R}}(x_0)|,
\end{equation*}
which contradicts \eqref{eq7210027}. 
The lemma is proved.
\end{proof}

\begin{proof}[Proof of Proposition \ref{prop7210028}]

\textbf{1.} The case $q\geq p$.

We first prove the a priori estimate \eqref{est_odd} when $u\in \bm{\fH}_{q,p,\omega_1,0}^{\alpha,\sigma}(T)$. Since the constant $N$ in \eqref{est_odd} is independent of $T$, we only handle the case $T<\infty$. By Lemma \ref{lem_dense}, it suffices to consider $u\in C_c^\infty([0,T]\times \bR^d)$ with $u(0,x)=0$.
 Due to Proposition \ref{lem_conti}, it suffices to prove \eqref{est_odd} when $\cL=-(-\Delta)^{\sigma/2}$. 
 
 Let $q_0:=p$, and $q_1\in(q_0,\infty]$ be as in Lemma \ref{lem7202344}.
 It follows from \eqref{eq7291548}, $T<\infty$, and $u\in C_c^\infty([0,T]\times \bR^d)$ that $u\in \bm{\fH}_{p,p,\omega_1,0}^{\alpha,\sigma}(T)$, which allows us to apply Lemma \ref{lem7301237} to $u$. Thus, by repeating the proof of Step \textbf{1} in Proposition \ref{prop7021543}, one can derive the estimate \eqref{est_odd} when $q\in (q_0,q_1)$.
Then the general case $q\geq q_0$ follows by the same argument as in Step \textbf{2} of the proof of Proposition \ref{prop7021543}.

Next, we prove the solvability. Due to Remark \ref{remCc} and \eqref{est_odd}, it suffices to find a solution when $f\in C_c^\infty((0,T)\times\bR^d)$. By Remarks \ref{rem_pqtime} and \ref{remsame}, there exists a solution $u$ to \eqref{eqmainodd} such that, for every $p\in(1,\infty)$ and every $\omega_1\in A_p(\bR)$ satisfying $[\omega_1]_{A_p(\bR)}\leq K_0$, we have $u\in\bH_{p,q,\omega_1,0}^{\alpha,\sigma}(T)$. By the Minkowski inequality, we get
\begin{equation*}
    \bL_{p,q,\omega_1}(T) \subset \bm{\fL}_{q,p,\omega_1}(T),
\end{equation*}
which yields $u\in \bm{\fH}_{q,p,\omega_1,0}^{\alpha,\sigma}(T)$.

\textbf{2.} Next, we deal with $q\leq p$.

To prove \eqref{est_odd}, as in Step \textbf{3} of the proof of Proposition \ref{prop7021543}, we use Lemma \ref{lem_dual}. 
Let $p':=p/(p-1)$, $q':=q/(q-1)$, and $\omega_1':=\omega_1^{-1/(p-1)}\in A_{p'}(\bR)$. Since $q'\geq p'$, by the result of Step \textbf{1}, the assertions of Theorem \ref{thm_odd} hold in $\bm{\fH}_{q',p',\omega_1',0}^{\alpha,\sigma}(T)$. By Lemma \ref{lem_dual}, both the estimate \eqref{est_odd} and the solvability are obtained in $\bm{\fH}_{q,p,\omega_1,0}^{\alpha,\sigma}(T)$.
The proposition is proved.
\end{proof}

\subsection{Proof of Theorem \ref{thm_odd}}

In this subsection, we aim to prove Theorem \ref{thm_odd}.

For a Banach space $Y$, a ball $B\subset\bR^d$, and $\vartheta\in(0,1)$, we use the notation
\begin{equation*}
    [h]_{C^\vartheta(B;Y)}:=\sup_{x,y\in B, x\neq y}\frac{\|h(\cdot,x)-h(\cdot,y)\|_Y}{|x-y|^\vartheta}.
\end{equation*}

We recall that $\tilde{\sigma}$ in Theorem \ref{thm_odd} is defined as \eqref{sigma}.
In Lemmas \ref{lem7211430} and \ref{lem8101843}, we use the convention $\sigma_0:=1$ when $\sigma=1$.

\begin{lemma} \label{lem7211430}
    Let the parameters in Theorem \ref{thm_odd}, except for $q$ and the weight $\omega_2$, be given, and fix $\omega_1\in A_p(\bR)$ satisfying $[\omega_1]_{A_p(\bR)}\leq K_0$.
    Assume further that $T<\infty$ and that $\cL$ is independent of $t$.
    Let $q_0\in(1,\infty)$, and $u\in \bm{\fH}_{q_0,p,\omega_1,0}^{\alpha,\sigma}(T)$ be a solution to \eqref{eqmainodd}. 
    Then for any $x_0\in\bR^d$ and $R>0$, there exist $v,w\in \bm{\fH}_{q_0,p,\omega_1,0}^{\alpha,\sigma}(T)$ such that $u=v+w$,
    \begin{align} \label{eq7211431}
        &\left(\aint_{B_{R}(x_0)} \|\cL w(\cdot,x)\|_Y^{q_0}dx\right)^{1/q_0}
        +\left(\aint_{B_{R}(x_0)} \|\lambda w(\cdot,x)\|_Y^{q_0}dx\right)^{1/q_0} \nonumber
        \\
        &\leq N\left(\aint_{B_R(x_0)} \|f(\cdot,x)\|_Y^{q_0}dx\right)^{1/q_0},
    \end{align}
    and, for any $\vartheta\in(0,\sigma_0/2)$,
    \begin{align} \label{eq7211432}
        &R^\vartheta[\cL v]_{C^\vartheta(B_{R/4}(x_0);Y)} \nonumber
        \\
        &\leq N\sum_{k=0}^{\infty}2^{-k\sigma}
        \left(\aint_{B_{2^kR}(x_0)}\|\cL u(\cdot,x)\|_Y^{q_0}dx\right)^{1/q_0} \nonumber
        \\
        &\quad+N\sum_{k=0}^{\infty}2^{-k\sigma}
        \left(\aint_{B_{2^kR}(x_0)}\|f(\cdot,x)\|_Y^{q_0}dx\right)^{1/q_0}
    \end{align}
    and
    \begin{align*}
        &R^\vartheta[\lambda v]_{C^\vartheta(B_{R/4}(x_0);Y)} 
        \\
        &\leq N\sum_{k=0}^{\infty}2^{-k\sigma}
        \left(\aint_{B_{2^kR}(x_0)}\|\lambda u(\cdot,x)\|_Y^{q_0}dx\right)^{1/q_0} 
        \\
        &\quad+N\sum_{k=0}^{\infty}2^{-k\sigma}
        \left(\aint_{B_{2^kR}(x_0)}\|f(\cdot,x)\|_Y^{q_0}dx\right)^{1/q_0},
    \end{align*}
    where $Y:=L_{p,\omega_1}(T)$ and $N=N(d,p,q_0,K_0,\Lambda,\nu,\tilde{\alpha},\tilde{\sigma},\vartheta)$.
\end{lemma}

\begin{proof}
    By a translation and a scaling, we may assume that $x_0=0$ and $R=1$. Let $\zeta\in C_c^\infty(B_1)$ be a cutoff function such that $\zeta=1$ in $B_{3/4}$. By Proposition \ref{prop7210028}, there is $w\in \bm{\fH}_{q_0,p,\omega_1,0}^{\alpha,\sigma}(T)$ satisfying
    \begin{equation*}
        \partial_t^\alpha w=Lw-\lambda w+\zeta f.
    \end{equation*}
    Moreover,
    \begin{align} \label{eq7211433}
        \|\partial_t^\alpha w\|_{\bm{\fL}_{q_0,p,\omega_1}(T)}
        +\|\cL w\|_{\bm{\fL}_{q_0,p,\omega_1}(T)}
        +\lambda\|w\|_{\bm{\fL}_{q_0,p,\omega_1}(T)}
        \leq N\|\zeta f\|_{\bm{\fL}_{q_0,p,\omega_1}(T)}.
    \end{align}
    This yields \eqref{eq7211431} with $R=1$.

    Now we consider $v:=u-w\in \bm{\fH}_{q_0,p,\omega_1,0}^{\alpha,\sigma}(T)$. Due to the similarity, we only prove \eqref{eq7211432}. Since $\cL$ is independent of $t$, $\tilde{v}:=\cL v$ satisfies
    \begin{equation} \label{eq7211434}
        \partial_t^\alpha \tilde{v} = L\tilde{v}-\lambda \tilde{v} + \cL((1-\zeta)f).
    \end{equation}
    Note that if $v$ is not regular enough, we can first mollify \eqref{eq7211434} and then take the limit of the estimate as in \eqref{eq7202131}.

    We first claim that for any $q\in(q_0,\infty)$,
    \begin{align} \label{eq7211437}
        \left(\aint_{B_{1/4}}\|\tilde{v}(\cdot,x)\|_Y^qdx\right)^{1/q}
        &\leq N \sum_{k=0}^{\infty}2^{-k\sigma}
        \left( \aint_{B_{2^{k}}} \|\tilde{v}(\cdot,x)\|_Y^{q_0} dx \right)^{1/q_0} \nonumber
        \\
        &\quad+ N\sum_{k=0}^{\infty}2^{-k\sigma}
        \left(\aint_{B_{2^k}} \|f(\cdot,x)\|_Y^{q_0}dx\right)^{1/q_0}.
    \end{align}
    For $n\in\bN$, as long as the right-hand side is positive, denote
    \begin{equation*}
        \frac{1}{q_n}:=\frac{1}{q_0}-\frac{n\sigma_0}{2d}.
    \end{equation*}
    Here, we put $\sigma_0=1$ when $\sigma=1$.
   Let $\tilde{\zeta}_n\in C_c^\infty(B_{2^{-2n+1}})$ be a cutoff function such that
    \begin{equation*}
        \tilde{\zeta}_n=1 \text{ in } B_{2^{-2n}}, \qquad
       |D_x \tilde{\zeta}_n|\leq N2^{2n},
    \end{equation*}
    which actually satisfies \eqref{cutoff} with $R=2^{1-2n}$.
   Suppose first that $q_0<2d/\sigma_0$. 
   By repeating the proof of \eqref{eq7031622}, one can easily get \eqref{eq7211437} for $q\in(q_0,q_1]$.

    Next, assume that $q_1$ is finite and
    \begin{equation*}
        \frac{1}{q_1}>\frac{\sigma_0}{2d}.
    \end{equation*}
    Then \eqref{eq7211437} with $q=q_1$ implies that $\tilde{\zeta}_2\tilde{v} \in \bm{\fL}_{q_1,p,\omega_1}(T)$. By repeating the above argument, one can obtain
        \begin{align*} 
        \left(\aint_{B_{1/16}}\|\tilde{v}(\cdot,x)\|_Y^{q_2 }dx\right)^{1/q_2}
        &\leq N \sum_{k=0}^{\infty}2^{-k\sigma}
        \left( \aint_{B_{2^{k}}} \|\tilde{v}(\cdot,x)\|_Y^{q_0} dx \right)^{1/q_0}
        \\
        &\quad+ N\sum_{k=0}^{\infty}2^{-k\sigma}
        \left(\aint_{B_{2^k}} \|f(\cdot,x)\|_Y^{q_0}dx\right)^{1/q_0}.
    \end{align*}
    By a covering argument and H\"older's inequality, \eqref{eq7211437} follows for $q\in(q_0,q_{2}]$.
    Actually, after finitely many iterations, we reach an exponent $q_n$ such that
    \begin{equation*}
        \frac{1}{q_n}\leq\frac{\sigma_0}{2d}.
    \end{equation*}
    At this stage, we define $q_{n+1}=\infty$ and repeat the same argument. Then \eqref{eq7211437} holds with $q=\infty$, which in particular proves the result for every $q<\infty$.

    Lastly, we prove \eqref{eq7211432}. Take $q\in(1,\infty)$ so that $\vartheta\leq \sigma_0-d/q$. Then by \eqref{eq7201657},
    \begin{align*}
        [\cL v]_{C^\vartheta(B_{1/16};Y)} &\leq N[\cL v]_{C^{\sigma_0-d/q}(B_{1/16};Y)} 
        \\
        &\leq N\left( \|\tilde{\zeta}_2\cL v\|_{\bm{\fL}_{q,p,\omega_1}(T)} + \|(-\Delta)^{\sigma/2}(\tilde{\zeta}_2\cL v)\|_{\bm{\fL}_{q,p,\omega_1}(T)} \right).
    \end{align*}
    Thus, again by \eqref{eq7202131} and a covering argument, one can deduce that
    \begin{align*}
        [\cL v]_{C^\vartheta(B_{1/4};Y)} &\leq N \sum_{k=0}^{\infty}2^{-k\sigma}
        \left( \aint_{B_{2^{k}}} \| \cL v(\cdot,x)\|_Y^{q_0} dx \right)^{1/q_0}
        \\
        &\quad+ N\sum_{k=0}^{\infty}2^{-k\sigma}
        \left(\aint_{B_{2^k}}\|f(\cdot,x)\|_Y^{q_0}dx\right)^{1/q_0}
        \\
        &\leq N \sum_{k=0}^{\infty}2^{-k\sigma}
        \left( \aint_{B_{2^{k}}} \| \cL u(\cdot,x)\|_Y^{q_0} dx \right)^{1/q_0}
        \\
        &\quad+ N \sum_{k=0}^{\infty}2^{-k\sigma}
        \left( \aint_{B_{2^{k}}} \| \cL w(\cdot,x)\|_Y^{q_0} dx \right)^{1/q_0}
        \\
        &\quad+ N\sum_{k=0}^{\infty}2^{-k\sigma}
        \left(\aint_{B_{2^k}}\|f(\cdot,x)\|_Y^{q_0}dx\right)^{1/q_0}.
    \end{align*}
    Then \eqref{eq7211433} easily yields the desired estimate.
    The lemma is proved.
\end{proof}

\begin{lemma} \label{lem8101843}
    Let the assumptions of Lemma \ref{lem7211430} hold. Then for any $x_0\in\bR^d$, $R>0$, $\rho\in(0,1/4)$, and $\vartheta\in(0,\sigma_0/2)$, we have
    \begin{align} \label{eq7211441}
        &\aint_{B_{R}(x_0)}\aint_{B_{R}(x_0)}
        \left|\|\cL u(\cdot,x)\|_Y-\|\cL u(\cdot,y)\|_Y\right|dxdy \nonumber
        \\
        &\leq N\rho^\vartheta\sum_{k=0}^{\infty}2^{-k\sigma}
        \left(\aint_{B_{2^k\rho^{-1}R}(x_0)}\|\cL u(\cdot,x)\|_Y^{q_0}dx\right)^{1/q_0} \nonumber
        \\
        &\quad+N\rho^{-d/q_0}\sum_{k=0}^{\infty}2^{-k\sigma}
        \left(\aint_{B_{2^k\rho^{-1}R}(x_0)}\|f(\cdot,x)\|_Y^{q_0}dx\right)^{1/q_0}
    \end{align}
    and
    \begin{align} \label{eq7211442}
        &\aint_{B_{R}(x_0)}\aint_{B_{R}(x_0)}
        \left|\|\lambda u(\cdot,x)\|_Y-\|\lambda u(\cdot,y)\|_Y\right|dxdy \nonumber
        \\
        &\leq N\rho^\vartheta\sum_{k=0}^{\infty}2^{-k\sigma}
        \left(\aint_{B_{2^k\rho^{-1}R}(x_0)}\|\lambda u(\cdot,x)\|_Y^{q_0}dx\right)^{1/q_0} \nonumber
        \\
        &\quad+N\rho^{-d/q_0}\sum_{k=0}^{\infty}2^{-k\sigma}
        \left(\aint_{B_{2^k\rho^{-1}R}(x_0)}\|f(\cdot,x)\|_Y^{q_0}dx\right)^{1/q_0},
    \end{align}
    where $N=N(d,p,q_0,K_0,\Lambda,\nu,\tilde{\alpha},\tilde{\sigma},\vartheta)$.
\end{lemma}

\begin{proof}
    By Lemma \ref{lem7211430}, there are $v,w\in \bm{\fH}_{q_0,p,\omega_1,0}^{\alpha,\sigma}(T)$ satisfying \eqref{eq7211431} and \eqref{eq7211432} with $\rho^{-1}R$ in place of $R$. Due to the similarity, we only prove \eqref{eq7211441}.

    For $w$, by the triangle inequality, H\"older's inequality, and \eqref{eq7211431},
    \begin{align*}
        &\aint_{B_{R}(x_0)}\aint_{B_{R}(x_0)}
        \left|\|\cL w(\cdot,x)\|_Y-\|\cL w(\cdot,y)\|_Y\right|dxdy 
        \\
        &\leq N\aint_{B_{R}(x_0)}\|\cL w(\cdot,x)\|_Ydx 
        \\
        &\leq N\rho^{-d/q_0}
        \left(\aint_{B_{\rho^{-1}R}(x_0)}\|\cL w(\cdot,x)\|_Y^{q_0}dx\right)^{1/q_0} 
        \\
        &\leq N\rho^{-d/q_0}
        \left(\aint_{B_{\rho^{-1}R}(x_0)}\|f(\cdot,x)\|_Y^{q_0}dx\right)^{1/q_0}.
    \end{align*}
    For $v$, by \eqref{eq7211432},
    \begin{align*}
        &\aint_{B_{R}(x_0)}\aint_{B_{R}(x_0)}
        \left|\|\cL v(\cdot,x)\|_Y-\|\cL v(\cdot,y)\|_Y\right|dxdy 
        \\
        &\leq NR^\vartheta[\cL v]_{C^\vartheta(B_{\rho^{-1}R/4}(x_0);Y)} 
        \\
        &\leq N\rho^\vartheta\sum_{k=0}^{\infty}2^{-k\sigma}
        \left(\aint_{B_{2^k\rho^{-1}R}(x_0)}\|\cL u(\cdot,x)\|_Y^{q_0}dx\right)^{1/q_0} 
        \\
        &\quad+N\rho^\vartheta\sum_{k=0}^{\infty}2^{-k\sigma}
        \left(\aint_{B_{2^k\rho^{-1}R}(x_0)}\|f(\cdot,x)\|_Y^{q_0}dx\right)^{1/q_0}.
    \end{align*}
    Then the triangle inequality yields \eqref{eq7211441}. The lemma is proved.
\end{proof}

\begin{proof}[Proof of Theorem \ref{thm_odd}]

$(i)$ Let us prove \eqref{est_odd}.
Note that it suffices to prove the estimate when $T<\infty$ and $u\in \bm{\fH}_{q,p,\omega,0}^{\alpha,\sigma}(T)$. 
Due to Lemma \ref{lem_dense} and Proposition \ref{lem_conti}, we may assume that $u\in C_c^\infty([0,T]\times \bR^d)$ and $u(0,x)=0$.

    By reverse H\"older's inequality (see e.g. \cite[Corollary 7.2.6]{G14}), there is $\gamma=\gamma(d,q,K_0)>0$ such that $q-\gamma>1$ and $\omega_2\in A_{q-\gamma}(\bR^d)$. Denote
    \begin{equation*}
        q_0:=\frac{q}{q-\gamma}\in(1,\infty),
    \end{equation*}
    which leads to
    \begin{equation} \label{eq7211443}
        \omega_2\in A_{q-\gamma}(\bR^d)=A_{q/q_0}(\bR^d).
    \end{equation}
    
    For $x_0\in\bR^d$, define the sharp function
    \begin{equation*}
        v^\sharp(x_0):=\sup_{B_r(x_1)\ni x_0}
        \aint_{B_r(x_1)}\aint_{B_r(x_1)}|v(x)-v(y)|dxdy.
    \end{equation*}
    By \eqref{eq8102022}, one can verify that $u\in \bm{\fH}_{q_0,p,\omega_1,0}^{\alpha,\sigma}(T)$ for any $q_0\in(1,\infty)$. We can therefore apply \eqref{eq7211441} with $x_1$ in place of $x_0$ to obtain
        \begin{align*}
        \left( \|\cL u(\cdot,\star)\|_Y \right)^\sharp(x_0)
        \leq N\rho^\vartheta \sM_{x,q_0}[\cL u](x_0) + N\rho^{-d/q_0} \sM_{x,q_0}[f](x_0),
    \end{align*}
    where $ Y:=L_{p,\omega_1}(T)$.
    By \eqref{eq7211443}, the weighted sharp function theorem and the weighted maximal function theorem yield
    \begin{align*}
        \|\cL u\|_{\bm{\fL}_{q,p,\omega}(T)} \leq N\rho^\vartheta\|\cL u\|_{\bm{\fL}_{q,p,\omega}(T)}
        +N\rho^{-d/q_0}\|f\|_{\bm{\fL}_{q,p,\omega}(T)}.
    \end{align*}
    Since $\|\cL u\|_{\bm{\fL}_{q,p,\omega}(T)}<\infty$ by the assumption that $u\in \bm{\fH}_{q,p,\omega,0}^{\alpha,\sigma}(T)$, by choosing $\rho\in(0,1/4)$ sufficiently small so that $N\rho^\vartheta<1/2$, we have
    \begin{equation} \label{eq7211445}
        \|\cL u\|_{\bm{\fL}_{q,p,\omega}(T)}
        \leq N\|f\|_{\bm{\fL}_{q,p,\omega}(T)}.
    \end{equation}
    Repeating the same argument with \eqref{eq7211442} gives
    \begin{equation} \label{eq7211446}
        \lambda\|u\|_{\bm{\fL}_{q,p,\omega}(T)}
        \leq N\|f\|_{\bm{\fL}_{q,p,\omega}(T)}.
    \end{equation}
    By \eqref{eqmainodd}, \eqref{eq7211445}, and \eqref{eq7211446}, we obtain \eqref{est_odd}.

$(ii)$ We consider the solvability. By Remark \ref{remCc} and Proposition \ref{lem_conti}, we only need to prove the result when $f\in C_c^\infty((0,T)\times\bR^d)$. By Proposition \ref{prop7210028}, there exists a solution $u\in \bm{\fH}_{q,p,\omega_1,0}^{\alpha,\sigma}(T)$ to \eqref{eqmainodd}.
    Let
\begin{equation*}
    \omega^{(n)}(t,x):= \omega_1(t)\cdot \min\{\omega_2(x),n\}.
\end{equation*}
Then by \eqref{eq7312315}, $[\omega^{(n)}]_{p,q}\leq N(K_0,q)$.
Moreover, $u$ is in the space $\bm{\fH}_{q,p,\omega^{(n)},0}^{\alpha,\sigma}(T)$. Thus, by \eqref{est_odd},
\begin{equation*}
    \|\partial_t^{\alpha}u\|_{\bm{\fL}_{q,p,\omega^{(n)}}(T)} + \|\cL u\|_{\bm{\fL}_{q,p,\omega^{(n)}}(T)} + \lambda\|u\|_{\bm{\fL}_{q,p,\omega^{(n)}}(T)} \leq N \|f\|_{\bm{\fL}_{q,p,\omega^{(n)}}(T)}.
\end{equation*}
Now \eqref{est_odd} follows by letting $n\to\infty$ and applying the monotone convergence theorem. Since $\|f\|_{\bm{\fL}_{q,p,\omega}(T)}<\infty$, we have $u\in \bm{\fH}_{q,p,\omega,0}^{\alpha,\sigma}(T)$ when $\lambda>0$. If $\lambda=0$, then necessarily $T<\infty$, and the odd-space version of \eqref{eq8110024} leads to $u \in \bm{\fL}_{q,p,\omega}(T)$.
The theorem is proved.
\end{proof}

\section{Proof of Theorem \ref{thm_main}} \label{sec_proof}

\begin{proof}[Proof of Theorem \ref{thm_main}]

When $p=q$, the assertions of Theorem \ref{thm_main} follow from Theorem \ref{thm_odd}. Thus, the desired result follows by Theorem \ref{thm_time}. The theorem is proved.
\end{proof}

\appendix

\section{Auxiliary results} \label{sec_appen}

We show that the minimum of two $A_p(\bR^d)$ weights still belongs to $A_p(\bR^d)$. We also refer the reader to \cite{RVV10}.

\begin{lemma}
Let $p\in(1,\infty)$ and let $\mu_1,\mu_2\in A_p(\mathbb{R}^d)$. 
Then $\mu:=\min\{\mu_1,\mu_2\}\in A_p(\mathbb{R}^d)$ and
\begin{equation} \label{eq7291856}
    [\mu]_{A_p}\leq N(p)\left([\mu_1]_{A_p} + [\mu_2]_{A_p}\right).
\end{equation}
\end{lemma}

\begin{proof}
    Due to \eqref{eqAp}, it suffices to prove that for $r>0$ and $x_0\in \bR^d$,
    \begin{equation*}
        \left(\int_{B_r(x_0)} \mu(x) dx\right) \left(\int_{B_r(x_0)} (\mu(x))^{1/(1-p)} dx\right)^{p-1} \leq N|B_r(x_0)|^p.
    \end{equation*}
    Let $A:=\{x\in B_r(x_0):\mu_1(x)\leq \mu_2(x)\}$ and $B:=\{x\in B_r(x_0):\mu_1(x)> \mu_2(x)\}$. Then
    \begin{align*}
        &\left(\int_{B_r(x_0)} \mu(x) dx\right) \left(\int_{B_r(x_0)} (\mu(x))^{1/(1-p)} dx \right)^{p-1}
        \\
        &\leq \left(\int_{A} \mu_1(x) dx + \int_{B} \mu_2(x) dx\right) \left(\int_{A} (\mu_1(x))^{1/(1-p)} dx + \int_{B} (\mu_2(x))^{1/(1-p)} dx \right)^{p-1}
        \\
        &\leq N(p) \left(\int_{B_r(x_0)} \mu_1(x) dx\right) \left(\int_{B_r(x_0)} (\mu_1(x))^{1/(1-p)} dx \right)^{p-1}
        \\
        &\quad+ N(p) \left(\int_{B_r(x_0)} \mu_2(x) dx\right) \left(\int_{B_r(x_0)} (\mu_2(x))^{1/(1-p)} dx \right)^{p-1}
        \\
        &\leq N(p)\left([\mu_1]_{A_p} + [\mu_2]_{A_p}\right)|B_r(x_0)|^p.
    \end{align*}
    The lemma is proved.
\end{proof}

\begin{remark}
    Let $\omega\in A_p(\bR^d)$. Since $[n]_{A_p(\bR^d)}=1$ for any $n$, it follows from \eqref{eq7291856} that
    \begin{equation} \label{eq7312315}
        [\min\{\omega,n\}]_{A_p(\bR^d)} \leq N(p) \left( [\omega]_{A_p(\bR^d)} + 1 \right).
    \end{equation}
\end{remark}

Next, we record some properties of $I^\alpha=I_0^\alpha$ and $\partial_t^\alpha$. In particular, we establish the robustness of the constant in \eqref{eq7152146} as $\alpha\to1$.

\begin{lemma}
    Let $T\in(0,\infty)$. 

    (i) Let $\alpha\in(0,1)$, $u\in C^1([0,T])$, and $u(0)=0$. Then
    \begin{equation} \label{eq7152153}
        I^\alpha \partial_t^\alpha u =u.
    \end{equation}

    (ii) Let $\alpha,\beta\in(0,1)$ and $u\in L_1((0,T))$. Then we have
    \begin{equation} \label{eq7152131}
        I^\alpha I^\beta u = I^{\alpha+\beta}u, \quad (a.e.) \, t\in (0,T).
    \end{equation}

    (iii) For $\alpha\in(0,1)$ and $u\in L_p((0,T))$,
    \begin{equation} \label{eq7152115}
        \|I^\alpha u\|_{L_p((0,T))} \leq \frac{T^\alpha}{\Gamma(1+\alpha)}\|u\|_{L_p((0,T))}.
    \end{equation}

    (iv) Let $\tilde{\alpha}\in(0,1)$, $\alpha\in[\tilde{\alpha},1)$. Let $p\in(1,\infty)$ and $\tilde{p}\in(1,\infty]$ satisfy
    \begin{equation*}
        \tilde{p}>p, \quad \tilde{\alpha}-1/p>-1/\tilde{p}.
    \end{equation*}
    Then for $u\in L_p((0,T))$, we have
    \begin{equation} \label{eq7152146}
        \|I^\alpha u\|_{L_{\tilde{p}}((0,T))} \leq N(\tilde{\alpha},p,\tilde{p}) T^{\alpha-1/p+1/\tilde{p}}\|u\|_{L_p((0,T))}.
    \end{equation}
\end{lemma}

\begin{proof}
    For $(i)$ and $(ii)$, we refer the reader to \cite[Lemma A.4]{DK19} and \cite[Section 2.3]{SKM87}.
    
    Now we consider $(iii)$. Note that for $t\in (0,T)$,
    \begin{equation*}
        I^\alpha u(t) = (k_{\alpha}*u1_{(0,T)})(t),
    \end{equation*}
    where $k_{\alpha}(t) := \frac{t^{\alpha-1}}{\Gamma(\alpha)}1_{(0,T)}(t)$. Thus, by Young’s convolution inequality,
    \begin{align*}
        \|I^\alpha u\|_{L_p((0,T))} \leq \|k_\alpha\|_{L_1(\bR)} \|u\|_{L_p((0,T))}.
    \end{align*}
    Since
    \begin{align*}
        \|k_{\alpha}\|_{L_1(\bR)} = \frac{1}{\Gamma(\alpha)} \int_0^T s^{\alpha-1} ds = \frac{T^\alpha}{\alpha\Gamma(\alpha)} = \frac{T^\alpha}{\Gamma(1+\alpha)},
    \end{align*}
    we obtain \eqref{eq7152115}.

    Lastly, we prove $(iv)$. By \eqref{eq7152131},
    \begin{equation*}
        I^\alpha u = I^{\tilde{\alpha}}I^{\alpha-\tilde{\alpha}}u = k_{\tilde{\alpha}}*(I^{\alpha-\tilde{\alpha}}u).
    \end{equation*}
    Let
    \begin{equation} \label{eq8131056}
        \frac{1}{q} := \frac{1}{\tilde{p}}-\frac{1}{p} +1 > 1-\tilde{\alpha}.
    \end{equation}
    Then by Young's integral inequality,
    \begin{align} \label{eq8131103}
        \|I^\alpha u\|_{L_{\tilde{p}}((0,T))} \leq \|k_{\tilde{\alpha}}\|_{L_{q}((0,T))} \|I^{\alpha-\tilde{\alpha}}u\|_{L_{p}((0,T))}.
    \end{align}
    Here, due to \eqref{eq8131056},
    \begin{align*}
        \|k_{\tilde{\alpha}}\|_{L_{q}((0,T))}^q = N(\tilde{\alpha})\int_0^T t^{(\tilde{\alpha}-1)q} dt \leq N(\tilde{\alpha},p,\tilde{p}) T^{(\tilde{\alpha}-1)q+1}.
    \end{align*}
    This together with \eqref{eq7152115} and \eqref{eq8131103} yields \eqref{eq7152146}. The proof is completed.
\end{proof}

Next, we present an $L_{q,\omega_2}(\bR^d)$-valued analogue of \cite[Lemmas A.5 and A.6]{DK19}.

\begin{lemma}
    Let $\tilde{\alpha}\in(0,1)$, $\alpha\in[\tilde{\alpha},1)$, $\sigma\in(0,2)$, $T\in(0,\infty)$, $q\in(1,\infty)$, $K_0>0$, and $[\omega_2]_{A_q(\bR^d)}\leq K_0$. Let $p\in(1,\infty)$ and $\tilde{p}\in(1,\infty]$ satisfy
    \begin{equation*}
        \tilde{p}>p, \quad \tilde{\alpha}-1/p>-1/\tilde{p}.
    \end{equation*}
    Then for $u\in \bH_{p,q,\omega_2,0}^{\alpha,\sigma}(T)$, we have
    \begin{equation} \label{eq7151540}
        \|u\|_{\bL_{\tilde{p},q,\omega_2}(T)} \leq N(\tilde{\alpha},p,\tilde{p}) T^{\alpha-1/p+1/\tilde{p}} \|\partial_t^\alpha u\|_{\bL_{p,q,\omega_2}(T)}.
    \end{equation}
\end{lemma}

\begin{proof}
    Due to Lemma \ref{lem_dense}, it suffices to prove \eqref{eq7151540} for $u\in C_c^\infty([0,T]\times\bR^d)$ such that $u(0,x)=0$.

    By \eqref{eq7152153}, it suffices to prove
    \begin{equation} \label{eq7152154}
        \|I^\alpha u\|_{\bL_{\tilde{p},q,\omega_2}(T)} \leq NT^{\alpha-1/p+1/\tilde{p}} \|u\|_{\bL_{p,q,\omega_2}(T)}.
    \end{equation}
    By the Minkowski inequality,
    \begin{align*}
        \|I^\alpha u(t,\cdot)\|_{L_{q}(\bR^d,\omega_2dx)} &\leq \frac{1}{\Gamma(\alpha)} \int_{0}^{t} (t-s)^{\alpha-1} \|u(s,\cdot)\|_{L_{q}(\bR^d,\omega_2dx)} ds 
        \\
        &= I^{\alpha} (\|u(\star,\cdot)\|_{L_{q}(\bR^d,\omega_2dx)})(t),
    \end{align*}
    which is the Riemann-Liouville fractional integral of a real valued function $t\to\|u(t,\cdot)\|_{L_{q}(\bR^d,\omega_2dx)}$.
    Hence, \eqref{eq7152154} follows from \eqref{eq7152146}.
    The lemma is proved.
\end{proof}

Next, we derive an $L_{q,\omega_2}(\bR^d)$-valued version of the robust H\"older continuity estimate by repeating the proof of \cite[Lemma A.14]{DK19}.

\begin{lemma}
Let $d\geq1$, $p\in(1,\infty)$, $\tilde{\alpha}\in(1/p,1)$, $\alpha\in[\tilde{\alpha},1)$, $\sigma\in(0,2)$, $T\in(0,\infty)$, $q\in(1,\infty)$, $K_0>0$, and $[\omega_2]_{A_q(\bR^d)}\leq K_0$. Assume that $u\in \bH_{p,q,\omega_2,0}^{\alpha,\sigma}(T)$.
Then we have
\begin{equation} \label{eq7181421}
    \|u(t_2) - u(t_1)\|_{L_{q,\omega_2}(\bR^d)} \leq N(\tilde{\alpha},p) (t_2 - t_1)^{\alpha-1/p} \|\partial_t^\alpha u\|_{\bL_{p,q,\omega_2}(T)}
\end{equation}
for $0 \leq t_1 < t_2 \leq T$.
\end{lemma}

\begin{proof}
Due to Lemma \ref{lem_dense}, we may assume that $u\in C_c^\infty([0,T]\times\bR^d)$ and $u(0,x)=0$.
By \eqref{eq7152153},
\begin{equation*}
    u(t_2) - u(t_1) = (I^\alpha \partial_t^\alpha u)(t_2) - (I^\alpha \partial_t^\alpha u)(t_1),
\end{equation*}
and thus
\begin{align} \label{eq7181417}
\Gamma(\alpha) \left(u(t_2) - u(t_1)\right) &= \int_0^{t_1} (t_2-s)^{\alpha-1} \partial_t^\alpha u(s) ds - \int_0^{t_1} (t_1 - s)^{\alpha-1}\partial_t^\alpha u(s) ds \nonumber
\\
&\quad+ \int_{t_1}^{t_2} (t_2-s)^{\alpha-1} \partial_t^\alpha u(s) ds \nonumber
\\
&:= I_1+I_2+I_3.
\end{align}
By the Minkowski inequality and H\"older's inequality,
\begin{align} \label{eq7181418}
&\|I_1 + I_2\|_{L_{q,\omega_2}(\bR^d)} \nonumber
\\
&\leq \int_0^{t_1} \left| (t_2-s)^{\alpha-1} - (t_1-s)^{\alpha-1} \right| \|\partial_t^\alpha u(s,\cdot)\|_{L_{q,\omega_2}(\bR^d)} ds \nonumber
\\
&\leq J(t_1,t_2)^{(p-1)/p} \left(\int_0^{t_1} \|\partial_t^\alpha u(s,\cdot)\|_{L_{q,\omega_2}(\bR^d)}^p ds\right)^{1/p},
\end{align}
where
\begin{align*}
J(t_1,t_2) := \int_0^{t_1} \left| (t_2-s)^{\alpha-1} - (t_1-s)^{\alpha-1} \right|^{p/(p-1)} ds.
\end{align*}
To estimate $J$, we consider two cases.
When $2 t_1 \leq t_2$, since $(\alpha-1)p/(p-1)>-1$ and $\alpha\geq\tilde{\alpha}$,
\begin{align*}
    J \leq \int_0^{t_1} (t_1-s)^{(\alpha-1)p/(p-1)} ds &\leq \frac{1}{(\alpha-1)p/(p-1)+1} t_1^{(\alpha-1)p/(p-1)+1} 
    \\
    &\leq N(\tilde{\alpha},p)(t_2 - t_1)^{(\alpha-1)p/(p-1)+1}.
\end{align*}
When $2t_1 > t_2$,
\begin{align*}
J &= \int_0^{2t_1-t_2} \cdots + \int_{2t_1-t_2}^{t_1} \cdots
\\
&\leq (1-\alpha)^{p/(p-1)} (t_2-t_1)^{p/(p-1)} \int_0^{2t_1-t_2}  (t_1-s)^{(\alpha-2)p/(p-1)} ds 
\\
&\quad+ \int_{2t_1-t_2}^{t_1} (t_1-s)^{(\alpha-1)\frac{p}{p-1}} ds
\\
&= N(\tilde{\alpha},p) (t_2-t_1)^{\frac{p}{p-1}} \left[ (t_2-t_1)^{(\alpha-2)\frac{p}{p-1} + 1} - t_1^{(\alpha-2)\frac{p}{p-1} + 1}\right] 
\\
&\quad+ N(\tilde{\alpha},p) (t_2-t_1)^{(\alpha-1)\frac{p}{p-1} + 1}
\\
&\leq N(\tilde{\alpha},p) (t_2-t_1)^{(\alpha-1)\frac{p}{p-1} + 1}.
\end{align*}
Thus, by \eqref{eq7181418}, we see that
\begin{equation} \label{eq7181419}
    \|I_1 + I_2\|_{L_{q,\omega_2}(\bR^d)} \leq N(\tilde{\alpha},p) (t_2 - t_1)^{\alpha-1/p} \|\partial_t^\alpha u\|_{\bL_{p,q,\omega_2}(T)}.
\end{equation}
For $I_3$, again by the Minkowski inequality and H\"older's inequality,
\begin{align} \label{eq7181420}
    \|I_3\|_{L_{q,\omega_2}(\bR^d)} &\leq \left(\int_{t_1}^{t_2} (t_2-s)^{(\alpha-1)\frac{p}{p-1}}  ds\right)^{\frac{p-1}{p}} \|\partial_t^\alpha u\|_{\bL_{p,q,\omega_2}(T)} \nonumber
    \\
    &\leq N(\tilde{\alpha},p) (t_2-t_1)^{\alpha-1/p} \|\partial_t^\alpha u\|_{\bL_{p,q,\omega_2}(T)}.
\end{align}
Therefore, combining \eqref{eq7181417}, \eqref{eq7181419},  and \eqref{eq7181420},
we obtain
\begin{equation*}
        \Gamma(\alpha)\|u(t_2) - u(t_1)\|_{L_{q,\omega_2}(\bR^d)} \leq N(\tilde{\alpha},p) (t_2 - t_1)^{\alpha-1/p} \|\partial_t^\alpha u\|_{\bL_{p,q,\omega_2}(T)}.
\end{equation*}
Since $\Gamma(\alpha)>\Gamma(1)=1$, we get \eqref{eq7181421}.
The proof is completed.
\end{proof}

For $\beta\in\bR$, we define the weighted Bessel potential space by
\begin{equation*}
    H_{q,\omega_2}^\beta(\bR^d):=\big\{u\in L_{q,\omega_2}(\bR^d): (1-\Delta)^{\beta/2}u\in L_{q,\omega_2}(\bR^d)\big\}.
\end{equation*}

\begin{lemma}
Let $\sigma\in(1,2)$, $p,q\in(1,\infty)$, $T\in(0,\infty)$, $K_0>0$, and $[\omega]_{p,q}\leq K_0$. Then for any $u\in C_c^\infty([0,T]\times \bR^d)$ and $\varepsilon>0$,
    \begin{equation} \label{eq7191852}
    \|\nabla u\|_{\bm{\fL}_{q,p,\omega}(T)} \leq N(d,p,q,K_0)\varepsilon^{1/(1-\sigma)}\|u\|_{\bm{\fL}_{q,p,\omega}(T)} + \varepsilon \|(-\Delta)^{\sigma/2}u\|_{\bm{\fL}_{q,p,\omega}(T)}.
\end{equation}
\end{lemma}

\begin{proof}
    First, we assume that $u(t,x)=u(x)$ is independent of $t$. Then it suffices to prove
    \begin{equation} \label{eq7191919}
        \|\nabla u\|_{L_{q,\omega_2}(\bR^d)} \leq N(d,q,K_0)\varepsilon^{1/(1-\sigma)}\|u\|_{L_{q,\omega_2}(\bR^d)} + \varepsilon \|(-\Delta)^{\sigma/2}u\|_{L_{q,\omega_2}(\bR^d)}.
    \end{equation}
    By \cite[Lemmas 3.4 and 3.5]{DJK23},
    \begin{equation*}
        \|\nabla u\|_{L_{q,\omega_2}(\bR^d)} \leq N\|u\|_{H_{q,\omega_2}^1(\bR^d)}.
    \end{equation*}
By \cite[Lemma 3.3]{D23}, $H_{q,\omega_2}^1(\bR^d)$ can be represented as the complex interpolation space;
\begin{equation*}
    H_{q,\omega_2}^1(\bR^d) = [L_{q,\omega_2}(\bR^d),H_{q,\omega_2}^{\sigma}(\bR^d)]_{1/\sigma},
\end{equation*}
which implies that
\begin{align*}
    \|u\|_{H_{q,\omega_2}^1(\bR^d)} \leq \|u\|_{L_{q,\omega_2}(\bR^d)}^{1-1/\sigma} \|u\|_{H_{q,\omega_2}^{\sigma}(\bR^d)}^{1/\sigma}.
\end{align*}
Let us recall that by \cite[Lemma 3.4]{DJK23},
\begin{equation*}
    \|u\|_{H_{q,\omega_2}^{\sigma}(\bR^d)} \approx \left( \|u\|_{L_{q,\omega_2}(\bR^d)} + \|(-\Delta)^{\sigma/2}u\|_{L_{q,\omega_2}(\bR^d)} \right),
\end{equation*}
where the equivalence constants depend only on $d,q$, and $K_0$, and are independent of $\sigma$.
Thus, by Young's inequality,
\begin{align*}
    \|u\|_{H_{q,\omega_2}^1(\bR^d)} &\leq N\varepsilon \|u\|_{H_{q,\omega_2}^{\sigma}(\bR^d)} + N\frac{\sigma-1}{\sigma} \sigma^{1/(1-\sigma)} \varepsilon^{1/(1-\sigma)} \|u\|_{L_{q,\omega_2}(\bR^d)}
    \\
    &\leq N \varepsilon \|(-\Delta)^{\sigma/2}u\|_{L_{q,\omega_2}(\bR^d)} + N \varepsilon^{1/(1-\sigma)} \|u\|_{L_{q,\omega_2}(\bR^d)},
\end{align*}
where $N$ is independent of $\sigma$. Choosing $\varepsilon>0$ appropriately again, we get \eqref{eq7191919}.

Now we prove \eqref{eq7191852}. Assume that $p=q$. For general $u\in C_c^\infty([0,T]\times \bR^d)$, $u(t,\cdot)$ satisfies \eqref{eq7191919} for every $t\in[0,T]$. 
Thus, by raising both sides of this inequality to the power of $p$, multiplying by $\omega_1(t)$, and integrating with respect to $t$, the case $p=q$ is obtained.
For general $p\neq q$, one just needs to use the extrapolation theorem (see e.g. \cite[Theorem 2.5]{DK18}).
The lemma is proved.
\end{proof}

\begin{lemma}
    Let $\tilde{\sigma}\in(0,2)$, $\sigma\in[\tilde{\sigma},2)$, $p\in(1,\infty)$, $T\in(0,\infty)$, $K_0>0$, and $[\omega_1]_{A_p(\bR)}\leq K_0$. Let $q\in(1,\infty)$ and $\tilde{q}\in(1,\infty]$ satisfy
    \begin{equation*}
        \tilde{q}>q, \quad 1/\tilde{q} > 1/q-\tilde{\sigma}/d.
    \end{equation*}
    Then for any $u\in C_c^\infty([0,T]\times \bR^d)$,
    \begin{equation} \label{eq7201656}
        \|u\|_{\bm{\fL}_{\tilde{q},p,\omega_1}(T)} \leq N \left(\|u\|_{\bm{\fL}_{q,p,\omega_1}(T)} + \|(-\Delta)^{\sigma/2}u\|_{\bm{\fL}_{q,p,\omega_1}(T)} \right),
    \end{equation}
    where $N=N(d,q,\tilde{q},\tilde{\sigma})$.
\end{lemma}

\begin{proof}
Let $\bar{\sigma}\in(0,\tilde{\sigma}]$ be such that $\bar{\sigma}=\tilde{\sigma}$ if $\tilde{q}=\infty$, and
\begin{equation*}
    1/\tilde{q}= 1/q-\bar{\sigma}/d
\end{equation*}
if $\tilde{q}\in(1,\infty)$.
    We notice that the multipliers
    \begin{equation*}
        \frac{(1+|\xi|^2)^{\sigma/2}}{1+|\xi|^{\sigma}}, \quad \frac{(1+|\xi|^2)^{\bar{\sigma}/2}}{(1+|\xi|^2)^{\sigma/2}}
    \end{equation*}
    satisfy the assumptions of Mikhlin multiplier theorem (see e.g. \cite[Theorem 6.2.7]{G14}), which yields that for $v\in C_c^\infty(\bR^d)$,
    \begin{align*}
        \|(1-\Delta)^{\bar{\sigma}/2}(1+(-\Delta)^{\sigma/2})^{-1}v\|_{L_{q}(\bR^d)} \leq N \|v\|_{L_{q}(\bR^d)}.
    \end{align*}
Moreover, by keeping track of the multiplier bounds in the application of the Mikhlin multiplier theorem, one can verify that $N$ is independent of $\sigma$ and $\bar{\sigma}$.
By the Sobolev embedding theorems (see e.g. \cite[Theorems 13.8.1 and 13.8.7]{Krylec}),
\begin{align*}
        \|(1+(-\Delta)^{\sigma/2})^{-1}v\|_{L_{\tilde{q}}(\bR^d)} &\leq N(\bar{\sigma}) \|(1-\Delta)^{\bar{\sigma}/2}(1+(-\Delta)^{\sigma/2})^{-1}v\|_{L_{q}(\bR^d)} 
        \\
        &\leq N \|v\|_{L_{q}(\bR^d)}.
    \end{align*}
    In other words,
    \begin{align} \label{eq7201654}
        \|G*v\|_{L_{\tilde{q}}(\bR^d)} \leq N \|v\|_{L_{q}(\bR^d)},
    \end{align}
    where $*$ denotes the convolution, and $G$ is the kernel of $(1+(-\Delta)^{\sigma/2})^{-1}$.

    Now we consider $u\in C_c^\infty([0,T]\times \bR^d)$. Note that
    \begin{equation} \label{eq7210122}
        u(t,x) = G*[(1+(-\Delta)^{\sigma/2})u](t,x),
    \end{equation}
    where $*$ is convolution in the spatial variable. By the Minkowski inequality and \eqref{eq7201654}, letting
    \begin{equation*}
        g(x) := \|(1+(-\Delta)^{\sigma/2})u(\cdot,x)\|_{L_{p,\omega_1}(T)},
    \end{equation*}
    we have
    \begin{align*}
        \|u\|_{\bm{\fL}_{\tilde{q},p,\omega_1}(T)} \leq \|G*g\|_{L_{\tilde{q}}(\bR^d)} \leq N\|g\|_{L_{q}(\bR^d)} = N \|(1+(-\Delta)^{\sigma/2})u\|_{\bm{\fL}_{q,p,\omega_1}(T)}.
    \end{align*}
    Then \eqref{eq7201656} follows from the triangle inequality.
    The lemma is proved.
\end{proof}

\begin{lemma}
    Let $\tilde{\sigma}\in(0,2)$, $\sigma\in[\tilde{\sigma},2)$, $p\in(1,\infty)$, $T\in(0,\infty)$, $K_0>0$, and $[\omega_1]_{A_p(\bR)}\leq K_0$. Assume that $q>1$, $q\in(d/\tilde{\sigma}, 2d/\tilde{\sigma})$ and denote
    \begin{equation*}
        \vartheta:=\tilde{\sigma}-d/q \in (0,\tilde{\sigma}/2).
    \end{equation*}
    Let $u\in C_c^\infty([0,T]\times \bR^d)$. Then for $x,y\in \bR^d$,
    \begin{equation} \label{eq7201657}
        \|u(\cdot,x)-u(\cdot,y)\|_{L_{p,\omega_1}(T)} \leq N |x-y|^{\vartheta} \left(\|u\|_{\bm{\fL}_{q,p,\omega_1}(T)} + \|(-\Delta)^{\sigma/2}u\|_{\bm{\fL}_{q,p,\omega_1}(T)} \right),
    \end{equation}
    where $N=N(d,q,\tilde{\sigma})$.
\end{lemma}

\begin{proof}
    By the Sobolev embedding theorem (see e.g. \cite[Theorem 13.8.1]{Krylec}), for $v\in C_c^\infty(\bR^d)$,
\begin{align*}
        \left[(1+(-\Delta)^{\sigma/2})^{-1}v\right]_{C^{\vartheta}(\bR^d)} &\leq N(\tilde{\sigma}) \|(1-\Delta)^{\tilde{\sigma}/2}(1+(-\Delta)^{\sigma/2})^{-1}v\|_{L_{q}(\bR^d)}.
    \end{align*}
    Thus, as in \eqref{eq7201654},
    \begin{align} \label{eq7311805}
        \left[G*v\right]_{C^{\vartheta}(\bR^d)} \leq N \|v\|_{L_{q}(\bR^d)}.
    \end{align}
Let $\phi\in L_{p',\omega_1'}(T)$ such that
\begin{equation*}
    \|\phi\|_{L_{p',\omega_1'}(T)} \leq 1,
\end{equation*}
where $1/p+1/p'=1$ and $\omega_1':=\omega_1^{-1/(p-1)}\in A_{p'}(\bR)$.
By \eqref{eq7210122},
\begin{align*}
    \int_0^T \left(u(t,x)-u(t,y)\right)\phi(t) dt = (G*h)(x) -(G*h)(y),
\end{align*}
where
\begin{equation*}
    h(x)= \int_0^T (1+(-\Delta)^{\sigma/2})u(t,x) \phi(t) dt.
\end{equation*}
    Thus, using \eqref{eq7311805} and H\"older's inequality,
    \begin{align*}
        \left|\int_0^T \left(u(t,x)-u(t,y)\right)\phi(t) dt\right| &\leq N |x-y|^{\vartheta} \|h\|_{L_q(\bR^d)} 
        \\
        &\leq N|x-y|^{\vartheta} \|(1+(-\Delta)^{\sigma/2})u\|_{\bm{\fL}_{q,p,\omega_1}(T)}.
    \end{align*}
    Since $\phi\in L_{p',\omega_1'}(T)$ is arbitrary,
    \begin{align*}
        \|u(\cdot,x)-u(\cdot,y)\|_{L_{p,\omega_1}(T)} \leq N|x-y|^{\vartheta} \|(1+(-\Delta)^{\sigma/2})u\|_{\bm{\fL}_{q,p,\omega_1}(T)}.
    \end{align*}
    Then \eqref{eq7201657} follows from the triangle inequality.
    The lemma is proved.
\end{proof}





\end{document}